\documentclass[a4paper, 11pt, reqno]{amsart}

\usepackage{amsmath,amssymb,amscd,amsthm,amsxtra, esint,bbm}

\usepackage{mathtools}
\usepackage[linktoc=page]{hyperref}

\usepackage{enumitem}

\usepackage{marginnote}

\usepackage{tikz}

\usepackage{cases}

\makeatletter
\renewcommand{\subjclassname}{%
  \textup{1991} Mathematics Subject Classification}
\@xp\let\csname subjclassname@1991\endcsname \subjclassname
\@namedef{subjclassname@2000}{%
  \textup{2000} Mathematics Subject Classification}
\@namedef{subjclassname@2010}{%
  \textup{2010} Mathematics Subject Classification}
  \@namedef{subjclassname@2020}{%
  \textup{2020} Mathematics Subject Classification}
\makeatother

\allowdisplaybreaks[2]

\renewcommand{\th}{\tau_h}

\newcommand{\scrit}{s_{\textup{crit}}}
\newcommand{\sscal}{s_{\text{scal}}}
\newcommand{\sconf}{s_{\text{conf}}}
\newcommand{\damp}{\textup{damp}}

\newtheorem{theorem}{Theorem} [section]

\newtheorem{lemma}[theorem]{Lemma}
\newtheorem{proposition}[theorem]{Proposition}
\newtheorem{remark}[theorem]{Remark}

\newtheorem{corollary}[theorem]{Corollary}

\DeclareMathOperator*{\supp}{supp}

\newcommand{\I}{\hspace{0.5mm}\text{I}\hspace{0.5mm}}
\newcommand{\II}{\text{I \hspace{-2.8mm} I} }
\newcommand{\III}{\text{I \hspace{-2.9mm} I \hspace{-2.9mm} I}}

\newcommand{\IV}{\text{I \hspace{-2.9mm} V}}

\newcommand{\noi}{\noindent}
\newcommand{\Z}{\mathbb{Z}}
\newcommand{\R}{\mathbb{R}}
\newcommand{\C}{\mathbb{C}}
\newcommand{\T}{\mathbb{T}}

\newcommand{\deff}{\overset{\textup{def}}{=}}

\let\P= \undefined
\newcommand{\P}{\mathbf{P}}

\newcommand{\E}{\mathbb{E}}
\newcommand{\En}{\mathcal{E}}

\renewcommand{\L}{\mathcal{L}}
\newcommand{\A}{\mathcal{A}}

\newcommand{\Id}{\mathrm{Id}}

\newcommand{\F}{\mathcal{F}}

\newcommand{\B}{\mathcal{B}}

\newcommand{\al}{\alpha}
\newcommand{\be}{\beta}
\newcommand{\dl}{\delta}

\newcommand{\g}{\mathfrak{g}}
\newcommand{\Dl}{\Delta}
\newcommand{\eps}{\varepsilon}
\newcommand{\kk}{\kappa}
\newcommand{\s}{\sigma}
\newcommand{\Si}{\Sigma}

\newcommand{\wt}{\widetilde}

\newcommand{\dx}{\partial_x}
\newcommand{\dt}{\partial_t}

\newcommand{\W}{\mathcal{W}}

\renewcommand{\o}{\omega}

\renewcommand{\O}{\Omega}

\newcommand{\Prob}{\mathbb{P}}

\newcommand{\x}{\mathbf{x}}
\newcommand{\y}{\mathbf{y}}

\newcommand{\les}{\lesssim}

\newcommand{\jb}[1]
{\langle #1 \rangle}

\renewcommand{\b}{\beta}

\newcommand{\ind}{\mathbf 1}

\renewcommand{\S}{\mathcal{S}}
\newcommand{\Dp}{\mathcal{D}'}

\newcommand{\J}{\mathcal{J}}
\newcommand{\ZZ}{Z}

\newcommand{\M}{\mathcal{M}}

\newcommand{\N}{\mathbb{N}}

\renewcommand{\H}{\mathcal{H}}

\newcommand{\Ld}{\Lambda}
\newcommand{\ld}{\lambda}

\newcommand{\Dlg}{\Delta_{\g}}

\newtheorem*{ackno}{Acknowledgements}

\numberwithin{equation}{section}
\numberwithin{theorem}{section}

\begin{document}
\baselineskip = 15pt

\title[Stochastic NLW dynamics on surfaces]
{Stochastic nonlinear wave dynamics\\ on compact surfaces}

\author[T.~Oh, T.~Robert, and N.~Tzvetkov]
{Tadahiro Oh, Tristan Robert, and Nikolay Tzvetkov}

\address{
Tadahiro Oh, School of Mathematics\\
The University of Edinburgh\\
and The Maxwell Institute for the Mathematical Sciences\\
James Clerk Maxwell Building\\
The King's Buildings\\
Peter Guthrie Tait Road\\
Edinburgh\\ 
EH9 3FD\\
 United Kingdom}

\email{hiro.oh@ed.ac.uk}

\address{
Tristan Robert, Institut  \'Elie Cartan\\
 Universit\'e de Lorraine, B.P. 70239, F-54506 Vand\oe uvre-l\`es-Nancy Cedex\\France}

\email{tristan.robert@univ-lorraine.fr}

\address{
Nikolay Tzvetkov, ENS Lyon,\\
UMPA-UMR 5669, 46, allée d'Italie, 69007
Lyon,\\ France
}
\email{nikolay.tzvetkov@ens-lyon.fr}

\begin{abstract}
We study 
the Cauchy problem  for the nonlinear wave equations (NLW) 
with random data and/or stochastic forcing
 on a two-dimensional compact Riemannian manifold without boundary. 
 (i) We first study the defocusing stochastic damped NLW driven by additive space-time white noise, and with initial data distributed according to the Gibbs measure.
 By introducing a suitable space-dependent renormalization, 
 we prove  local well-posedness of the renormalized equation.
Bourgain's invariant measure argument then allows
us to establish 
  almost sure global well-posedness and invariance of the Gibbs measure for the renormalized stochastic damped NLW.
  (ii)~Similarly, we study the random data defocusing NLW (without stochastic forcing or damping), and establish the same results as in the previous setting.
  (iii)~Lastly, we study  the stochastic NLW without damping.  
 By introducing a space-time dependent renormalization, 
 we  prove its local well-posedness
 with deterministic initial data in all subcritical spaces. 
  
These results extend the corresponding recent results
on the two-dimensional torus obtained by (i)~Gubinelli-Koch-Oh-Tolomeo~(2021), (ii)~Oh-Thomann~(2020), and (iii)~Gubinelli-Koch-Oh~(2018),
to a general class of compact manifolds.
The main ingredient is the Green's function estimate for the Laplace-Beltrami operator 
in this setting 
to study regularity properties of stochastic terms appearing in each of the problems.

\end{abstract}

%
%
%
\maketitle
\tableofcontents

\baselineskip = 14pt

\section{Introduction}
\subsection{Nonlinear wave equations}
We investigate the stochastic damped nonlinear wave equations (SDNLW):
\begin{align}\label{SDNLW}
\dt^2 u + (1 -  \Dlg)  u + \dt u +   u^k = \sqrt{2}\xi, 
\qquad (t,x) \in \R_+ \times \M,
\end{align}

\noi
\noi
where 
the unknown $u$ is real-valued, 
$k \geq 2$ is an integer, and $(\M,\g)$ is a two-dimensional compact Riemannian manifold without boundary.
In particular, we
study the Cauchy problem for \eqref{SDNLW}
with random initial data of low regularity distributed according to the Gibbs measure and with stochastic forcing $\xi$ given by the space-time white noise. See below for precise definitions.

We also consider
the nonlinear wave equations (NLW) without stochastic forcing:
\begin{align}\label{NLW}
\dt^2 u + (1 -  \Dlg)  u +   u^k = 0, 
\qquad (t,x) \in \R \times \M, 
\end{align}

\noi
with initial data distributing according to the Gibbs measure, as well as the stochastic nonlinear wave equations (SNLW) with deterministic data:
\begin{align}\label{SNLW}
\dt^2 u + (1 -  \Dlg)  u +   u^k = \xi, 
\qquad (t,x) \in \R \times \M.
\end{align}

In the case of the two-dimensional torus $\T^2 = (\R/\Z)^2$, 
these equations have been studied
in recent works 
by Gubinelli-Koch-Oh-Tolomeo \cite{GKOT}, Oh-Thomann \cite{OT2}, and Gubinelli-Koch-Oh \cite{GKO}.
Our main goal in this paper is to primarily investigate the Cauchy problem for \eqref{SDNLW}\footnote{Our argument also works for \eqref{NLW} and \eqref{SNLW}.} to extend the main results in \cite{GKOT, OT2, GKO}
to a more general setting of  two-dimensional compact Riemannian manifolds without boundary.

\begin{remark}\rm

The equations \eqref{SDNLW}, \eqref{NLW},  and \eqref{SNLW} 
indeed correspond
to the (stochastic) nonlinear (damped) Klein-Gordon equations.
As for local-in-time results, 
the same results with  inessential modifications
also hold for
the (stochastic) nonlinear wave equations,
where we replace $(1-\Dlg)$ in the left-hand side of \eqref{SDNLW}, \eqref{NLW}, 
and \eqref{SNLW}
by $- \Dlg u$.
In the following, we simply refer to 
\eqref{SDNLW}, \eqref{NLW}, and \eqref{SNLW} as 
the (stochastic) nonlinear wave equations.

\end{remark}

\subsection{The $\Phi^4_2$-measure and the corresponding hyperbolic dynamical problem}
The motivation to study SDNLW comes from looking at a hyperbolic counterpart of the so-called stochastic quantization equation (SQE) which is given by the following parabolic equation
\begin{align}\label{SQE}
\dt u = \Dl u + u^3 - \infty\cdot u + \xi,
\end{align}
where $\xi$ is as above and ``$\infty\cdot u$"~refers to a counter term arising in the renormalization procedure. The equation \eqref{SQE} was introduced in \cite{PW} as a dynamical problem whose limiting behavior of the solutions as $t\rightarrow +\infty$ is at least formally given by the $\Phi^4_2$-measure:
\begin{align*}
``d\rho_4 = Z^{-1}\exp\Big(-\int_{\M}|\nabla u|^2 dx -\int_{\M}u^2 dx -\frac{1}{4}\int_{\M}(u^4-\infty\cdot u^2) dx\Big)du".
\end{align*}

\noi
Hereafter, we use 
 $Z$ to denote various normalizing constants. This measure does not make sense as it is, 
 since, first of all,  the measure ``$du$" is not well defined. This is overcome by viewing
 it as 
\begin{align}\label{def-rho4}
d\rho_4 = Z^{-1}e^{-\int_{\M}(u^4-\infty\cdot u^2) dx}d\mu_0,
\end{align}
where $\mu_0$ is a Gaussian measure on the Sobolev space $H^s(\M)$ for any $s<0$ with covariance operator $(1-\Dlg)^{s-1}$ (see \eqref{def-mu} below). In particular, the nonlinearity $u^4$ is not integrable with respect to $\mu_0$, 
and hence there is a need for a renormalization in \eqref{def-rho4} and correspondingly in~\eqref{SQE}, which we discuss in the following subsection.

Now, for a stochastic hyperbolic equation with a general power nonlinearity, the corresponding measure on the phase-space \begin{align*}
\H^s(\M)=H^s(\M)\times H^{s-1}(\M)
\end{align*} is given similarly by the formal Gibbs measure
\begin{align*}
d\rho_{k+1}(u,v) = e^{-\En(u,v)}dudv,
\end{align*}
where $v=\dt u$, and $\En(u,v)$ is the (renormalized) energy given by
\begin{align*}
\En(u,v)=\frac12\int_{\M}\big\{v^2 + |\nabla u|^2 + u^2\big\}dx + \frac{1}{k+1}\int_{\M}\,:\!u^{k+1}\!:\,dx,
\end{align*}
and $\,:\!u^{k+1}\!:\,$ denotes the renormalization of the nonlinearity. In this case, the full measure is given by
\begin{align}\label{def-rhok}
d\rho_{k+1}(u,v)=Z^{-1}e^{-\int_{\M}\,:\,u^{k+1}:\,dx}d\big(\mu_0\otimes \mu_1\big) ,
\end{align}
where $\mu_1$ is the white noise measure on $\M$. Note that when there is no stochastic forcing as in  
NLW \eqref{NLW}, since it admits the Hamiltonian structure
\begin{align*}
\dt\begin{pmatrix}
u\\v
\end{pmatrix} = J\nabla_{(u,v)}\En(u,v)\text{ with }J=\begin{pmatrix}0&1\\-1&0
\end{pmatrix},
\end{align*}
then the energy $\En$ is preserved along the flow, and so at least formally $\rho_{k+1}$ is invariant for \eqref{NLW}. On the other hand, adding a stochastic forcing in the equation breaks down the Hamiltonian structure and in particular changes the equation satisfied by the speed~$v = \dt u $.
Thus,  in order for $\mu_1$ to be stationary for $v$,  one needs to add an extra damping term, 
making the equation into 
a Langevin equation
 with the momentum $v = \dt u$.
 This leads us to consider \eqref{SDNLW}.\footnote{In the physics literature, when $k$ is odd the stochastic equation \eqref{SDNLW} is then known at the ``canonical" stochastic quantization \cite{RSS} of the $\Phi^{k+1}_2$-measure \eqref{def-rhok}.}

\subsection{Renormalization of the nonlinearity}
Let us now describe the renormalization procedure. Let $\{\varphi_n\}_{n\geq 0}$ be an orthonormal basis of $L^2(\M)$ consisting of real-valued $C^\infty(\M)$-eigenfunctions of $-\Dlg$ with corresponding eigenvalues $\{\lambda_n^2\}_{n\geq 0}$ assumed to be arranged in increasing order, so that for any $u\in\Dp(\M)$, where $\Dp(\M)$ is the dual of $C^{\infty}(\M)$, one can decompose 
\begin{align*}
u = \sum_{n\geq 0}a_n\varphi_n,
\end{align*}
for some sequence $\{a_n\}_{n\geq 0}$ of real numbers. 
Then,  we can see $\mu =\mu_0\otimes \mu_1$ as the Gaussian probability measure induced under the map
\begin{align}
 X : (\o_0,\o_1) \in \O_0\times\O_1 \longmapsto (u_0^{\o_0}, u_1^{\o_1}) 
 = \bigg( \sum_{n \ge 0} \frac{g_{n}(\o_0)}{\jb{\lambda_n}}\varphi_n,
\sum_{n \ge 0} h_{n}(\o_1)\varphi_n\bigg)\in\H^s(\M), 
\label{def-mu}
 \end{align}
 \noi
where 
$\jb{\lambda_n} = \sqrt{1+\lambda_n^2}$ and 
$\{(g_{n}, h_{n})\}_{n \ge 0}$
is a sequence of independent standard real-valued Gaussian
random variables on a probability space $(\O_0\times \O_1, \F, \Prob_0\otimes\Prob_1)$. 
From  Weyl's law~\eqref{Weyl}, 
which in particular says that $\ld_n \sim n^\frac{1}{2}$, 
it is easy to see that 
the convergence of these series holds in $L^2(\O_0\times\O_1;\H^s)$ whenever $s<0$.
Moreover $\supp \mu \subset \H^s$ for any $s<0$ but $\mu(\H^0)=0$.

Now, the space-time white noise $\xi$ is a centered Gaussian random variable on a probability space $(\O,\Prob)$ with values in 
the space of Schwartz distributions
$\S'(\R;\Dp(\M))$, which is delta correlated. 
This means that for any space-time test functions $\eta,\widetilde{\eta}\in \S(\R;C^{\infty}(\M))$, we have
\begin{align*}
\E\big[\xi(\eta)\xi(\widetilde{\eta})\big] = \langle \eta,\widetilde{\eta}\rangle_{L^2_{t,x}},
\end{align*}
where $\langle \cdot,\cdot\rangle_{L^2_{t,x}}$ stands for the usual inner product on $L^2(\R\times\M)$.
In the following, we impose that 
the space-time white noise $\xi$ is independent 
 of $g_n,h_n$ in \eqref{def-mu}.

In particular, we see that $\xi$ is given by $\dt B$, where $B$ is a two-sided cylindrical Wiener process on $L^2(\M)$, defined as
\begin{align}\label{STWN}
B(t)=\sum_{n\geq 0}\beta_n(t)\varphi_n,
\end{align}
with $\be_n(0) = 0$ and 
$\be_n(t) = \jb{\xi, \ind_{[0, t]} \cdot \varphi_n}_{ t, x}$.
Here, $\jb{\cdot, \cdot}_{t, x}$ denotes 
the duality pairing on $\R\times \M$.
As a result, 
we see that $\{ \be_n \}_{n \ge 0}$ is a family of mutually independent (and independent of $g_n,h_n$ above) two-sided Brownian motions on $(\O,\Prob)$. In particular, we have $B\in C^{0,b}(\R;W^{s-1,\infty}(\M))$ almost surely for any $b\in[0,\frac12)$ and $s<0$. 
In the following, we look at the base probability space $(\O_0\times\O_1\times\O,\Prob_0\otimes\Prob_1\otimes\Prob)$ as 
\begin{align*}
\big(\H^s(\M)\times\O,\mu\otimes\Prob\big),
\end{align*}
where  
\begin{equation*}
\mu = \mu_0\otimes\mu_1 = X_{\star}(\Prob_0\otimes\Prob_1) = (\Prob_0\otimes\Prob_1)\circ X^{-1}
\end{equation*}
 is the push-forward of $X$ defined in \eqref{def-mu}.

With these notations at hand, let us first discuss the renormalization for \eqref{SDNLW}. A  solution~$u$ to \eqref{SDNLW} can be represented through Duhamel's formula:
\begin{align}
u(t)= \dt V(t)u_0 + V(t)(u_0+u_1) - \int_0^tV(t-t')u^k(t')dt'+\sqrt{2}\int_0^tV(t-t')dB(t'),
\label{X1}
\end{align}
where $(u_0,u_1)$ is as in \eqref{def-mu}, and
\begin{align}\label{def-V}
V(t) = e^{-\frac{t}2} \frac{\sin\big(t\sqrt{\frac34-\Dlg}\big)}{\sqrt{\frac34-\Dlg}}
\end{align}
is the propagator for the damped Klein-Gordon equation:
$\dt^2 u + (1 -  \Dlg)  u + \dt u = 0$, 
 i.e. the (deterministic) linear part of \eqref{SDNLW}.

We see that the roughness of a solution $u$ 
already appears at the linear level:
\begin{align}\label{def-Psidamp}
\Psi_{\damp}(t)\deff\dt V(t)u_0 + V(t)(u_0+u_1) +\sqrt{2}\int_0^tV(t-t')dB(t'), 
\end{align}
which lies in $C\big(\R;H^s(\M)\big)$ almost surely for any $s<0$ (see Proposition~\ref{prop-psidamp} below). The strategy to define the product $u^k$ in the
Duhamel formula \eqref{X1} is then to regularize the rough term $\Psi_{\damp}$ and to replace $u^k$ by another well-chosen\footnote{In particular, note that the renormalized power defined below is a monic polynomial with its lower-order coefficients becoming infinite as the regularization is removed, 
which justifies the notation $\infty\cdot u$ in \eqref{SQE} for the cubic case.} polynomial such that, as we remove the regularization, the corresponding renormalized power ${\displaystyle \,:\!u^{k}\!:\,}$ converges to some finite random variable almost surely.

More precisely, for any $N\geq 0$, let $\P_N$ be (a smooth version of) the frequency projection on the set of frequencies $\{\lambda_n\leq N\}$ (see \eqref{def-P} below). For each $(t,x) \in \R\times \M$,
 $\P_N \Psi_{\damp}(t,x)$ 
 is then a mean-zero real-valued Gaussian random variable with variance 
\begin{align}
\label{sN}
\begin{split}
\s_N(x) 
& \stackrel{\text{def}}{=} \E [ (\P_N \Psi_{\damp}(t,x))^2]
= \E[(\P_Nu_0(x))^2]\\
& \hspace{0.45mm}
 = \sum_{n\geq 0}\psi_0(N^{-2}\lambda_n^2)^2\frac{\varphi_n(x)^2}{\jb{\lambda_n}^2}  = O(\log N),
\end{split}
\end{align}
where the second equality results from the invariance of (the truncated version of)
the Gaussian  measure $\mu$ under the (truncated) linear stochastic damped wave equations given by Proposition~\ref{prop-psidamp}, and the last estimate comes from Lemma~\ref{LEM:ev} along with Weyl's law~\eqref{Weyl}.
We note that $\s_N(x)$ in \eqref{sN} is time independent.

As in the case $\M=\T^2$ investigated in \cite{OT1,OT2,GKO}, when the truncated nonlinearity $(\P_Nu)^k$ is replaced by the Wick ordered monomial defined for all\footnote{ When $\M = \T^2$, since the Gaussian process $\P_N\Psi_{\damp}(t,x)$ is also stationary in $x$, $\s_N$ is then independent of $x$. Here the renormalization must be defined pointwise in $x$.} $(t,x)\in\R\times\M$ by
 \begin{align}\label{Wick1}
 \,:\!(\P_Nu)^{k}\!:\,(t,x)=H_{k}\big(\P_Nu(t,x);\s_N(x)\big),
 \end{align}
 where $H_{k}(x,\s)$ is the $k$th Hermite polynomial, the renormalized powers of the stochastic contribution $\,:\!(\P_N\Psi_{\damp})^k\!:\,$ converge almost surely to some random variable $\,:\!\Psi_{\damp}^k\!:\,$. See Section~\ref{sec-proba} below.
 
\subsection{Well-posedness of the renormalized dynamics} 

In view of the above discussion, we look at the following smoothed renormalized version of \eqref{SDNLW}.
 \begin{align}
\begin{cases}
\dt^2 u_N + (1- \Dlg)  u_N + \dt u_N  +  H_k\big(u_N;\s_N(x)\big) = \sqrt{2}\P_N\xi, \\
(u_N,\dt u_N)\big|_{t=0}=(\P_Nu_0,\P_Nu_1),
\end{cases}~~(t,x)\in \R_+\times\M,
\label{trunc-SDNLW}
\end{align}
with the random initial data $(u_0,u_1)$ given by \eqref{def-mu}. Our main result is then the following.

\begin{theorem}\label{thm-SDNLW}
Let $k\ge 2$ be an integer and $s < 0$.
Then, there exists a stopping time $T$, $\mu\otimes\Prob$-almost surely positive, such that for any $N\in \N$,  there exists a unique solution $u_N $  to \eqref{trunc-SDNLW} which belongs $\mu\otimes\Prob$-almost surely to $C\big([0,T];H^{s}(\M)\big)$.
Moreover, $\{u_N\}_{N\in\N}$ converges $\mu\otimes\Prob$-almost surely to a stochastic process $u\in C\big([0,T];H^{s}(\M)\big)$.
\end{theorem}

\medskip
\begin{remark}\label{REM:X1}\rm
\textup{(i)} Formally, the limiting process $u$ is a solution of the full equation:
\begin{align}\label{WSDNLW}
\dt^2 u +(1-\Dlg)u  +\dt u \,+ :\!u^k\!:\, = \sqrt{2}\xi.
\end{align}
 This is only formal since the renormalized nonlinearity \eqref{Wick1} is only defined for smoothed (i.e. frequency truncated) noise and data.

\smallskip

\noi
\textup{(ii)} The limit $u$ in Theorem~\ref{thm-SDNLW} is unique in the class
\begin{align*}
\Psi_{\damp} + C\big([0,T]; H^{s_1}(\M)\big)\subset C\big([0,T];H^{s}(\M)\big)
\end{align*}
for $0<1-s_1\ll 1$.

\smallskip

\noi
\textup{(iii)} The full Wick ordered nonlinearity is actually well defined on the above class (see \eqref{Hsum} below), which justifies that $u$ ``is a solution" of the full renormalized dynamics \eqref{WSDNLW}.
\end{remark}
We now investigate the {\it global} well-posedness of \eqref{WSDNLW} and the invariance of the Gibbs measure \eqref{def-rhok} when $k\ge 3$ is an odd integer. Instead of considering the approximate dynamics given by truncating the noise and the initial data
(as in \eqref{trunc-SDNLW}), we truncate the nonlinearity and look at the following approximate dynamics:
\begin{align}\label{FNLW1}
\begin{cases}
\dt^2 u_N +(1-\Dlg)u_N  + \dt u_N + \P_NH_k\big(\P_N u_N;\s_N(x)\big) = \sqrt{2}\xi,\\
(u_N,\dt u_N)\big|_{t=0}=(u_0,u_1)\sim\rho_{N,k+1},
\end{cases}
\end{align}
where $\rho_{N,k+1}$ is the truncated Gibbs measure, defined in \eqref{rhoN} below. 
Here, the notation $(u_0,u_1)\sim\rho_{N,k+1}$ means that the random initial data $(u_0, u_1)$ has the law $\rho_{N, k+1}$.
Since $\rho_{N,k+1}\ll\mu$, the same local well-posedness and convergence result as in Theorem \ref{thm-SDNLW} also holds for \eqref{FNLW1}, and gives again a\footnote{Actually, a straightforward adaptation of our argument shows that the limits obtained by \eqref{trunc-SDNLW} or by \eqref{FNLW1} are the same. See also Remark~\ref{REM:tronc} below.} local solution $(u,\dt u)$ to \eqref{WSDNLW}. Then we can exploit the invariance of $\rho_{N,k+1}$ under the flow of \eqref{FNLW1} by following Bourgain's argument as in \cite{BO94,BO96,Tz,BT2,BTT2}, and extend the local well-posedness result into a global one.

\begin{theorem}\label{thm-invariance}
 Let  $k \geq 3$ be an odd integer\footnote{Here, we only consider the defocusing case, namely the case of an odd integer $k\in\N$ with the ``+" sign in front of the nonlinear term on the left-hand side of \eqref{SDNLW}, since in the focusing case the density of $``d\widetilde{\rho}_{k+1}(u,\dt u)=e^{+\int_\M\,:\!u^{k+1}\!:\,}d\mu(u,\dt u)"$ cannot be properly defined \cite{BS, OST}. When $k$ is even, there is no notion of focusing or defocusing.
 When $k = 2$, it is still possible to construct 
a focusing Gibbs measure, at least on the flat torus $\T^2$; see  
\cite{BO95, OST}.
This focusing Gibbs measure is, however, endowed 
with a taming by a power of the Wick-ordered $L^2$-norm, leading to a slightly different equation.
Hence, we do not consider it in this paper.
 A similar comment applies to Theorem \ref{thm-global}.}
 and $s < 0$.
 Then,  the limit $(u,\dt u)$ of the dynamics \eqref{FNLW1} can be $\mu\otimes\Prob$-almost surely extended globally in time, thus defining a global measurable flow map $\Phi(t) : \H^s(\M)\times \O \to \H^s(\M)$. Moreover,  the Gibbs measure $\rho_{k+1}$ is invariant, in the sense that for any $t\ge 0$ and any $F\in C_b(\H^s(\M);\R)$, we have
 \begin{align*}
 \int_{\H^s(\M)}\int_\O F\big[\Phi(t)(u_0,u_1,\o)\big]d\Prob(\o)d\rho_{k+1}(u_0,u_1) = \int_{\H^s(\M)}F(u_0,u_1)d\rho_{k+1}(u_0,u_1).
 \end{align*}
\end{theorem}

\begin{remark}\label{REM:tronc}\rm
As pointed out above, Theorem~\ref{thm-invariance} is concerned with the invariance of the Gibbs measure $\rho_{k+1}$ in \eqref{def-rhok} for the limit of the truncated equation \eqref{FNLW1}. The reason to consider this dynamics (rather than \eqref{trunc-SDNLW})
is that it also admits an invariant (truncated) Gibbs measure $\rho_{N,k+1}$ (see the definition in \eqref{rhoN} below), which makes it easier to apply Bourgain's invariant measure argument \cite{BO94,BO96} to globalize the dynamics in Section~\ref{sec-global}. However, this approximation is somehow less natural than \eqref{trunc-SDNLW} considered in Theorem~\ref{thm-SDNLW}, as this latter deals with solutions arising from smooth approximations of the initial data and noise instead of truncating the nonlinearity. It turns out that there are situations where the truncated dynamics \eqref{FNLW1} is actually easier to handle than the natural approximation \eqref{trunc-SDNLW}, as one can benefit of the invariance of $\rho_{N,k+1}$ also in the local theory. We refer the reader to the introduction of \cite{STz} for a more thorough discussion on this point. In our case, the local theory and stability property established in Propositions~\ref{prop-glocal} and~\ref{prop-approx} below are robust enough to handle both the truncated dynamics \eqref{FNLW1} and the natural approximation \eqref{trunc-SDNLW}, and the result of Theorem~\ref{thm-invariance} should also hold for the natural approximation \eqref{trunc-SDNLW} up to minor modifications of the argument presented in Section~\ref{sec-global}.
\end{remark}

As mentioned above, we can also look at the evolution of $\rho_{k+1}$ under (a suitably renormalized version of) the deterministic NLW \eqref{NLW} (i.e.~without stochastic forcing or damping).
For this purpose, we first study the following renormalized NLW:
 \begin{align}
\begin{cases}
\dt^2 u_N + (1- \Dlg)  u_N +  H_k\big(u_N;\s_N(x)\big) = 0 \\
(u_N,\dt u_N)\big|_{t=0}=(\P_Nu_0,\P_Nu_1),
\end{cases}
\label{trunc-NLW}
\end{align}
where $(u_0,u_1)$ has the law $\mu$ defined in \eqref{def-mu}. In this case we have similar results.
\begin{theorem}\label{THM:LWP}
Let $k \geq 2$ be an integer and $s<0$. 
Then, there exists a stopping time $T$, $\mu$-almost surely positive, such that for $\mu$-almost every initial data $(u_0,u_1)\in\H^s(\M)$ and for any $N\in \N$, there exists a unique solution $u_N \in C\big([0,T];H^{s}(\M)\big)$ to \eqref{trunc-NLW}.
Moreover,  $\{u_N\}_{N\in\N}$ converges $\mu$-almost surely to a function $u\in C\big([0,T];H^{s}(\M)\big)$.
\end{theorem}

Here,  the uniqueness of $u$ is in the corresponding class:
\begin{align*}
z+C\big([0,T];H^{s_1}(\M)\big),
\end{align*}
where $s_1$ is as in Remark \ref{REM:X1}\,(ii) and
\begin{align}\label{zN}
z(t) = S(t)(u_0,u_1)=\cos(t\sqrt{1-\Dlg})u_0 + \frac{\sin(t\sqrt{1-\Dlg})}{\sqrt{1-\Dlg}}u_1
\end{align}
is the linear solution with the random initial data $(u_0,u_1) = (u_0^{\o_0},u_1^{\o_1})$ defined in \eqref{def-mu}. Note that we have $\E\big[(\P_Nz(t,x))^2\big]=\s_N(x)$ as in \eqref{sN}, 
and hence the renormalization in \eqref{trunc-NLW} is also defined by \eqref{Wick1}. 

As before, we can alternatively look at the approximations given by solving the truncated NLW:
 \begin{align}
\begin{cases}
\dt^2 u_N + (1- \Dlg)  u_N +  \P_NH_k\big(\P_Nu_N;\s_N(x)\big) = 0 \\
(u_N,\dt u_N)\big|_{t=0}=(u_0,u_1)\sim\rho_{N,k+1}.
\end{cases}
\label{trunc-NLW2}
\end{align}
 Due to the conservation of the energy and subsequently of the truncated Gibbs measure, we also have a global statement for the limit of the solutions to \eqref{trunc-NLW2}.
  
\begin{theorem}\label{thm-global}
Let $k \geq 3$ be an odd integer. Then,  there exists a set $\Si$ of full $\rho_{k+1}$-measure such that for any initial data $(u_0,u_1)\in\Si$, the limit $(u,\dt u)$ of the solutions $(u_N,\dt u_N)$ to~\eqref{trunc-NLW2} 
exists
globally in time. Moreover,  the flow map $\Phi(t) : (u_0,u_1)\mapsto (u,\dt u)$ leaves 
the Gibbs measure $\rho_{k+1}$ invariant.
Namely,  for any $t\in\R$ and any $F\in C_b(\H^s(\M);\R)$, we have 
\begin{align*}
\int_{\H^s(\M)}F\big(\Phi(t)(u_0,u_1)\big)d\rho_{k+1}(u_0,u_1) = \int_{\H^s(\M)}F(u_0,u_1)d\rho_{k+1}(u_0,u_1).
\end{align*} 
\end{theorem}

\medskip
\begin{remark}\rm
The same comment as in Remark~\ref{REM:tronc} above also applies to the result stated in Theorem~\ref{thm-global}. In particular, for the deterministic equation \eqref{NLW}, the approximation by smooth initial data in \eqref{trunc-NLW} 
(while studying the same equation) is genuinely more natural than the one given by the truncated equation \eqref{trunc-NLW2}, since in this case the use of randomness on initial data can be interpreted as a way to give a meaning to limits of smooth solutions to \eqref{NLW} at a super-critical regularity.
See, for example, \cite{BT1, BT3, OPT}.
\end{remark}

Finally,  we consider the case with stochastic forcing but with deterministic initial and no damping\footnote{Let us recall that the damping term was added in \eqref{SDNLW} in order to preserve the measure $\rho_{k+1}$. Hence when there is no damping term as in \eqref{SNLW} there is no point in considering random initial data, since there is no invariant measure for \eqref{SNLW}.}:
\begin{align}
\begin{cases}
\dt^2 u_N +(1-  \Dlg)  u_N + H_k \big(u_N;\s_N(t,x)\big) = \P_N\xi \\
(u, \dt u) |_{t = 0} = (u_0, u_1), 
\end{cases}
\qquad (t,x) \in \R \times \M
\label{trunc-SNLW}
\end{align}
for deterministic initial data in $(u_0,u_1)\in\H^s(\M)$, where $\s_N(t,x)$ is as in \eqref{def-sigmat} below. Here, the renormalization is slightly different. 
Let us first define the stochastic convolution
\begin{align*}
\Psi(t)\deff\int_0^t\frac{\sin\big((t-t')\sqrt{1-\Dlg}\big)}{\sqrt{1-\Dlg}}dB(t') = \sum_{n\geq 0}\bigg(\int_0^t\frac{\sin\big((t-t')\jb{\lambda_n}\big)}{\jb{\lambda_n}}d\beta_n(t')\bigg)\varphi_n,
\end{align*}
which is the solution of the linear stochastic wave equation with the zero initial data. Then, from It\^o's isometry,  we have for any $x\in\M$ and $t\geq 0$:
\begin{align}\label{def-sigmat}
\begin{split}
\s_N(t,x) 
& \overset{\text{def}}{=} \E\big[\P_N\Psi(t,x)^2\big]=\sum_{n\geq 0}\psi_0(N^{-2}\lambda_n^2)(\varphi_n(x))^2\int_0^t \bigg[\frac{\sin\big((t-t')\jb{\lambda_n}\big)}{\jb{\lambda_n}}\bigg]^2dt'\\
& \hspace{0.45mm}
=\sum_{n\geq 0}\psi_0(N^{-2}\lambda_n^2)(\varphi_n(x))^2\bigg(\frac{t}{2\jb{\lambda_n}^2}-\frac{\sin(2t\jb{\lambda_n})}{4\jb{\lambda_n}^3}\bigg)
=O( t\log N).
\end{split}
\end{align}

\noi
As in \eqref{Wick1}, we thus define the renormalized Wick powers by
\begin{align}\label{def-WickPsi}
:(\P_N\Psi(t,x))^k\!:\, \overset{\text{def}}{=} H_k\big(\P_N\Psi(t,x);\s_N(t,x)\big).
\end{align}
We emphasize here that since now $\P_N\Psi$ is not stationary in $x$ or $t$, the renormalization needs to be performed pointwise in both $x$ and $t$.
\begin{theorem}\label{thm-SNLW}
Given an integer $k \geq 2$,  let $s_\textup{crit}$ be the critical regularity defined in \eqref{scrit} below. Let $0<s_1<1$ satisfying $s_1>s_\textup{crit}$ if $k=2,3$ or $s_1\geq s_\textup{crit}$ when $k\geq 4$. Then, the truncated Wick ordered SNLW \eqref{trunc-SNLW} is almost surely locally well-posed in $\H^{s_1}(\M)$, in the sense that for any data $(u_0,u_1)\in\H^{s_1}(\M)$ and any $s<0$, 
there exists an almost surely positive  stopping time $T=T_\o(u_0,u_1)$ such that for any $N\in\N$, there is a unique solution $u_N$
to \eqref{trunc-SNLW}
in the class
\begin{align*}
\P_N\Psi + X^{s_1}_T\subset C([0, T];H^{s}(\M)),
\end{align*}
where the space $X^{s_1}_T$ is defined in \eqref{def-X} below. 
Moreover, the solutions $u_N$ converge to a stochastic process $u\in C\big([0,T];H^s(\M)\big)$ almost surely.
\end{theorem}

The critical exponent $\scrit$ corresponds to the one given by the deterministic well-posedness theory:
\begin{align}
\scrit=\max(\sscal,\sconf,0)=\max\bigg(1-\frac2{k-1},\frac34-\frac1{k-1},0\bigg)
\label{scrit}
\end{align}
where $\sscal$ and $\sconf$ correspond respectively to the scaling invariance and the conformal symmetry.

Unlike in the previous models, there is no invariant Gibbs measure available for \eqref{SNLW}, and as a consequence globalizing the solutions is not as straightforward. We point out that in the special case $\M=\T^2$, this has been investigated very recently in \cite{GKOT}.

\subsection{Scheme of proofs and organization of the paper}
As transpired in the discussion above, the general strategy used in \cite{DPD} (see also \cite{McKean,BO96}) to prove Theorems~\ref{thm-SDNLW},~\ref{THM:LWP}, and~\ref{thm-SNLW} is to look for a solution under the form $u_N=r_N+w_N$ with $r_N\in\{\P_N\Psi_{\damp},z_N,\P_N\Psi\}$, where $w_N$ is expected to be smoother and hence falling into the scope of applicability of the deterministic well-posedness theory. Then, we aim to solve the perturbed equation for $w_N$ with the enhanced data set $\{w_N(0),\dt w_N(0),r_N,...,\,:\!r_N^k\!:\,\}$. Indeed, in view of the formula~\eqref{Hsum} below for the renormalization of the sum, we see that $w_N$ solves (in the case of \eqref{SDNLW})
\begin{align}\label{eq-w-r}
\dt^2w_N +(1-\Dlg)w_N +\dt w_N + \sum_{\ell=0}^k{\binom k \ell}\,:\!r_N^{\ell}\!:\,w_N^{k-\ell} = 0.
\end{align} 
Hence it is enough to estimate the Wick ordered monomials $:r_N^{\ell}:$ uniformly in $N$ in order to estimate $:\!u_N^k\!:$. Then,  we can solve the equation for $w_N$ uniformly in $N$ by a 
standard fixed point argument as in the deterministic setting. The difficulty with working on a general compact Riemannian manifold without boundary appears in the first step when trying to get good probabilistic estimates on the random objects appearing after renormalization. Indeed, the Fourier analytic proofs of these estimates in the previous works on $\T^2$ \cite{OT1,OT2,GKO} fail here because of the lack of structure of a commutative group and of uniform boundedness of the eigenfunctions. Thus we cannot rely only on \lq\lq global \rq\rq (on $\M$) arguments. Instead, we give a local description of the stochastic objects in the spirit of \cite{BGT}, so that up to localizing and controlling various error terms which appear in this process, the probabilistic estimates in the case of a manifold follow from analyzing the kernel of some pseudo-differential operators ($\Psi$DOs) in $\R^2$. Note that the semi-classical analysis
that we employ  is somehow non-standard, since not all the pseudo-differential operators involved depend on the semi-classical parameter, so we have to work with \lq\lq semi\rq\rq\,semi-classical $\Psi$DOs.

Alternatively, in the context of parabolic singular stochastic PDEs, the authors in \cite{BB} developed a functional calculus adapted to the heat semi-group on manifolds, which enabled them to build a robust and general theory for the study of singular stochastic PDEs in a more complex geometrical setting. Though we believe that their approach could be adapted to treat our problem, it seems that the general bound on the powers of the truncated Green function for the Laplace-Beltrami operator established in Proposition~\ref{prop-Green}, which is in the core of our proof, is new and of independent interest. In particular, it would prove itself useful if one wishes to extend the result of \cite{DPD} for \eqref{SQE} on compact surfaces. See also Remark~\ref{rk-SQE} below.

Another contribution of this work is to extend Bourgain's invariant measure argument \cite{BO94,BO96} to the case of a singular stochastic PDE, allowing us to globalize the local result of Theorem~\ref{thm-SDNLW}. This argument has indeed previously been used mainly in the context of a deterministic Hamiltonian PDE with random initial data such as \eqref{NLW} considered here.
In Section~\ref{sec-global},  we carefully detail its implementation in the presence of a singular random forcing term.

We begin by recalling the tools that we need from spectral theory and semi-classical calculus in Section~\ref{SEC:PDO}, in particular the local description of semi-classical pseudodifferential operators given in \cite{BGT} that we shall use extensively. In Section~\ref{sec-proba}, after recalling the basic tools from probability theory and Euclidean quantum field theory, 
 we establish the crucial probabilistic estimates on the aforementioned stochastic objects. 
 Sections~\ref{sec-local} and~\ref{sec-global} are dedicated to the proof of the local and global well-posedness results and the invariance property of the Gibbs measure $\rho_{k+1}$.

\section{Functional calculus and semi-classical pseudo-differential calculus}\label{SEC:PDO}
In this section, we collect the tools from micro-local analysis that we will need in the  next sections. Most of the background needed here can be found in \cite{Zworski}, except for the few results on the functional calculus which can be found in \cite{Davies}.

\subsection{Geometric setting}
We begin by recalling the general setting for our results. We consider a $d$-dimensional\footnote{ In this section we state some results for a general dimension $d\in\N$, but in the rest of the paper we only consider $d=2$.} compact Riemannian manifold without boundary $(\M,\g)$, on which we fix a \emph{finite} atlas $(U_j,V_j,\kk_j)_{j\in\J}$ for some finite index set $\J$, i.e. the $V_j$ are open sets covering $\M$:
\begin{align*}
\M=\bigcup_{j\in\J}V_j,
\end{align*}
and $U_j$ are open sets in $\R^d$, with\footnote{ In the differential geometry literature, atlases are generally defined with the opposite convention that $U\subset \M$ and $\kk: U\rightarrow \kk(U)\subset \R^d$. Here we chose to keep the convention of \cite{BGT}.} some homeomorphisms $\kk_j: U_j\subset \R^d\rightarrow V_j\subset\M$ such that $\kk_j^{-1}\circ \kk_k$ are smooth diffeomorphisms on $U_j\cap U_k$, for any $j,k\in\J$ such that $U_j\cap U_k\neq\emptyset$. We also fix an associated smooth partition of unity $(\chi_j)_{j\in\J}$, i.e. $\chi_j\in C^{\infty}(M)$ with $\supp\chi_j\subset V_j$ and for any $x\in\M$,
\begin{align*}
\sum_{j\in\J}\chi_j(x) = 1.
\end{align*}
 For $j\in \J$ and a smooth function $u\in C^{\infty}(V_j)$, the pull-back of $u$ is then the function $\kk_j^{\star}u = u\circ \kk_j \in C^{\infty}(U_j)$. 
 
Given a local chart $(U_j,V_j,\kk_j)$, the metric $\g$ is given by a smooth mapping $\g: x\in U_j\mapsto \big(\g_{m,\ell}(x)\big)_{m,\ell=1...d}$ where $\big(\g_{m,\ell}(x)\big)_{m,\ell=1...d}$ is a symmetric positive definite matrix, with inverse denoted by $\big(\g^{m,\ell}(x)\big)_{m,\ell=1...d}$.
 
The Laplace-Beltrami operator can then be described as the negative\footnote{ Again, it is common to define the Laplace-Beltrami operator as the positive operator $-\Dlg$, but we stick to the negative one so that the wave equations \eqref{SDNLW}-\eqref{NLW}-\eqref{SNLW} have the same formulation as on $\T^2$.} operator acting locally on smooth functions $u\in C^{\infty}(V_j)$ by
\begin{align*}
\kk_j^{\star}\chi_j(\Dlg u)(x) &= \sum_{m,\ell=1}^d\frac{1}{\sqrt{\det\g(x)}}\partial_{x_m}\big(\sqrt{\det\g(x)}\g^{m,\ell}(x)\partial_{x_\ell}\big)\kk_j^{\star}(\wt\chi_j u)\\
&=\big(p_{2}(x,D)+p_{1}(x,D)\big)\kk_j^{\star}(\wt\chi_j u),
\end{align*}
for any $x\in U_j$, where $\wt\chi_j \in C^{\infty}_0(V_j)$ satisfies $\wt\chi_j\equiv 1$ on $\supp\chi_j$. Here $p_1$ is a differential operator of order 1, and the differential operator $p_{2}$ is given by
\begin{align}\label{def-psDelta}
p_{2}(x,D)=\sum_{m,\ell=1}^d\g^{m,\ell}(x)\partial_{x_m}\partial_{x_\ell}.
\end{align}
In particular, since $\g$ is smooth with values in symmetric positive definite matrices and $\M$ is compact, there exists $c,C>0$ such that for any $x\in \bigcup_{j\in\J}\supp\kk_j^\star \chi_j$ and $\xi\in\R^d$ we have
\begin{align}\label{bound-p}
-C|\xi|^2\leq p_{2}(x,\xi)\leq -c|\xi|^2.
\end{align}
We recall that $-\Dlg$ admits an orthonormal basis $\{\varphi_n\}_{n\geq 0}\subset C^{\infty}(\M)$ of $L^2(\M)$ consisting of eigenfunctions with corresponding eigenvalues $\{\lambda_n^2\}_{n\geq 0}$ assumed to be arranged in the non-decreasing order, and that we have Weyl's law
\begin{align}\label{Weyl}
\#\{n\ge 0,~\lambda_n \leq \lambda\} \sim \lambda^d,
\end{align} 
for any $\lambda\geq 0$. In particular we have $\lambda_n \sim n^\frac1d$.

The eigenfunctions $\varphi_n$'s are not uniformly bounded (in $n$), but we have (see e.g. \cite[Proposition 8.3]{BTT1}) that they are bounded in a mean value meaning:
\begin{lemma}\label{LEM:ev}
Let $d=2$. There exists $C>0$ such that for any $\Ld\in\R$ and $x\in\M$, we have
\begin{align*}
\sum_{n\geq 0}\mathbf{1}_{(\Ld,\Ld+1]}(\lambda_n)\frac{(\varphi_n(x))^2}{1+\lambda_n^2}\leq C\sum_{n\geq 0}\mathbf{1}_{(\Ld,\Ld+1]}(\lambda_n)\frac{1}{1+\lambda_n^2},
\end{align*}
where $\mathbf{1}_{(\Ld,\Ld+1]}$ is the indicator function of the interval $(\Ld,\Ld+1]$.
\end{lemma}
Indeed, this lemma follows directly from the following asymptotic behavior for the spectral function of $\Dlg$ due to H\"ormander \cite{Hormander}: for any $d\in\N$, there exists $c_d>0$ such that for any $\Ld\geq 0$ and $x\in\M$,
\begin{align*}
e(x,\Ld^2) \overset{\text{def}}{=} \sum_{\lambda_n^2\leq \Ld^2}(\varphi_n(x))^2 = c_d\Ld^d + O(\Ld^{d-1}).
\end{align*}
\subsection{Functional calculus}
We finally move onto the definition and the local description in terms of $\Psi$DOs of some operators used to describe the stochastic objects and to construct the Sobolev and Besov spaces needed to measure them.

 To this end, let us first define $\P_N$ to be a smooth version of the Dirichlet projection onto the frequencies $\{\lambda_n\leq N\}$. Namely, take a smooth even non-increasing cut-off $\psi_0\in C_0^{\infty}(\R)$ satisfying $\supp \psi_0 \subset [-1,1]$ and $\psi_0\equiv 1$ on $[-1/2,1/2]$. For any real-valued $u\in L^2(\M)$, we have 
\begin{align*} 
u = \sum_{n\geq 0}\langle u,\varphi_n\rangle_{L^2(\M)}\varphi_n,
\end{align*}
where
\begin{align*}
\langle u,v\rangle_{L^2(\M)} = \int_{\M}u(x)v(x)dx
\end{align*} is the inner product in $L^2(\M)$ and we simply wrote $dx$ for the volume density on $(\M,\g)$. For any $N>0$, $\P_N$ is then defined as the linear operator on $L^2(\M)$ given by
\begin{align}\label{def-P}
\P_Nu = \sum_{n\geq 0} \psi_0\Big(\frac{\lambda_n^2}{N^2}\Big)\langle u,\varphi_n\rangle_{L^2(\M)}\varphi_n.
\end{align}
In particular, if we define the finite-dimensional subspace of $L^2(\M)$
\begin{equation*}
E_N = \text{Span}\{\varphi_n,~\lambda_n \leq N\}
\end{equation*}
with the orthogonal projection 
\begin{equation*}
\Pi_N: L^2(\M)\rightarrow E_N,
\end{equation*}
then $\P_N$ maps $L^2(\M)$ into $E_N$ and
\begin{align}\label{prop-PiN}
\Pi_N\P_N = \P_N\Pi_N = \P_N.
\end{align}
Next, we define the sets of dyadic integers for $N$ as
\begin{align*}
2^{\Z_+} = \{1,2,4,...\} \text{ and }2^{\N}=2^{\Z_+}\setminus\{1\}.
\end{align*}
Hereafter, we will use the Sobolev and Besov spaces $W^{s,p}(\M)$ and $B^{s}_{p,q}(\M)$, $s\in\R$, $1\leq p,q\leq \infty$, which are defined via the norms
\begin{align*}
\|u\|_{W^{s,p}} \deff \Big\|\sum_{n\geq 0} \jb{\lambda_n}^s\langle u,\varphi_n\rangle_{L^2(\M)}\varphi_n\Big\|_{L^p(\M)},
\end{align*}
and
\begin{align*}
\|u\|_{B^s_{p,q}} \deff \Big(\big\|\P_1 u\big\|_{L^{p} }^q+\sum_{N\in 2^{\N}}N^{sq}\big\|(\P_N-\P_{N/2}) u\big\|_{L^{p} }^q\Big)^{\frac1q}.
\end{align*}

For now the Besov norms of a function $u$ are only defined in terms of projections in the eigenfunction expansion of $u$. Although it is easy to handle these norms when $p=2$ (since the $\varphi_n$'s form an orthonormal basis of $L^2(\M)$), we need an equivalent characterization to be able to estimate them when $p\neq 2$.

Let us recall the definition of the $L^2$ functional calculus. For any bounded continuous function $f$ on $\R$, we can define the bounded linear operator $f(-\Dlg)$ on $L^2(\M)$ as
\begin{align}\label{def-func-calculus}
f(-\Dlg)u = \sum_{n\geq 0}f(\lambda_n^2)\langle u,\varphi_n\rangle_{L^2(\M)}\varphi_n.
\end{align}
This defines a continuous linear map from $C_b(\R)$ to the space $\L(L^2(\M))$ of bounded linear operators on $L^2(\M)$. More generally, if $f\in\S^m$ for some $m>0$ (see \eqref{def-Sm} below), then $f(-\Dlg)$ is an unbounded operator on $L^2(\M)$ with domain given by 
\begin{align*}
D\big(f(-\Dlg)\big) = \Big\{u\in \Dp(\M),~\sum_{n\geq 0}\big|f(\lambda_n^2)\langle u,\varphi_n\rangle\big|^2<\infty\Big\}.
\end{align*}
For $N\in 2^{\N}$, we define 
\begin{align*}
\psi_{N^2}(x)= \psi_0(N^{-2}x)-\psi_0(4N^{-2}x),
\end{align*}
and
\begin{align*}
\psi_1(x)= \psi_0(x)
\end{align*}
for $N=1$. In view of the previous definition, we have $\P_N = \psi_0(-N^{-2}\Dlg)$ and for $N\in 2^{\N}$, we have 
\begin{align*}
\P_N-\P_{N/2} = \psi_{N^2}(-\Dlg).
\end{align*}

Thus we need to give a local description of the bounded linear operators which are functions of $-\Dlg$ on $L^2(\M)$ given by the functional calculus. This is the content of the next subsection.
\subsection{Pseudo-differential calculus}
We begin by collecting a few facts about (semi-classical) $\Psi$DOs. First, for $d\in\N$ and any $m\in\R$ we say that a function $f\in C^{\infty}(\R^d)$ belongs to the space $\S^m$ if for any multiindex $\beta\in\N^d$ and any $\xi\in\R^d$,
\begin{align}\label{def-Sm}
|\partial_{\xi}^{\beta}f(\xi)|\les \jb{\xi}^{m-|\beta|},
\end{align}
where $\jb{\xi}=\sqrt{1+|\xi|^2}$ and $|\beta|$ is the length of the multiindex $\beta$. Here we use the notation $A\les B$ if there exists $c>0$ (independent of the sets where $A$ and $B$ vary) such that $A\leq cB$. We also use the notations $A\sim B$ if $A\les B$ and $B\les A$, and $A\ll B$ if we can take $c=10^{-12}$. We extend this definition to functions $a: \R^d\times\R^d \rightarrow \R$, which belong to the symbol class $\S^m$ if $a \in C^{\infty}(\R^d\times\R^d)$ and satisfy for any $\alpha,\beta\in\N^d$ and $(x,\xi)\in\R^d\times\R^d$,
\begin{align}\label{def-symbolclass}
\big|\partial_x^{\alpha}\partial_{\xi}^{\beta}a(x,\xi)\big|\les \jb{\xi}^{m-|\beta|}.
\end{align}
 Then for $m\in\R$ and a symbol $a\in\S^m$ we define the semi-classical $\Psi$DO of order $m$ with symbol $a$ with respect to some semi-classical parameter\footnote{ In the following, we will take for the semi-classical parameter $h=N^{-1}$ for some $N\in\N$. } $h\in (0,1]$ to be the linear operator acting on Schwartz functions $u\in \S(\R^d)$ by the quantization rule
\begin{align}\label{def-pdo}
a(x,hD)u=\frac{1}{(2\pi)^d}\int_{\R^d}e^{ix\cdot\xi}a(x,h\xi)\widehat{u}(\xi)d\xi,
\end{align}
and $\widehat{u}$ stands for the Fourier transform of $u$. Hereafter we systematically neglect the constants $2\pi$ appearing either in \eqref{def-pdo} or in the Fourier transform. 

A particular case of Fefferman's result \cite{F} is that a (semi-classical) $\Psi$DO of order 0 extends to a bounded linear operator on $L^p(\R^d)$ (with norm independent of $h$ in the semi-classical case), for any $1<p<\infty$. It is also well-known (see for example \cite{Zworski}) that the composition of $\Psi$DOs of order $m_1$ and $m_2$ gives a $\Psi$DO of order $m_1+m_2$, and moreover the symbolic calculus gives
\begin{align*}
a(x,hD)\circ b(x,hD) = (a\#b)(x,hD),
\end{align*}
where for arbitrary $M\in\N$,
\begin{align}\label{def-pscomp}
(a\#b)(x,h\xi)=\sum_{|\alpha|=0}^{M-1}c_\alpha h^{|\alpha|}\big(\partial_{\xi}^{\alpha}a\cdot\partial_{x}^{\alpha}b\big)(x,h\xi) + O_{\S^{m_1+m_2-M}}(h^M).
\end{align}
Here we use the notation $O_{\S^{m_1+m_2-M}}(h^M)$ to mean
\begin{align*}
O_{\S^{m_1+m_2-M}}(h^M)= h^Mr_{M,a,b}(x,hD)
\end{align*}
for some $r_{M,a,b}\in \S^{m_1+m_2-M}$ (and depending continuously upon $a$ and $b$ for the composition). This implies that if $a\in\S^m$, then for any $s\in\R$, $a(x,hD)$ maps continuously $H^{s}(\R^d)$ into $H^{s-m}(\R^d)$, and for any $u\in\S(\R^d)$ we have the estimate\footnote{ The operator norm of $a(x,hD): H^s(\R^d)\rightarrow H^{s-m}(\R^d)$ depends on $h$ here because we always work with \emph{classical} Sobolev spaces, as opposition to the \emph{semi-classical} Sobolev spaces generally used in the semi-classical analysis. This is due to the \lq\lq hybrid\rq\rq~nature of our problem where we have to measure the composition of classical $\Psi$DOs with semi-classical ones.}
\begin{align}\label{estim-SCPDO}
\|a(x,hD)u\|_{H^{s-m}(\R^d)}\les h^{(m-s)\wedge 0+s\wedge 0}\|u\|_{H^s(\R^d)}.
\end{align}
Here $s\wedge 0 = \min (s,0)$. This follows directly from the uniform (in $h$) $L^2$ boundedness of the semi-classical $\Psi$DO $\jb{hD}^{s-m}a(x,hD)\jb{hD}^{-s}$ which is of order 0, and the estimates $\jb{\xi}^s\les h^{(-s)\wedge 0}\jb{h\xi}^s$ and $\jb{h\xi}^s\les h^{s\wedge 0}\jb{\xi}^s$ for any $s\in\R$ and $\xi\in\R^d$.

Let us now give a local description in terms of $\Psi$DOs of the bounded linear operators on $L^2(\M)$ given by the previous functional calculus. If $\psi$ is any smooth and compactly supported function, we can also view $\psi(-N^{-2}\Dlg)$ as a semi-classical $\Psi$DO (with semi-classical parameter $h=N^{-1}$) in local coordinates. Indeed, let us recall the result of Proposition 2.1 in \cite{BGT}.
\begin{proposition}\label{prop-dvp}
Let $\psi\in C^{\infty}_0(\R)$, $\kk: U\subset \R^d\rightarrow V\subset \M$ be a coordinate patch, and $\chi,\widetilde{\chi}\in C^{\infty}_0(V)$ with $\widetilde{\chi}\equiv 1$ on $\supp\chi$. Then there exists a sequence of symbols $(a_m)_{m\geq 0}$ in $C^{\infty}_0(U\times\R^d)$ with the following properties:\\

\noi\smallskip
\textup{(i)}  for any $M\in\N$, any $h\in (0,1]$ and any $s\in\R$, $0\le \s \le M$, we have the expansion
\begin{align}\label{dvp-psi}
\Big\|\kk^{\star}\big(\chi\psi(-h^2\Dlg)v\big) - \sum_{m=0}^{M-1}h^ma_m(x,hD)\kk^\star(\wt\chi v)\Big\|_{H^ {s+\s}(\R^d)} \les h^{M-\max(\s+s,\s,|s|)}\|v\|_{H^s(\M)}
\end{align} 
for any $v\in C^{\infty}(\M)$;\\
\noi
\textup{(ii)} for any $x\in U$ the principal symbol is given by
\begin{equation*}
a_0(x,\xi)= \chi(\kk(x))\psi\big(-p_2(x,\xi)\big),
\end{equation*}
where $p_2$ has been defined in \eqref{def-psDelta},\\

\noi\smallskip
\textup{(iii)} for all $m\geq 0$, $a_m$ is supported in 
\begin{align}\label{support}
\Big\{(x,\xi)\in U\times\R^d,~\kk(x)\in\supp\chi,~-p_2(x,\xi)\in \supp \psi\Big\}.
\end{align}
\end{proposition}
In particular, this means that for $\psi\in C^{\infty}_0(\R)$, the semi-classical operator $\psi(-h^2\Dlg)\in\L\big(L^2(\M)\big)$ defined by the functional calculus can be described locally by some $\Psi$DOs with symbol in\footnote{See also \cite[Section 14.3.2]{Zworski}.} 
\begin{align*}
\S^{-\infty}(\R^d\times\R^d)=\bigcap_{m\in\R}\S^m(\R^d\times\R^d).
\end{align*}
Note that the smoothing property of the remainder in \eqref{dvp-psi} is only stated for $s=0$ in \cite[Proposition 2.1]{BGT}, but one can derive the bound in \eqref{dvp-psi} by the same computation as in \cite{BGT} and using \eqref{estim-SCPDO}.
\begin{remark}\label{rk-pdolocal}
\rm
This result relies on describing $\psi(-\Dlg)$ through Helffer-Sj\"ostrand's formula
\begin{align*}\psi(-\Dlg)=-\frac{1}{\pi}\int_{\C}\bar{\partial}\widetilde{\psi}(z) (z+\Dlg)^{-1}dz,\end{align*}
where $\widetilde{\psi}$ is an almost analytic extension of $\psi$, and using that the resolvent $(z+\Dlg)^{-1}$ is locally a $\Psi$DO of order $-2$. In particular, one can see that the above integral is absolutely convergent for any function $\psi$ in the class
\begin{align*}
\A = \bigcup_{m<0}\S^m(\R)
\end{align*}
(which contains $C_0^{\infty}(\R)$), so that the integral representation of $\psi(-\Dlg)$ also holds for $\psi\in\A$ (see \cite[Chapter 2]{Davies}). Using the same argument, for any $\psi\in\S^m(\R)$, $m<0$, then $\psi(-\Dlg)$ is locally given by a $\Psi$DO of order $-2m$ with principal symbol
\begin{align*}
\psi\big(-p_{2}(x,\xi)\big)\in\S^{-2m}(\R^d\times\R^d).
\end{align*}
\end{remark}
Using the previous proposition, we get the following Bernstein type estimate for the $L^p(\M)\rightarrow L^q(\M)$ mapping property of the operator $\psi(-h^2\Dlg)$. See Corollary 2.4 in \cite{BGT}.
\begin{corollary}\label{lem-boundedness-psi-delta}
Under the conditions of the previous proposition, for any $1\leq p\leq q\leq \infty$, there exists $C>0$ such that for any $u\in C^{\infty}(\M)$ and $h\in (0,1]$,
\begin{equation*}
\|\psi(-h^2\Dlg)u\|_{L^q(\M)}\le C h^{d\big(\frac{1}{q}-\frac{1}{p}\big)}\|u\|_{L^p(\M)}.
\end{equation*}
\end{corollary}
\subsection{More on the function spaces}
In order to close the fixed point argument in the proofs of the well-posedness results, we will need a fractional Leibniz rule in $B^s_{p,q}(\M)$. First, we need an equivalent characterization of the topology on the Besov spaces $B^s_{p,q}(\M)$.
\begin{proposition}\label{prop-norm}
Let $\kk: U\subset \R^d\rightarrow V\subset\M$ be a coordinate patch and $\chi\in C^{\infty}_0(V)$. For any $s\in\R$ and $1\leq p,q\leq \infty$, there exist $c,C>0$ such that for any $u\in C^{\infty}(\M)$,
\begin{align}\label{equiv-norm}
c\|\chi u\|_{B^s_{p,q}(\M)}\leq  \|\kk^{\star}(\chi u)\|_{B^{s}_{p,q}(\R^d)}\leq C\|u\|_{B^s_{p,q}(\M)}.
\end{align}
\end{proposition}
\begin{proof}
First, observe that it is enough to establish the right-hand side inequality, since by duality it holds
\begin{align*}
\|\chi u\|_{B^s_{p,q}(\M)} &= \sup_{\|\wt u\|_{B^{-s}_{p',q'}(\M)}\le 1} \int_\M \chi u\cdot\wt u  = \sup_{\|\wt u\|_{B^{-s}_{p',q'}(\M)}\le 1} \int_{\R^d} \kk^{\star}(\chi u)\cdot\kk^{\star}(\wt\chi\wt u)\\
& \les \sup_{\|\wt u\|_{B^{-s}_{p',q'}(\M)}\le 1}\|\kk^{\star}(\chi u)\|_{B^s_{p,q}(\R^d)} \|\kk^{\star}(\wt\chi\wt u)\|_{B^{-s}_{p',q'}(\R^d)} \les \|\kk^{\star}(\chi u)\|_{B^s_{p,q}(\R^d)},
\end{align*}
where in the last step we used the right-hand side inequality in \eqref{equiv-norm}. This shows that the left-hand side inequality follows from the right-hand side one.

We thus need to estimate 
\begin{align*}
\sum_{N\in 2^{\Z_+}}N^{sq}\|\theta_N(D)\kk^{\star}(\chi u)\|_{L^p(\R^d)}^q,
\end{align*}
 where $\{\theta_N\}_{N\in 2^{\Z_+}}$ is an inhomogeneous dyadic partition of unity in $\R^d$. We first take a fattened version $\widetilde{\psi}_{N_1^2}$ of $\psi_{N_1^2}$, where $\psi_{N_1^2}$ is the multiplier in the definition of $\P_{N_1}$, and decompose
\begin{align*}
\theta_N(D)\kk^{\star}(\chi u) &= \sum_{N_1\in 2^{\Z_+}}\theta_N(D)\kk^{\star}(\chi u_{N_1})\\
&=\sum_{N_1\sim N}\theta_N(D)\kk^{\star}(\chi u_{N_1})+\sum_{N_1\not\sim N}\theta_N(D)\kk^{\star}\big(\chi\widetilde{\psi}_{N_1^2}(-\Dlg)u_{N_1}\big),
\end{align*}
where $u_{N_1}=\psi_{N_1^2}(-\Dlg)u$. To bound the terms in the second sum above, we have the following lemma.
\begin{lemma}\label{LEM:Picasso}
Let $\kk$ and $\chi$ as in Proposition~\ref{prop-norm}. Then for any $u\in L^p(\M)$, $p\geq 1$, and any $N,N_1\in  2^{\Z_+}$ with $(N\vee N_1)\gg (N\wedge N_1)$, we have for arbitrary $B>0$:
\begin{align}\label{estim-key-norm}
\big\|\theta_N(D)\kk^{\star}\big(\chi\widetilde{\psi}_{N_1^2}(-\Dlg)u\big)\big\|_{L^p(\R^d)}\les (N\vee N_1)^{-B}\|u\|_{L^p(\M)}.
\end{align}
\end{lemma} 
With this lemma at hand, we can finish establishing the right-hand side inequality in \eqref{equiv-norm}. Indeed, for the almost diagonal terms, we have from Minkowski's inequality, the uniform boundedness of the Littlewood-Paley projectors $\theta_N(D)$ on $L^p(\R^d)$, and H\"older's inequality with Fubini's theorem that
\begin{align*}
\big\|\mathbf{1}_{N_1\sim N}N^s\theta_N(D)\kk^\star(\chi u_{N_1})\big\|_{\ell^q_N(2^{\Z_+}) L^p(\R^d)\ell^1_{N_1}(2^{\Z_+})}&\les \big\|\mathbf{1}_{N_1\sim N}N^s u_{N_1}\big\|_{\ell^q_{N_1}(2^{\Z_+})\ell^q_{N}(2^{\Z_+}) L^p(\M)}\\
&\les \|u\|_{B^s_{p,q}(\M)},
\end{align*}
while for the off-diagonal terms we have from Minkowski's inequality and Lemma~\ref{LEM:Picasso} applied to $u_{N_1}$ with $B>2|s|$:
\begin{align*}
&\big\|\mathbf{1}_{N_1\not\sim N}\theta_N(D)\kk^{\star}\big(\chi\widetilde{\psi}_{N_1^2}(-\Dlg)u_{N_1}\big)\big\|_{\ell^q_N(2^{\Z_+})L^p(\R^d)\ell^1_{N_1}(2^{\Z_+})}\\
&\qquad\les \big\|\mathbf{1}_{N_1\not\sim N}N^s(N\vee N_1)^{-B}u_{N_1}\big\|_{\ell^q_{N}(2^{\Z_+})L^p(\M)\ell^1_{N_1}(2^{\Z_+})}\\
&\qquad\les \big\|N_1^{|s|-B}u_{N_1}\big\|_{\ell^1_{N_1}(2^{\Z_+})L^p(\M)} \les \|u\|_{B^s_{p,q}(\M)}.
\end{align*}
This concludes the proof of Proposition~\ref{prop-norm}, assuming Lemma~\ref{LEM:Picasso}.
\end{proof}

\begin{proof}[Proof of Lemma~\ref{LEM:Picasso}] 
For $M\gg 1$ to be chosen later, we use Proposition~\ref{prop-dvp} to decompose
\begin{align*} 
\theta_N(D)\kk^{\star}\big(\chi\widetilde{\psi}_{N_1^2}(-\Dlg)u\big) = \theta_N(D)\Big\{\sum_{m=0}^{M-1} N_1^{-m} a_m(x,N_1^{-1}D)\kk^{\star}(\widetilde{\chi}u)+U_{-M,N_1}\Big\},
\end{align*}
where
\begin{align*}
\big\|U_{-M,N_1}\big\|_{H^{s_1}(\M)} \les N_1^{s_1+s_2-M}\|u\|_{H^{-s_2}(\M)}
\end{align*}
for any $s_1,s_2\ge 0$ with $s_1+s_2\le M$, in view of Proposition~\ref{prop-dvp}~(i) with $s=-s_2$ and $\s=s_1+s_2$.

Note that from the support property \eqref{support} of $a_{m}$ and the assumption $(N\vee N_1)\gg (N\wedge N_1)$, we have from the symbolic calculus that $\theta_N(D)\circ a_{m}(x,N_1^{-1}D)$ vanishes at infinite order, but we have to be cautious with the dependence in $N$ and $N_1$ within the remainder in \eqref{def-pscomp}. Namely for any $A\geq 1$, we use the composition rule \eqref{def-pscomp} to expand
\begin{align*}
&\theta_N(D)\circ a_{m}(x,N_1^{-1}D)\\&~~ = \sum_{|\alpha|=0}^{A-1}c_{\alpha}N^{-|\alpha|}\Big\{\partial^{\alpha}\theta_N(\xi)\cdot\partial ^{\alpha}a_{m}(x,N_1^{-1}\xi)\Big\}(x,D)+N^{-A}r_{A,N,N_1}(x,D)
\\&~~=N^{-A}r_{A,N,N_1}(x,D)
\end{align*}
for some constants $c_{\alpha}$. Indeed the last equality results of the support property of $a_{m}$ and the assumption $(N\vee N_1)\gg (N\wedge N_1)$ so that the supports (in $\xi$) of $\theta_N$ and $a_m(x,N_1^{-1}\xi)$ are disjoint. Here $r_{A,N,N_1}$ is a $\Psi$DO with symbol
\begin{align}\label{symbol_rest}
\sum_{|\alpha|=A}c_{\alpha}\int_{\R^d}\int_{\R^d}\int_0^1e^{-iz\cdot\xi_1}\partial^{\alpha}\theta_N(\xi+\xi_1)\partial ^{\alpha}a_{m}(x+tz,N_1^{-1}\xi)(1-t)^{A-1}dtd\xi_1 dz.
\end{align}
This is obtained as a by-product of the proof of the symbolic product rule for $\Psi$DOs: writing down the symbol of the composition, performing the Taylor expansion of this symbol and integrating by parts gives the sum for $|\alpha|<A$, and the rest which corresponds to the symbol in \eqref{symbol_rest}. In particular, in view of the support properties in $\xi$ of $\theta_N(\xi)$ and $a_{m}(x,N_1^{-1}\xi)$ (and the boundedness of $\M$), we can integrate by parts the kernel 
\begin{align*}
R_{A,N,N_1}(x,y)=\frac{1}{(2\pi)^d}\int_{\R^d}e^{i(x-y)\cdot\xi}r_{A,N,N_1}(x,\xi)d\xi
\end{align*} of $r_{A,N,N_1}(x,D)$ with respect to $z$ in \eqref{symbol_rest} to get some negative powers of $\xi_1$. Indeed, for any $\ell_1\in\N$, we integrate by parts to get
\begin{align*}
R_{A,N,N_1}(x,y)=&\sum_{|\alpha|=A}c_{\alpha}\int_{\R^d}\int_{\R^d}\int_{\R^d}\int_0^1\jb{\xi_1}^{-\ell_1}e^{-i(z\cdot\xi_1-(x-y)\cdot\xi)}\partial^{\alpha}\theta_N(\xi+\xi_1)\\&~~\cdot\jb{D_z}^{\ell_1}\big(\partial ^{\alpha}a_{m}(x+tz,N_1^{-1}\xi)t^{|\beta|}\big)(1-t)^{A-1}dtd\xi_1 dzd\xi.
\end{align*} 
Similarly, in order to get some decay in $x$, we can integrate by parts in $\xi$ to get for any $\ell_2\in\N$
\begin{align*}
&R_{A,N,N_1}(x,y)\\&=\sum_{|\alpha|=A}c_{\alpha}\int_{\R^d}\int_{\R^d}\int_{\R^d}\int_0^1\jb{\xi_1}^{-\ell_1}\jb{x-y}^{-\ell_2}e^{-i(z\cdot\xi_1-(x-y)\cdot\xi)}\\&\times\jb{D_{\xi}}^{\ell_2}\Big[\partial^{\alpha}\theta_N(\xi+\xi_1)\jb{D_z}^{\ell_1}\big(\partial^{\alpha}a_{m}(x+tz,N_1^{-1}\xi)\big)(1-t)^{A-1}\Big]dtd\xi_1 dzd\xi.
\end{align*} 
We finally integrate by parts in $\xi_1$ to get some decay in $z$, leading to
\begin{align*}
&R_{A,N,N_1}(x,y)\\
&=\sum_{|\alpha|=A}c_{\alpha}\int_{\R^d}\int_{\R^d}\int_{\R^d}\int_0^1\jb{\xi_1}^{-\ell_1}\jb{x-y}^{-\ell_2}\jb{z}^{-\ell_3}e^{-i(z\cdot\xi_1-(x-y)\cdot\xi)}\\&\times\jb{D_{\xi}}^{\ell_2}\Big[\partial^{\alpha}\jb{D_{\xi_1}}^{\ell_3}\theta_N(\xi+\xi_1)\jb{D_z}^{\ell_1}\big(\partial^{\alpha}a_{m}(x+tz,N_1^{-1}\xi)\big)(1-t)^{A-1}\Big]dtd\xi_1 dzd\xi.
\end{align*}
In view of $(N\vee N_1)\gg (N\wedge N_1)$ and the localization of $\xi$ and $(\xi+\xi_1)$, we have the localization $|\xi_1|\sim (N\vee N_1)$. Moreover, for fixed $\xi_1$, in view of the support properties of $\theta_N$ and $a_m$ then $\xi$ lies in a set of size at most $(N\wedge N_1)^d$. Hence for any $\ell_1,\ell_2,\ell_3>2$ the integrand is absolutely integrable and we get the bound
\begin{align*}
\big|R_{A,N,N_1}(x,y)\big|\les (N\wedge N_1)^d(N\vee N_1)^{d-\ell_1}\jb{x-y}^{-\ell_2}.
\end{align*} 
We can then integrate in $x$ or $y$ provided that we take $\ell_2>d$, to obtain
\begin{align*}
\|R_{A,N,N_1}\|_{L^{\infty}_x L^1_y}+\|R_{A,N,N_1}\|_{L^{\infty}_yL^1_x }\les (N\wedge N_1)^{d}(N\vee N_1)^{d-\ell_1}.
\end{align*}
This is enough to estimate the contribution
\begin{align*}
N^{-A}N_1^{-m}\|\theta_N(D)\circ a_{m}(x,N_1^{-1}D)\kk^{\star}(\wt\chi u)\|_{L^p(\R^d)}
\end{align*}
by the right-hand side of \eqref{estim-key-norm} in view of Schur's lemma, since $\ell_1\in\N$ is arbitrary.

As for the remainder in the use of Proposition~\ref{prop-dvp}, we first take $M=B+s_1+s_2+10$ with $s_1$ and $s_2$ large enough so that, by Sobolev embedding, $H^{s_1}(\R^d)\subset L^p(\R^d)$, and by Sobolev embedding and the compactness of $\M$, $L^p(\M)\subset H^{-s_2}(\M)$. Then, in the case $N\ll N_1$, we use the boundedness of $\theta_N(D): L^p(\R^d)\rightarrow L^p(\R^d)$ to bound
\begin{align*}
\|\theta_N(D)U_{-M,N_1}\|_{L^p(\R^d)}&\les \|U_{-M,N_1}\|_{H^{s_1}(\R^d)}\les N_1^{s_1+s_2-M}\|u\|_{H^{-s_2}(\M)} \les N_1^{-B}\|u\|_{L^p(\M)}.
\end{align*} 
In the other case $N\gg N_1$, using that $\theta_N$ is then supported on an annulus we have
\begin{align*}
\|\theta_N(D)U_{M,N_1}\|_{L^p(\R^d)}&\les N^{-B}\|U_{M,N_1}\|_{H^{s_1+B}(\R^d)}\\
&\les N_1^{s_1+s_2+B-M} N^{-B}\|u\|_{L^p(\M)}\les N^{-B}\|u\|_{L^p(\M)}.
\end{align*}
This concludes the proof of the lemma.
\end{proof}

Using Proposition~\ref{prop-norm}, the finiteness of $\J$ and that the embeddings and the fractional Leibniz rule hold on $\R^d$, we get the following consequences of Proposition \ref{prop-norm}.
\begin{corollary}\label{cor-besov}
Let $\M$ be any compact Riemannian manifold of dimension $d$ without boundary. \\
\textup{(i)} For any $s\in\R$ we have $B^s_{2,2}(\M)=H^s(\M)$, and more generally for any $2\le p <\infty$ and $\eps>0$ we have
\begin{align*}
\|u\|_{B^s_{p,\infty}(\M)}\les \|u\|_{W^{s,p}(\M)}\les \|u\|_{B^s_{p,2}(\M)} \les \|u\|_{B^{s+\eps}_{p,\infty}(\M)}.
\end{align*}
\textup{(ii)} Let $s\in\R$ and $1\leq p_1\leq p_2\leq \infty$ and $q\in [1,\infty]$. Then for any $f\in B^s_{p_1,q}(\M)$ we have
\begin{align*}
\|f\|_{B^{s-d\big(\frac1{p_1}-\frac1{p_2}\big)}_{p_2,q}(\M)}\les \|f\|_{B^s_{p_1,q}(\M)}.
\end{align*}
\textup{(iii)} Let $\alpha,\beta\in \R$ with $\alpha+\beta>0$ and $p_1,p_2,q_1,q_2\in [1,\infty]$ with
\begin{align*}
\frac1p = \frac1{p_1}+\frac1{p_2} \qquad\text{ and }\qquad \frac1q=\frac1{q_1}+\frac1{q_2}.
\end{align*}
 Then for any $f\in B^{\alpha}_{p_1,q_1}(\M)$ and $g\in B^{\beta}_{p_2,q_2}(\M)$, we have $fg \in B^{\alpha\wedge \beta}_{p,q}(\M)$, and moreover it holds
\begin{align*}
\|fg\|_{B^{\alpha\wedge\beta}_{p,q}(\M)}\les \|f\|_{B^{\alpha}_{p_1,q_1}(\M)}\|g\|_{B^{\beta}_{p_2,q_2}(\M)}.
\end{align*}
\end{corollary}
\begin{proof}
The first estimate in (i) is a direct consequence of the boundedness of $(\P_N-\P_{N/2})$ provided by Corollary \ref{lem-boundedness-psi-delta}, whereas the second one follows from the square function estimate given in \cite[Corollary 2.3]{BGT}, and the last one from Cauchy-Schwarz inequality. Similarly, (ii) follows directly from Corollary \ref{lem-boundedness-psi-delta}. 

For the product rule (iii), we take a partition of unity $\{\chi_j\}_{j\in\J}$ and a fattened version $\{\widetilde{\chi}_j\}_{j\in\J}$, so that using Proposition~\ref{prop-norm}, we have
\begin{align*}
\|fg\|_{B^{\alpha\wedge\beta}_{p,q}(\M)}\les \sum_{j\in\J}\|\kk_j^{\star}(\chi_jf\cdot\widetilde{\chi}_jg)\|_{B^{\alpha\wedge\beta}_{p,q}(\R^d)}.
\end{align*}
Then using the standard product rule for Besov spaces on $\R^d$ (see \cite{BCD}, using the paraproduct estimates of Theorems 2.82 and 2.85), we can estimate the term above with
\begin{align*}
\sum_{j\in\J}\|\kk_j^{\star}(\chi_jf)\|_{B_{p_1,q_2}^{\alpha}(\R^d)}\|\kk_j^{\star}(\widetilde{\chi}_jg)\|_{B^{\beta}_{p_2,q_2}(\R^d)}.
\end{align*}
We can then use the finiteness of $\J$ along with Proposition~\ref{prop-norm} to conclude.
\end{proof}
\section{Probabilistic estimates}\label{sec-proba}
\subsection{Probabilistic tools and construction of the Gibbs measure}
We recall briefly here some basic probabilistic estimates and the outline of the construction of the Gibbs measure. A fully detailed construction on a $2d$-manifold can be found in \cite{OT1} in the context of the nonlinear Schr\"odinger equation, which, up to replacing the Laguerre polynomials used in \cite{OT1} with the Hermite polynomials, can be adapted in a straightforward manner to treat the invariant measure for \eqref{SDNLW} and \eqref{NLW}.

Let us first recall a few facts about the Hermite polynomials $H_k(x;\s)$. They are defined through the generating function
\begin{align*}e^{tx-\s\frac{t^2}{2}}=\sum_{k\geq 0}\frac{t^k}{k!}H_k(x;\s),\end{align*}
for any $t,x\in\R$. When $\s=1$ we simply write $H_k(x;1)=H_k(x)$, and we have the scaling property \begin{align*}H_k(x;\s)=\s^{\frac{k}{2}}H_k(\s^{-\frac12}x).\end{align*}
Moreover, the following formula hold:
\begin{align}\label{Hsum}
H_k(x+y;\s) = \sum_{\ell=0}^k {\binom k \ell} H_{\ell}(x;\s)y^{k-\ell}.
\end{align}
and
\begin{align}\label{Hdif}
\dx H_k(x;\s) = kH_{k-1}(x;\s).
\end{align}
Now if we define the (spatial) white noise on $\M$
\begin{align*}\xi_0=\sum_{n\geq 0}g_n\varphi_n,\end{align*}
where $g_n$ are as in \eqref{def-mu}, then we can define the white noise functional to be the action of the distribution $\xi_0$ extended to $L^2$ functions, i.e.
\begin{equation*}
\W: f\in L^2(\M)\longmapsto \W_f = \langle f,\xi_0\rangle_{L^2(\M)}\in L^2(\O).
\end{equation*}
It is easy to see that $\W$ is unitary, and moreover we have the relation
\begin{align}\label{WNF-perp}
\E\big[H_k(\W_f)H_{\ell}(\W_g)\big] = \dl_{k,\ell}k!\langle f,g\rangle_{L^2(\M)}^k,
\end{align}
for any $f,g$ normalized $L^2$ functions, where $\dl_{k,\ell}$ stands for Kronecker's delta function.

As in \cite{GKO}, we also have the following lemma.
\begin{lemma}
Let $f,g$ be centered jointly Gaussian random variables with variances $\s_f$ and $\s_g$, then
\begin{align}\label{prod-wick}
\E\big[H_k(f;\s_f)H_{\ell}(g;\s_g)\big] = \delta_{k,\ell}k!\E\big[fg\big]^k.
\end{align}
\end{lemma}
\noi
See \cite[Lemma 1.1.1]{Nualart}.

Now, if we then define the real-valued random variables $G_{N,k+1}$ on $(H^s(\M),\mu_0)$ as
\begin{align*}
G_{N,k+1}(u_0) = \frac{1}{k+1}\int_{\M} 
:\!  (\P_N u_0)^{k+1} (x)\!: dx ,
\end{align*} 
then we have the following lemma.
\begin{lemma}\label{lem-Gibbs}
Let $G_{N,k+1}$ be the random variable on $(H^s(\M),\mu_0)$ defined above.\\
\noi
\textup{(i)} $\{G_{N,k+1}\}_{N\in\N}$ is a Cauchy sequence in $L^p(\mu_0)$ for any finite $p\geq 1$, thus converging to some $G_{k+1}\in L^p(\mu_0)$,\\
\noi
\textup{(ii)} $e^{-G_{N,k+1}}$ converges to $e^{-G_{k+1}}$ almost surely and in $L^p(\mu_0)$ for any finite $p\geq 1$.
\end{lemma}
This last convergence result allows to define the Gibbs measure $\rho_{k+1}$ as the limit in total variation of $Z_N^{-1}e^{-G_{N,k+1}}d\mu$.

The proof of (i) for $p=2$ follows from a direct computation using \eqref{WNF-perp} and Lemma \ref{LEM:ev}, and for $p>2$ it is a consequence of the case $p=2$ along with the following Wiener chaos estimate (see \cite{Simon}):
\begin{lemma}
Let $d,m\in\N$ and $Q(X_1,...,X_m)$ be a polynomial of degree $d$ in $m$ variables. Let $\{g_n\}$ be as in \eqref{def-mu}. Then for any $p\geq 2$ we have
\begin{align}\label{Wiener-chaos}
\|Q(g_1,...,g_m)\|_{L^p(\O)}\leq (p-1)^{\frac{d}{2}}\|Q(g_1,...,g_m)\|_{L^2(\O)}.
\end{align}
\end{lemma}
This lemma is itself a consequence of the hypercontractivity of Ornstein-Uhlenbeck's semi-group \cite{Nelson2}. As for Lemma \ref{lem-Gibbs} (ii), it then follows from the same argument as in \cite[Proposition 4.5]{OT1}.

As explained in the introduction, Lemma~\ref{lem-Gibbs} allows us to define the Gibbs measure $\rho_{k+1}$ on $\H^s(\M)$ by the formula \eqref{def-rhok}. In particular, $\rho_{k+1}\ll\mu$ as $e^{-G_{k+1}}$ is a finite positive random variable, so that $\supp \rho_{k+1} = \supp \mu \subset \H^s(\M)$ for any $s<0$ but $\rho_{k+1}(\H^0)=0$.

\subsection{Stochastic estimates for (\ref{SDNLW}) and (\ref{NLW})}
Now we move onto the construction of the Wick ordered monomials $:\Psi_{\damp}^k:$ and their large deviation bounds. We first deal with the stochastic objects for \eqref{NLW}, and so we recall that $z_N=\P_NS(t)(u_0,u_1)$ is the truncated linear solution with the random initial data $(u_0,u_1)$ given in \eqref{def-mu}.
\begin{proposition}\label{prop-z}
For any $k\geq 1$, $T>0$, $0<\eps \ll 1$ and $1\leq p,q<\infty$\footnote{Unlike when $\M=\T^2$, it is not as straightforward to get the convergence of $H_\ell\big(z_N;\s_N(x)\big)$ in $C\big([0,T];W^{-\eps,\infty}(\M)\big)$ almost surely when $\ell\ge 2$, which prevents us from taking $q=\infty$. See also Remark \ref{rk-FIO} below.}, the random variables $\big\{H_k\big(\P_NS(t)(u_0,u_1);\s_N(x)\big)\big\}_{N\in \N}$ form a Cauchy sequence in $L^p\big(\mu;L^q([0,T];W^{-\eps,\infty}(\M))\big)$. Moreover, there exists $C>0$ such that for any $T,R>0$ and $N\in\N$ the following tail estimate holds:
\begin{align}\label{tail-z}
\mu\Big(\big\|H_k\big(\P_NS(t)(u_0,u_1);\s_N(x)\big)\big\|_{L^q_TW^{-\eps,\infty}}>R\Big)\leq Ce^{-cR^{\frac2k}T^{-\frac2{qk}}}.
\end{align} Denoting the limit by $\,:\!z^k\!:\,$, it also holds $H_k\big(\P_NS(t)(u_0,u_1);\s_N(x)\big)\rightarrow \,:\!z^k\!:\,$ in $L^q([0,T],W^{-\eps,\infty}(\M))$,  $\mu$-almost surely, and $\,:\!z^k\!:\,$ also satisfies the tail estimate \eqref{tail-z}. Moreover, for $k=1$ we have $z\in C\big([0,T];W^{-\eps,\infty}(\M)\big)\cap C^1\big([0,T],W^{-1-\eps,\infty}(\M)\big) $, $\mu$-almost surely, for any $\eps>0$. Lastly, we also have the following tail estimate for the convergence:
\begin{align}
 \mu\Big(\big\|H_k\big(\P_{N_1}S(t)(u_0,u_1);\s_{N_1}(x)\big)&-H_k\big(\P_{N_2}S(t)(u_0,u_1);\s_{N_2}(x)\big)\big\|_{L^q_TW^{-\eps,\infty}}>R\Big)\notag\\
 &\qquad\qquad\leq Ce^{-cN_1^{\widetilde{\eps}}R^{\frac2k}T^{-\frac2{qk}}},
 \label{tail-z-zN}
 \end{align}
 for some $0<\wt\eps\ll\eps$ and any $N_2\ge N_1$.
\end{proposition}
\begin{proof}
We begin by proving that ${\displaystyle H_k\big(\P_NS(t)(u_0,u_1);\s_N(x)\big)}$ is uniformly bounded in $L^p\big(\mu;L^q([0,T];W^{-\eps,\infty}(\M))\big)$. Note that it is enough to consider the case $p,q\ge 2$. In the following, we write $\x,\y$ for the space variables on $\M$ and $x,y$ for the points in $\R^2$. Let us start with the following lemma which collects the main properties of $\P_NS(t)(u_0,u_1)$ that we will use.
\begin{lemma}\label{lem-covar}
The measure $\mu$ is invariant under the transformation $(u_0,u_1)\mapsto \big(S(t)(u_0,u_1),\dt S(t)(u_0,u_1)\big)$, for any $t\in\R$. Moreover, if we define the (truncated) covariance function
\begin{align*}
\gamma_N(t_1,t_2,\x,\y)\deff\int_{\H^s(\M)}\big[\P_NS(t_1)(u_0,u_1)(\x)\P_NS(t_2)(u_0,u_1)(\y)\big]d\mu(u_0,u_1),
\end{align*}
then we have for any $(t,\x)\in \R\times\M$
\begin{align}
&\int_{\H^s(\M)}\big|(1-\Dlg)_\x^{-\frac\eps2}H_k\big(\P_NS(t)(u_0,u_1)(\x);\s_N(\x)\big)\big|^2d\mu(u_0,u_1)\notag\\ &\qquad\qquad=k!\big[(1-\Dlg)_{\x_1}^{-\frac\eps2}(1-\Dlg)_{\x_2}^{-\frac\eps2}\big(\gamma_N(t,t,\x_1,\x_2)^k\big)\big]\big|_{\x_1=\x_2=\x}.
\label{keypoint}
\end{align}
Lastly, we have the identity
\begin{align}
\gamma_N(\x,\y)\deff\gamma_N(t,t,\x,\y) = (\P_N\otimes\P_N)\gamma(\x,\y),
\label{covarGreen}
\end{align}
where $\gamma$ is the Green function for the Laplace-Beltrami operator on $\M$, i.e. $ \gamma$ is the kernel of $(1-\Dlg)^{-1}$:
\begin{align*}
\gamma(\x,\y)=\sum_{n\ge 0}\frac{\varphi_n(\x)\varphi_n(\y)}{\jb{\ld_n}^2}.
\end{align*}
\end{lemma}
Here the notation $(\P_N\otimes\P_N)\gamma(\x,\y)$ means that we apply $\P_N$ to both $\gamma(\cdot,\y)$ and $\gamma(\x,\cdot)$. Note that since $\gamma$ has a diagonal expansion on the basis $\varphi_n\otimes\varphi_{n'}$ of $L^2(\M\times\M)$, this is the same as $(\P_N^2\otimes\Id)\gamma$ or $(\Id\otimes\P_N^2)\gamma$.
\begin{proof}[Proof of Lemma~\ref{lem-covar}]
In order to prove the invariance, we first compute for $(u_0^{\o_0},u_1^{\o_1})$ given by \eqref{def-mu}:
\begin{align*}
S(t)(u_0^{\o_0},u_1^{\o_1}) = \sum_{n\geq 0}\frac{\varphi_n}{\jb{\lambda_n}}\big\{\cos(t\jb{\lambda_n})g_n(\o_0)+\sin(t\jb{\lambda_n})h_n(\o_1)\big\}=\sum_{n\geq 0}\frac{\varphi_n}{\jb{\lambda_n}}g_n^t(\o_0,\o_1),
\end{align*}
where for any $t\in\R$, $\{g_n^t\}_{n\ge 0}$ is a family of independent real-valued standard Gaussian random variables on $\O_0\times\O_1$, and similarly for $\dt S(t)(u_0^{\o_0},u_1^{\o_1})$. In particular this shows that if $(u_0,u_1)\sim \mu$ then for any $t\in\R$, $\big(S(t)(u_0,u_1),\dt S(t)(u_0,u_1)\big)\sim\mu$ too. 

Next, with the definition of the operator $(1-\Dlg)^{-\frac\eps2}$, we compute for any fixed $(t,\x)\in [0,T]\times\M$:
\begin{align*}
&\int_{\O_0}\int_{\O_1}\big|(1-\Dlg)_\x^{-\frac\eps2}H_k\big(\P_NS(t)(u_0^{\o_0},u_1^{\o_1})(\x);\s_N(\x)\big)\big|^2d\Prob_0d\Prob_1\\
&= \sum_{n,n'\in\N}\frac{\varphi_n(\x)\varphi_{n'}(\x)}{\jb{\lambda_n}^{\eps}\jb{\lambda_{n'}}^{\eps}}\int_{\M\times\M}\varphi_n(\x_1)\varphi_{n'}(\x_2)\\
&\times\E\Big[H_k\big(\P_NS(t)(u_0^{\o_0},u_1^{\o_1})(\x_1);\s_N(\x_1)\big)H_k\big(\P_NS(t)(u_0^{\o_0},u_1^{\o_1})(\x_2);\s_N(\x_2)\big)\Big]d\x_1d\x_2
\intertext{where the expectation is taken with respect to $\Prob_0\otimes\Prob_1$. We can then use \eqref{prod-wick} and the definition of $\gamma_N(t,t,\x_1,\x_2)$ to continue with}\\
&= \sum_{n,n'\in\N}\frac{\varphi_n(\x)\varphi_{n'}(\x)}{\jb{\lambda_n}^{\eps}\jb{\lambda_{n'}}^{\eps}}\int_{\M\times\M}k! \gamma_N(t,t,\x_1,\x_2)^k\varphi_n(\x_1)\varphi_{n'}(\x_2)d\x_1 d\x_2\\
&=k!\big[(1-\Dlg)_{\x_1}^{-\frac\eps2}(1-\Dlg)_{\x_2}^{-\frac\eps2}\big(\gamma_N(t,t,\x_1,\x_2)^k\big)\big]\big|_{\x_1=\x_2=\x}.
\end{align*}
This shows \eqref{keypoint}. 

As for \eqref{covarGreen}, in view of the definitions of the (truncated) covariance function $\gamma$ and of the propagator $S(t)$, we can compute
\begin{align}
&\gamma_N(t_1,t_2,\x,\y)\notag\\
&=\sum_{n_1,n_2\ge 0}\psi_0(N^{-2}\lambda_{n_1}^2)\psi_0(N^{-2}\lambda_{n_2}^2)\frac{\varphi_{n_1}(\x)\varphi_{n_2}(\y)}{\jb{\lambda_{n_1}}\jb{\lambda_{n_2}}} \int_{\O_0}\int_{\O_1}\Big[\big(\cos(t_1\jb{\lambda_{n_1}})g_{n_1}+\sin(t_1\jb{\lambda_{n_1}})h_{n_1}\big)\notag\\
&\qquad\qquad\times\big(\cos(t_2\jb{\lambda_{n_2}})g_{n_2}+\sin(t_2\jb{\lambda_{n_2}})h_{n_2}\big)\Big]d\Prob_0d\Prob_1\notag\\
&=\sum_{n\ge 0}\psi_0^2(N^{-2}\lambda_{n}^2)\frac{\varphi_{n}(\x)\varphi_{n}(\y)}{\jb{\lambda_{n}}^2}\cos\big((t_1-t_2)\jb{\lambda_{n}}\big).
\label{def-gamma}
\end{align}
The identity \eqref{covarGreen} thus follows from \eqref{def-gamma} by taking $t_1=t_2$.

\end{proof}
 Note that in order to estimate the right-hand side of \eqref{keypoint}, we do not need the smoothing in $\x_1$, and using Sobolev inequality in $\x_1$ with some (large) $p_{\eps}$ and the compactness of $\M$, we have
\begin{align}
&\sup_{\x\in\M}\big[(1-\Dlg)_{\x_1}^{-\frac\eps2}(1-\Dlg)_{\x_2}^{-\frac\eps2}\big(\gamma_N(\x_1,\x_2)^k\big)\big]\big|_{\x_1=\x_2=\x}\notag\\
& \les \|(1-\Dlg)_{\x_2}^{-\frac\eps2}\big(\gamma_N(\x_1,\x_2)^k\big)\|_{L^{p_\eps}(\M)\times L^{\infty}(\M)}\label{rk-gamma}\\
&\les  \|(1-\Dlg)_{\x_2}^{-\frac\eps2}\big(\gamma_N(\x_1,\x_2)^k\big)\|_{L^{\infty}(\M\times\M)}\notag.
\end{align}
The following proposition allows us to bound the powers of the covariance function $\gamma_N$, viewed through the identity \eqref{covarGreen}.
\begin{proposition}\label{prop-Green}
Let $\gamma_N: \M\times\M\rightarrow\R$ be the truncated Green function of the Laplace-Beltrami operator on $\M$ defined in \eqref{covarGreen}. Then for any $\eps>0$ and $k\in\N$, there exists $C=C(\eps,k)>0$ such that for any $N\in\N$,
\begin{align}\label{estim-Green}
\big\|(1-\Dlg)_{\x_2}^{-\frac\eps2}\big(\gamma_N(\x_1,\x_2)^k\big)\big\|_{L^{\infty}(\M\times\M)} \leq C<\infty.
\end{align}
Moreover, $\{\gamma_N^k\}_{N\in\N}$ defines a Cauchy sequence in 
\begin{align*}
W^{0,-\eps,\infty}(\M\times\M)=\big\{u\in\Dp(\M\times\M),~\|(1-\Dlg)_{\x_2}^{-\frac\eps2}u(\x_1,\x_2)\|_{L^{\infty}(\M\times\M)}<\infty\big\}
\end{align*} and satisfies
\begin{align}\label{estim-Green-diff}
\big\|(1-\Dlg)_{\x_2}^{-\frac\eps2}\big(\gamma_{N_1}(\x_1,\x_2)^k-\gamma_{N_2}(\x_1,\x_2)^k\big)\big\|_{L^{\infty}(\M\times\M)} \leq CN_1^{-\tilde{\eps}},
\end{align}
for any $N_1\leq N_2\in\N$ and some $0<\tilde{\eps}\ll\eps$ and $C>0$ independent of $N_1,N_2$.

Finally, if $\wt\P_N$ is defined similarly to $\P_N$ but with another cut-off $\wt\psi_0$ in place of $\psi_0$ with the same properties, then
\begin{align}\label{estim-Green-diff2}
\big\|(1-\Dlg)_{\x_2}^{-\frac\eps2}\big(\wt\P_N^2\gamma(\x_1,\x_2)^k-\P_N^2\gamma(\x_1,\x_2)^k\big)\big\|_{L^{\infty}(\M\times\M)} \leq CN^{-\tilde{\eps}}.
\end{align}
\end{proposition}
We postpone the proof of this proposition and finish the proof of Proposition~\ref{prop-z}. Now, for any finite $p\geq 1$, we first use Sobolev inequality to get for any $t\in\R$:
\begin{align*}
\big\|H_k\big(\P_NS(t)(u_0,u_1);\s_N(\x)\big)\big\|_{W^{-\eps,\infty} }\les \big\|H_k\big(\P_NS(t)(u_0,u_1);\s_N(\x)\big)\big\|_{W^{-\frac{\eps}{2},r_{\eps}} },
\end{align*}
for some $r_{\eps}\in [2,+\infty)$. Thus if $p\geq \max(q, r_{\eps})$, using Minkowski's inequality, the Wiener chaos estimate \eqref{Wiener-chaos} along with \eqref{keypoint}, \eqref{rk-gamma} and Proposition \ref{prop-Green} with the compactness of $\M$, we obtain
\begin{align*}
&\big\|H_k\big(\P_NS(t)(u_0,u_1)(\x);\s_N(\x)\big)\big\|_{ L^p_{\mu}L^{q}_TW^{-\eps,\infty} }\\&\les \Big\|\big\|(1-\Dlg)^{-\frac\eps2}H_k\big(\P_NS(t)(u_0,u_1)(\x);\s_N(\x)\big)\big\|_{L^p_\mu}\Big\|_{L^q_TL^{r_{\eps}} }\\
& \les p^{k/2} \Big\|\big\|(1-\Dlg)^{-\frac\eps2}H_k\big(\P_NS(t)(u_0,u_1)(\x);\s_N(\x)\big)\big\|_{L^2_\mu}\Big\|_{L^q_TL^{r_{\eps}}}\\
&= p^{k/2} \sqrt{k!}\big\|(1-\Dlg)_{\x_1}^{-\frac\eps2}(1-\Dlg)_{\x_2}^{-\frac{\eps}2}\big(\gamma_N(\x_1,\x_2)^k\big))\big|_{\x_1=\x_2=\x}\big\|_{L^{\frac{q}2}_TL^{\frac{r_{\eps}}2}_\x}^\frac12\\
&\les_k T^{1/q}p^{k/2}.
\end{align*}  
This proves that ${\displaystyle \big\{H_k\big(\P_NS(t)(u_0,u_1)(\x);\s_N(\x)\big)\big\}_{N\in\N}}$ is bounded in ${\displaystyle L^p\big(\mu;L^q([-T,T];W^{-\eps,\infty}(\M))\big)}$ for any finite $p,q\geq 1$ with $p$ large enough. Using then Chebyshev's inequality, we get that there is $C>0$ such that for any $p\geq 1$ and $R>0$
\begin{align*}
&\mu\Big(\big\|H_k\big(\P_NS(t)(u_0,u_1)(\x);\s_N(\x)\big)\big\|_{L^q_TW^{-\eps,\infty}}>R\Big) \\
&\qquad\leq R^{-p}\big\|H_k\big(\P_NS(t)(u_0,u_1)(\x);\s_N(\x)\big)\big\|_{L^p_{\mu} L^q_TW^{-\eps,\infty} }^p\le C^p p^{p\frac{k}2}T^{\frac{p}q}R^{-p},
\end{align*}
 and optimizing in $p$ leads to \eqref{tail-z}. 
 
 Now for any $N_1< N_2$, we can compute, similarly to \eqref{keypoint},
 \begin{align*}
& \int_{\H^s(\M)}\Big|(1-\Dlg)^{-\frac\eps2}\Big[H_k\big(\P_{N_1}S(t)(u_0,u_1)(\x);\s_{N_1}(\x)\big)\\
&\qquad\qquad-H_k\big(\P_{N_2}S(t)(u_0,u_1)(\x);\s_{N_2}(\x)\big)\Big]\Big|^2d\mu(u_0,u_1)\\ &=k!\big[(1-\Dlg)_{\x_1}^{-\frac\eps2}(1-\Dlg)_{\x_2}^{-\frac\eps2}\big(\P_{N_1}^2\gamma(\x_1,\x_2)^k-2\P_{N_1}\gamma(\x_1,\x_2)^k+\P_{N_2}^2\gamma(\x_1,\x_2)^k\big)\big]\big|_{\x_1=\x_2=\x},
\end{align*}
where we used that $\P_{N_2}\P_{N_1}=\P_{N_1}$ for $N_2> N_1$. Then \eqref{estim-Green-diff}-\eqref{estim-Green-diff2} in Proposition \ref{prop-Green} show that the sequence $\big\{H_k\big(\P_{N}S(t)(u_0,u_1)(\x);\s_{N}(\x)\big)\big\}_{N\in\N}$ defines a Cauchy sequence, thus converging to some $:z^k:$ in $L^p(\mu;L^q([-T,T];W^{-\eps,\infty}))$ and from the same argument as above with \eqref{estim-Green-diff} we get the tail estimate \eqref{tail-z-zN}. Then, as in the proof of Proposition 3.2 in \cite{OPT}, Borel-Cantelli's lemma yields that $H_k\big(\P_{N}S(t)(u_0,u_1)(\x);\s_{N}(\x)\big)$ converges to $:z^k:$ in $L^q([-T,T];W^{-\eps,\infty})$, $\mu$-almost surely, and moreover $:z^k:$ also satisfies \eqref{tail-z}.

Lastly, we prove the continuity in time of $z$. If we define the translation operator $\th: u \mapsto u(\cdot+h)$ for any $h \in [-1,1]$, we can use \eqref{def-gamma} and the mean value theorem to estimate
\begin{align*}
&\int_{\H^s(\M)}\big|(1-\Dlg)^{-\frac\eps2}(\th z-z)(t,\x)\big|^2d\mu\\& =2\sum_{n_1,n_2}\frac{\varphi_{n_1}(\x)\varphi_{n_2}(\x)}{\jb{\lambda_{n_1}}^{\eps}\jb{\lambda_{n_2}}^{\eps}}\int_{\M}\int_{\M}\varphi_{n_1}(\x_1)\varphi_{n_2}(\x_2)\big\{\gamma(\x_1,\x_2)-\gamma(t+h,t,\x_1,\x_2)\big\}d\x_1d\x_2\\ 
&\,\,\les \sum_{n\geq 0}\frac{\varphi_n(\x)^2}{\jb{\lambda_n}^{2+2\eps}}(1\wedge |h|\jb{\lambda_n})   \les \sum_{n\geq 0} \frac{\varphi_n(\x)^2}{\jb{\lambda_n}^{2+\eps}} |h|^{\eps},
\end{align*}
uniformly in $h\in [-1,1]$, $\x\in\M$ and $t\in\R$. Finally, using Lemma~\ref{LEM:ev}, we obtain the bound
\begin{align*}
\int_{\H^s(\M)}\big|(1-\Dlg)^{-\frac\eps2}(\th z-z)(t,\x)\big|^2d\mu &\les |h|^{\eps}\sum_{k\geq 0}\jb{k}^{-\eps} \sum_{n\geq 0} \mathbf{1}_{[k,k+1)}(\lambda_n)\frac{\varphi_n(\x)^2}{\jb{\lambda_n}^{2}}\\
&\les |h|^{\eps}\sum_{k\geq 0}\jb{k}^{-\eps} \sum_{n\geq 0} \mathbf{1}_{[k,k+1)}(\lambda_n)\frac{1}{\jb{\lambda_n}^{2}}\\
&\les |h|^{\eps}\sum_{n\geq 0} \frac{1}{\jb{\lambda_n}^{2+\eps}}\les |h|^{\eps}.
\end{align*}
 Hence using Sobolev's and Minkowski's inequalities as above, together with the Wiener chaos estimate \eqref{Wiener-chaos}
\begin{align}
\|(\th z-z)(t)\|_{L^p_{\mu}W^{-\eps,\infty} }^p \les |h|^{p\frac{\eps}{2}},
\label{kolmogorov}
\end{align}
uniformly in $t\in[0,T]$, which suffices to conclude that $z\in C\big([0,T];W^{-\eps,\infty}(\M)\big)$ almost surely by using Kolmogorov's continuity criterion for $p$ large enough. We can use the same argument to bound $\dt z$ in $C\big([0,T];W^{-1-\eps,\infty}(\M)\big)$ almost surely, which concludes the proof of Proposition \ref{prop-z}.
\end{proof}
\begin{proof}[Proof of Proposition \ref{prop-Green}]
We now give the proof of \eqref{estim-Green}. Since this is clear for $N\les 1$, we can assume that $N\gg 1$. First, in view of the finiteness of $\J$, it is enough to fix $j,j_1\in \J$ and to estimate 
\begin{equation}\label{G1}
 \Big\|(\kk_{j}\otimes\kk_{j_1})^{\star}\big\{\chi_{j}(\x)\chi_{j_1}(\y)(1-\Dlg)_{\y}^{-\frac\eps2}(\gamma_N(\x,\y))^k\big\}\Big\|_{L^{\infty}(\R^2\times\R^2)},
\end{equation}
where for functions $f$ on $\M\times\M$  and $(x,y)\in U_j\times U_{j_1}$ we write $(\kk_{j}\otimes\kk_{j_1})^\star f(x,y) = f(\kk_j(x),\kk_{j_1}(y))$.

By a variant of Proposition~\ref{prop-dvp} (see Remark~\ref{rk-pdolocal}) with fixed $x\in\R^d$, we can write
\begin{align*} 
(\kk_{j}\otimes\kk_{j_1})^{\star}\big\{\chi_{j}\chi_{j_1}(1-\Dlg)_{\y}^{-\frac\eps2}(\gamma_N)^k\big\}= a_{j_1,-\eps}(y,D)\big\{(\kk_{j}\otimes\kk_{j_1})^{\star}\big(\chi_{j}\widetilde{\chi}_{j_1}(\gamma_N)^k\big)\big\} + G_{-M,N}
\end{align*}
for some symbol $a_{j_1,-\eps}\in \S^{-\eps}(\R^2\times\R^2)$ with  compact support in $y$ included in $U_{j_1}$, some fattened version $\widetilde{\chi}_{j_1}$ of $\chi_{j_1}$, and for arbitrary $M>0$ with
\begin{align*}
\big\|G_{-M,N}\big\|_{L^{\infty}(\R^2)\times H^{s_1}(\R^2)} \les N^{s_1+s_2-M}\big\|\chi_j(\x)\gamma_N(\x,\y)^k\|_{L^{\infty}(\M)\times H^{-s_2}(\M)}
\end{align*}
for any $s_1,s_2\ge 0$ with $s_1+s_2\le M$. In particular the contribution of this last term to \eqref{G1} is
\begin{align*}
\big\|G_{-M,N}\big\|_{L^{\infty}(\R^2\times\R^2)} &\les \big\|G_{-M,N}\big\|_{L^{\infty}(\R^2)\times H^2(\R^2)} \les N^{2-M}\big\|\gamma_N\big\|_{L^{\infty}(\M)\times L^{\infty}(\M)}^k\\
& \les N^{2-M}\sup_{\x,\y\in\M}\Big(\sum_{n\ge 0}\psi_0(N^{-2}\ld_n^2)\frac{\varphi_n(\x)^2}{\jb{\ld_n}^2}\Big)^{\frac{k}2}\Big(\sum_{n\ge 0}\psi_0(N^{-2}\ld_n^2)\frac{\varphi_n(\y)^2}{\jb{\ld_n}^2}\Big)^{\frac{k}2}\\
& = O\Big(N^{2-M}\log(N)^k\Big),
\end{align*}
where the last two estimates come from Cauchy-Schwarz inequality and \eqref{sN}. This term is uniformly bounded by choosing $M>2$.

Taking again fattened versions of $\chi_{j},\widetilde{\chi}_{j_{1}}$ (which to simplify notations we still write $\chi_{j},\chi_{j_{1}}$) it then remains to estimate
\begin{align*}
a_{j_1,-\eps}(y,D)\big\{(\kk_{j}\otimes\kk_{j_1})^{\star}(\chi_{j}\chi_{j_1}\gamma_N)\big\}^k.
\end{align*}

Now, in view of the definition of the functional calculus and \eqref{covarGreen}, we can see $\gamma_N$ as the kernel of the $\Psi$DO $(1-\Dlg)^{-1}\psi_0^2(-N^2\Dlg)$. First, using e.g. \cite[Theorem 18.1.24]{HormanderIII}, we can expand the resolvent as
\begin{align*}
\kk_j^{\star}\big(\chi_j(1-\Dlg)^{-1}\big) = a_{j,-2}(x,D)\kk_j^{\star}\wt\chi_j + R_{j,-3}
\end{align*}
for some symbol $a_{j,-2}\in \S^{-2}(\R^2\times\R^2)$ compactly supported in $x$ in $U_j$, and some smoothing operator $R_{j,-3}$ of order $-3$ satisfying for any $s\in\R$
\begin{align*}
\big\|R_{j,-3}\big\|_{H^{s}(\M)\to H^{s+3}(\R^d)}\les 1.
\end{align*}
Next, using Proposition~\ref{prop-dvp}, and dropping the tilde for the fattened cut-offs, we get the expansion
\begin{align*}
&\kk_{j}^{\star}\big(\chi_{j}(1-\Dlg)^{-1}\psi_0^2(-N^2\Dlg)\chi_{j_1}\big)\\
&=\Big[a_{j,-2}(x,D)\kk_{j}^{\star}\chi_{j} + R_{j,-3}\Big]\psi_0^2(-N^2\Dlg)\chi_{j_1}
\\&=a_{j,-2}(x,D)(\kk_j^\star\chi_{j})\Big(\sum_{m=0}^{M-1}N^{-m}a_{j,m}(x,N^{-1}D)\kk_j^{\star}(\chi_j\chi_{j_1}\big)+R_{j,-M,N}\chi_{j_1}\Big)\\
&\qquad\qquad\qquad + R_{j,-3}\psi_0^2(-N^2\Dlg)\chi_{j_1},
\end{align*}
where $R_{j,-M,N}$ is a smoothing operator of order $-M$, with
\begin{align*}
\big\|R_{j,-M,N}\big\|_{H^{-s_2}(\M)\to H^{s_1}(\R^d)} \les N^{s_1+s_2-M},
\end{align*} 
for any $s_1,s_2\ge 0$ with $s_1+s_2\le M$.
  
 Then, taking $M=1$ in the above expansion, we have for any $(x,y)\in U_j\times U_{j_1}$ and $(\x,\y)=(\kk_j(x),\kk_{j_1}(y))$:
\begin{align}
\gamma_{N,j,j_1}(x,y)&\overset{\text{def}}{=}(\kk_{j}\otimes\kk_{j_1})^{\star}\big(\chi_{j}(\x)\chi_{j_1}(\y)\gamma_N(\x,\y)\big)\notag\\
& = \big(\Id\otimes \zeta_{j,j_1}\big)^{\star}\big(\kk_j^{\star}(\chi_j\chi_{j_1})(y)K_0(x,y)\big) + \chi_{j_1}(\y)K_1(x,\y),
\label{gammaNj}
\end{align}
where $K_0$ is the kernel of 
\begin{align}\label{K0}
(\kk_j^\star\chi_j)a_{j,-2}(x,D)(\kk_j^\star\widetilde{\chi}_j)\psi_0^2(-p_{j,2}(x,N^{-1}D)),
\end{align}
and $K_1$ the one to 
\begin{align}\label{K1}
(\kk_j^\star\chi_j)a_{j,-2}(x,D)(\kk_j^\star\widetilde{\chi}_j)R_{j,-M,N}+\kk_j^{\star}(\chi_j)R_{j,-3}\psi_0^2(-N^{-2}\Dlg).
\end{align}
Here $\zeta_{j,j_1}= \kk_j^{-1}\circ\kk_{j_1}$ is a diffeomorphism on $U_j\cap U_{j_1}$, provided that $\supp \chi_j\cap \supp\chi_{j_1}\neq \emptyset$, otherwise the contribution of $K_0$ in \eqref{gammaNj} vanishes. Let us also decompose $K_1 = K_{1,1}+ K_{1,2}$ corresponding to the two operators in \eqref{K1}. 

We will use that we can bound these kernels by the operator norm of the corresponding operators from $H^{-1-\dl}(\R^2)$ (or $H^{-1-\dl}(\M)$) to $H^{1+\dl}(\R^2)$. For $K_{1,1}$, since $a_{j,-2}(x,D)$ is bounded from $H^{s_1}(\R^2)$ to $H^{s_1+2}(\R^2)$, and using the smoothing property of $R_{j,-M,N}$ for $M=1$, we deduce that for $s_1 = -1+\dl$ and $s_2=1+\dl$ for some ${\displaystyle 0<\dl<\frac{1}{2k}}$, the operator with kernel $K_{1,1}$ maps $H^{-1-\dl}(\M)$ to $H^{1+\dl}(\R^2)$ with operator norm bounded by $N^{2\dl-1}$. Thus we obtain
\begin{align}
&\big\|\kk_{j_1}^{\star}\big(\chi_{j_1}(y)K_{1,1}(x,y)\big)\big\|_{L^{\infty}(\R^2\times\R^2)}\notag\\
&\les \big\|(\kk_j^\star\chi_j)a_{j,-2}(x,D)(\kk_j^\star\widetilde{\chi}_j)R_{j,-M,N}(\kk_{j_1}^{\star}\chi_{j_1})\big\|_{H^{-1-\dl}(\M)\to H^{1+\dl}(\R^2)}\notag\\
&\les \big\|R_{j,-M,N}(\kk_{j_1}^{\star}\chi_{j_1})\big\|_{H^{-1-\dl}(\M)\to H^{-1+\dl}(\R^2)}\notag\\
&\les N^{2\dl-1}\big\|(\kk_{j_1}^{\star}\chi_{j_1})\big\|_{H^{-1-\dl}(\M)\to H^{-1-\dl}(\M)}\notag\\
&\les N^{2\dl-1},\label{K11}
\end{align}
where in the last step we used the product rule of Corollary~\ref{cor-besov}~(iii).

As for $K_{1,2}$, we have that $\|R_{j,-3}\|_{H^{-2+\dl}(\M)\to H^{1+\dl}(\R^2)}\les 1$ and since we assumed that $N\gg 1$ we also have that $\|\psi_0(-N^{-2}\Dlg)\|_{H^{-1-\dl}(\M)\to H^{-2+\dl}(\M)}\les N^{2\dl-1}$. Thus we also have the bound
\begin{align*}
\big\|\kk_{j_1}^{\star}\big(\chi_{j_1}K_{1,2}\big)\big\|_{L^{\infty}(\R^2\times\R^2)}\les N^{2\dl-1}.
\end{align*}

Now we compute
\begin{align*}
\gamma_{N,j,j_1}^k &= \Big((\Id\otimes\zeta_{j,j_1})^{\star}\big(\kk_j^{\star}(\chi_j\chi_{j_1})K_0\big)\Big)^k\\
&\qquad\qquad+\sum_{\ell=1}^k{\binom k \ell}\big(\kk_{j_1}^{\star}(\chi_{j_1}K_1)\big)^{\ell}\big((\Id\otimes\zeta_{j,j_1})^{\star}[\kk_j^\star(\chi_j\chi_{j_1})K_0]\big)^{k-\ell}.
\end{align*}
We first deal with the terms with $\ell \ge 1$. Since $a_{j_1,-\eps}\in\S^{-\eps}(\R^2\times\R^2)$, in particular it is bounded on $L^p(\R^2)$ for any $1<p<\infty$ (see e.g. \cite{F}), hence using as above the Sobolev inequality $W^{\eps,r_{\eps}}(\R^d)\subset L^\infty(\R^d)$ for some $r_\eps\gg1$, and the compactness of $\supp \big(\kk_j^{\star}(\chi_j\chi_{j_1})K_0\big)$ and $\supp \kk_{j_1}^{\star}(\chi_{j_1}K_1)$, we get the crude estimate
\begin{align*}
&\Big\| a_{j_1,-\eps}(y,D)\Big\{\big(\kk_{j_1}^{\star}(\chi_{j_1}K_1)\big)^{\ell}\Big((\Id\otimes\zeta_{j,j_1})^{\star}\big(\kk_j^{\star}(\chi_j\chi_{j_1})K_0\big)\Big)^{k-\ell}\Big\}\Big\|_{L^{\infty}(\R^2\times\R^2)}\\
&\qquad\les \Big\|\big(\kk_{j_1}^{\star}(\chi_{j_1}K_1)\big)^{\ell}\Big((\Id\otimes\zeta_{j,j_1})^{\star}\big(\kk_j^{\star}(\chi_j\chi_{j_1})K_0\big)\Big)^{k-\ell}\Big\|_{L^{\infty}(\R^2\times\R^2)}\\ &\qquad \les \big\|\kk_{j_1}^{\star}(\chi_{j_1}K_1)\big\|_{L^{\infty}(\R^2\times\R^2)}^{\ell}\big\|\kk_j^{\star}(\chi_j\chi_{j_1})K_0\big\|_{L^{\infty}(\R^2\times\R^2)}^{k-\ell}.
\end{align*}
Along with the previous bounds for $K_{1,1}$ and $K_{1,2}$, we finally obtain
\begin{align}
\big\| a_{j_1,-\eps}(y,D)(\gamma_{N,j,j_1}^k)\big\|_{L^{\infty}(\R^2\times\R^2)}&\les \Big\| a_{j_1,-\eps}(y,D)\Big((\Id\otimes\zeta_{j,j_1})^{\star}\big(\kk_j^{\star}(\chi_j\chi_{j_1})K_0\big)\Big)^k\Big\|_{L^{\infty}(\R^2\times\R^2)}\notag\\
&\qquad\qquad+\sum_{\ell=1}^kN^{(2\dl-1)\ell}\|K_0\|_{L^{\infty}(\R^2\times\R^2)}^{k-\ell}.
\label{bound-gamma1}
\end{align}
Now, with the definition of $K_0$, we proceed as in \eqref{K11} to get the rough bound
\begin{align*}
&\|K_0\|_{L^{\infty}(\R^2\times\R^2)}\\
&\les \big\|(\kk_j^\star\chi_j)(x)a_{j,-2}(x,D)(\kk_j^\star\widetilde{\chi}_j)(x)\psi_0^2(-p_{j,2})(x,N^{-1}D)\big\|_{H^{-1-\dl}(\R^2)\to H^{1+\dl}(\R^2)}\\
&\les\|\psi_0^2(-p_{j,2})(x,N^{-1}D)\|_{H^{-1-\dl}(\R^2)\to H^{-1+\dl}(\R^2)} \les N^{2\dl},
\end{align*}
so that with our choice for $\dl$, the second term in the right-hand side of \eqref{bound-gamma1} is $O(N^{-\dl'})$ for $\dl'=1-2k\dl>0$. We are then left with estimating 
\begin{align*}
\Big\| a_{j_1,-\eps}(y,D)\Big((\Id\otimes\zeta_{j,j_1})^{\star}\big(\kk_j^{\star}(\chi_j\chi_{j_1})K_0\big)\Big)^k\Big\|_{L^{\infty}(\R^2\times\R^2)}.
\end{align*} 

First, to deal with $\zeta_{j,j_1}$, since the symbol class $\S^m$ in \eqref{def-symbolclass} is invariant by diffeomorphisms for any $m\in\R$ (see e.g. Theorem 18.1.17 in \cite{HormanderIII}), we can then write 
\begin{align*}
a_{j_1,-\eps}(y,D)\Big((\Id\otimes\zeta_{j,j_1})^{\star}\big(\kk_j^{\star}(\chi_j\chi_{j_1})K_0\big)\Big)^k&=(\Id\otimes\zeta_{j,j_1})^\star\Big(\wt a_{j_1,-\eps}(y,D) \big(\kk_j^{\star}(\chi_j\chi_{j_1})K_0\big)^k\Big)
\end{align*}
for some $\wt a_{j_1,-\eps}\in \S^{-\eps}$.

Next, we compute the symbol $c_0(x,\xi)$ of \eqref{K0} as
\begin{align}
c_0(x,\xi)=(\kk_j^\star\chi_j)(x)\int_{\R^2}\int_{\R^2}&e^{-ix_1\cdot\xi_1}a_{j,-2}(x,\xi+\xi_1)(\kk_j^\star\wt\chi_j)(x+x_1)\notag\\
&\qquad\times\psi_0^2(-p_{j,2}(x+x_1,N^{-1}\xi))d\xi_1dx_1.
\label{def-c0}
\end{align}
First, since $p_{j,2}$ and $(\kk_j^\star\wt\chi_j)$ are smooth in $x_1$ with bounded derivatives, we can integrate by parts in $x_1$ to get enough decay in $\xi_1$. Using that $a_{j,-2}$ is a $\Psi$DO of order $-2$, which therefore satisfies the bound \eqref{def-symbolclass} with $m=-2$, this gives for any $\al \in \N$
\begin{align*}
|c_0(x,\xi)|&\les(\kk_j^\star\chi_j)(x)\int_{\R^2}\int_{\R^2}\jb{\xi_1}^{-2\al}\jb{\xi+\xi_1}^{-2}\\
&\qquad\times\big|(1-\partial_{x_1}^2)^{\al}\big((\kk_j^\star\wt\chi_j)(x+x_1)\psi_0^2(-p_{j,2}(x+x_1,N^{-1}\xi))\big)\big|d\xi_1dx_1.
\end{align*}
Now, when a derivative in $x_1$ hits $\psi_0^2(-p_{j,2}(x+x_1,N^{-1}\xi))$, we pick up a term $(1-\partial_{x_1}^2)^\al p_{j,2}(x+x_1,N^{-1}\xi)=O(\jb{N^{-1}\xi}^2)$ by \eqref{def-symbolclass} which, due to the localization $|\xi|\les N$ on the support of $\psi_0^2(-p_{j,2}(x+x_1,N^{-1}\xi))$, is then bounded uniformly in $N$ . Thus we see that the term on the second line above can be bounded by
\begin{align*}
\mathbf{1}(x+x_1\in\supp(\kk_j^\star\wt\chi_j))\mathbf{1}(|\xi|\les N).
\end{align*} 
Then we can take $\al$ large enough to ensure that the integral in $\xi_1$ converges, so that we arrive at
\begin{align}
|c_0(x,\xi)|
&\les (\kk_j^\star\chi_j)(x)\mathbf{1}\big(|\xi|\les N\big)\int_{\R^2}\jb{\xi_1}^{-2\alpha}\jb{\xi+\xi_1}^{-2}d\xi_1\notag\\
&\les  (\kk_j^\star\chi_j)(x)\mathbf{1}\big(|\xi|\les N\big)\jb{\xi}^{-2}.
\label{c0}
\end{align}

Now, the kernel $K_0$ is related to the symbol $c_0(x,\xi)$ via the formula, 
\begin{align*}
K_0(x,y)=\int_{\R^2}e^{i(x-y)\cdot\xi}c_0(x,\xi)d\xi = \F_{\xi}^{-1}(c_0)(x,x-y),
\end{align*}
where $\F_{\xi}^{-1}$ means the inverse Fourier transform in the $\xi$ variable. This means that $\Big(\big(\kk_j^\star(\chi_j\chi_{j_1})\big)(y)K_0\Big)^k$ can be seen as 
\begin{align*}
\Big(\big(\kk_j^\star(\chi_j\chi_{j_1})\big)(y)K_0\Big)^k = \big(\kk_j^\star(\chi_j\chi_{j_1})\big)^k(y)\F_{\xi}^{-1}\big(c_0 *_{\xi}^k\big)(x,x-y),
\end{align*}
where $*_{\xi}^k$ stands for the iterated convolution in the $\xi$ variable:
\begin{align*}
(c_0*_{\xi}^k)(x,\xi_0) = \int_{\xi_0=\xi_1+\cdots+\xi_k}\prod_{j=1}^k c_0(x,\xi_j)d\xi_j.
\end{align*} 
Next, using that $\wt a_{j_1,-\eps}\in\S^{-\eps}(\R^2\times\R^2)$, we have for any $\xi,\xi_1\in\R^2$, 
\begin{align*}|\wt a_{j_1,-\eps}(y,\xi)|\les \jb{\xi}^{-\eps}\les \jb{\xi_1}^{-\eps}\jb{\xi-\xi_1}^{\eps},
\end{align*} 
and since $\kk_j^\star(\chi_j\chi_{j_1})\in C^{\infty}_0(\R^2)$, we can compute
\begin{align*}
&\Big\|\wt a_{j_1,-\eps}(y,D)\Big(\big(\kk_j^\star(\chi_j\chi_{j_1})\big)(y)K_0\Big)^k\Big\|_{L^{\infty}(\R^2\times\R^2)}\\&= \Big\|\int_{\R^2}e^{iy\cdot \xi}\wt a_{j_1,-\eps}(y,\xi)\int_{\R^2}\widehat{\big(\kk_j^\star(\chi_j\chi_{j_1})\big)^k}(\xi-\xi_1) e^{-ix\cdot\xi_1}(c_0*_{\xi}^k)(x,-\xi_1)d\xi_1d\xi\Big\|_{L^{\infty}(\R^2\times\R^2)}\\
&\les \sup_{x\in \supp(\kk_j^\star\chi_j)}\sup_{y\in \supp\kk_j^\star(\chi_j\chi_{j_1})} \int_{\R^2}\int_{\R^2}\jb{\xi_1}^{-\eps}\jb{\xi-\xi_1}^{\eps}\cdot\jb{\xi-\xi_1}^{-10}\big|(c_0*_{\xi}^k)(x,-\xi_1)\big|d\xi_1d\xi\\
&\les \sup_{x\in \supp(\kk_j^\star\chi_j)} \int_{\R^2}\jb{\xi_1}^{-\eps}\big|(c_0*_{\xi}^k)(x,-\xi_1)\big|d\xi_1.
\end{align*}
Thus, expanding the iterated convolution above and using the triangle inequality with the bound \eqref{c0}, we get the estimate
\begin{align}
\Big\|\wt a_{j_1,-\eps}(y,D)\Big(\big(\kk_j^\star(\chi_j\chi_{j_1})\big)(y)K_0\Big)^k\Big\|_{L^{\infty}(\R^2\times\R^2)}\les  \int_{\Gamma_{k,N}}\jb{\xi_1+\cdots+\xi_k}^{-\eps}\prod_{\ell=1}^k\jb{\xi_\ell}^{-2}d\xi_\ell,
\label{integral}
\end{align}
where 
\begin{align*}
\Gamma_{k,N}=\big\{(\xi_1,...,\xi_k)\in (\R^2)^k,~|\xi_\ell|\les N,~\ell=1,...,k\big\}.
\end{align*}
 So it remains to bound the integral in \eqref{integral}, uniformly in $N$. By symmetry in $\xi_1,...,\xi_k$, it is enough to bound the contribution of 
\begin{align*}
\widetilde{\Gamma}_{k,N}=\{(\xi_1,...,\xi_k)\in\Gamma_{k,N},~|\xi_k|\geq \cdots\geq |\xi_1|\}.
\end{align*}

First, to estimate the integral in $\xi_k$, if $\jb{\xi_1+\cdots+\xi_k}\geq \jb{\xi_k}$ then we have
\begin{align*}
\int_{|\xi_k|\geq |\xi_{k-1}|}\jb{\xi_1+\cdots+\xi_k}^{-\varepsilon}\jb{\xi_k}^{-2}d\xi_k&\les \int_{|\xi_k|\geq |\xi_{k-1}|}\jb{\xi_k}^{-2-\varepsilon}d\xi_k\les \jb{\xi_{k-1}}^{-\varepsilon}.
\end{align*}
On the other hand, in the case $\jb{\xi_1+\cdots+\xi_k}\leq \jb{\xi_k}$ we have
\begin{align*}
\int_{|\xi_k|\geq |\xi_{k-1}|}\jb{\xi_1+\cdots+\xi_k}^{-\varepsilon}\jb{\xi_k}^{-2}d\xi_k&\les \jb{\xi_{k-1}}^{-\frac{\varepsilon}2}\int_{\R^2}\jb{\xi_1+\cdots+\xi_k}^{-2-\frac{\varepsilon}2}d\xi_k\les \jb{\xi_{k-1}}^{-\frac{\varepsilon}2}.
\end{align*}
Hence we end up with the bound
\begin{align*}
\int_{\widetilde{\Gamma}_{k-1,N}}\jb{\xi_{k-1}}^{-2-\frac{\varepsilon}2}\prod_{\ell=1}^{k-2}\jb{\xi_\ell}^{-2}d\xi_\ell d\xi_{k-1},
\end{align*}
for which we can integrate successively in $|\xi_{k-1}|\geq |\xi_{k-2}|\geq \cdots\geq |\xi_1|$:
\begin{align*}
\int_{\widetilde{\Gamma}_{k-1,N}}\jb{\xi_{k-1}}^{-2-\frac{\varepsilon}2}\prod_{\ell=1}^{k-2}\jb{\xi_\ell}^{-2}d\xi_\ell d\xi_{k-1}&\les \int_{\widetilde{\Gamma}_{k-2,N}}\jb{\xi_{k-2}}^{-2-\frac{\varepsilon}2}\prod_{\ell=1}^{k-3}\jb{\xi_\ell}^{-2}d\xi_\ell d\xi_{k-2}\\
&\les \cdots \les \int_{\R^2}\mathbf{1}(|\xi_1|\les N)\jb{\xi_1}^{-2-\frac\eps2}d\xi_1 \leq C <\infty
\end{align*}
uniformly in $N$. This proves \eqref{estim-Green}.

For  \eqref{estim-Green-diff}, we can decompose locally $\gamma_{N_1}=K_{0,N_1}+N_1^{-1}K_{1,N_1}$ and $\gamma_{N_2}=K_{0,N_2}+N_2^{-1}K_{1,N_2}$ similarly as above, and following the computations we end up with estimating
\begin{align*}
\|\wt a_{j_1,-\eps}(y,D)\big(K_{0,N_1}^k-K_{0,N_2}^k\big)\|_{L^{\infty}(\R^2\times\R^2)},
\end{align*}
which follows as before except that we notice that the corresponding symbols satisfy
\begin{align*}
&\big(c_{0,N_1}*_{\xi}^k\big)(x,\xi_0)-\big(c_{0,N_2}*_{\xi}^k\big)(x,\xi_0)\\& = \int_{\xi_0=\xi_1+\cdots+\xi_k}\Big(\prod_{\ell=1}^k c_{0,N_1}(x,\xi_\ell)-\prod_{\ell=1}^k c_{0,N_2}(x,\xi_\ell)\Big)\prod_{\ell=1}^kd\xi_\ell.
\end{align*}
In view of \eqref{def-c0}, for the integral above to be non-zero, this requires at least one of the $\xi_\ell$ to be in the region $N_1\les |\xi_\ell|\les N_2$; otherwise, in the case all $|\xi_\ell|\ll N_1$ we have both $\psi_0^2(-p_{j,2}(x,N_1^{-1}\xi_\ell))=1=\psi_0^2(-p_{j,2}(x,N_2^{-1}\xi_\ell))$ and we see that $\prod_{\ell=1}^k c_{0,N_1}(x,\xi_\ell)=\prod_{\ell=1}^k c_{0,N_2}(x,\xi_\ell)$. For $N_1\les |\xi_\ell|\les N_2$, we can then replace the factor $\jb{\xi_\ell}^{-2-\frac{\eps}2}$ in the corresponding integral by $N_1^{-\frac{\eps}4}\jb{\xi_\ell}^{-2-\frac{\eps}4}$ and finish integrating as above. The estimate \eqref{estim-Green-diff2} follows from the same argument, replacing $N_1\les |\xi_\ell|\les N_2$ by $|\xi_\ell|\sim N$.

\end{proof}
\begin{remark}\label{rk-FIO}\rm
In Proposition~\ref{prop-z}, we only estimated the higher Wick powers $:z^{\ell}:$, $\ell \geq 2$, in $L^q([0,T];W^{-\eps,\infty}(\M))$ and did not show the continuity in time for these objects. Though we would only need a very rough bound in space (just to get a power of $h$ as in \eqref{kolmogorov}), the global argument as the one we used for $z$ does not seem to apply since we would need to estimate a product of $k$ eigenfunctions $\varphi_{n_1}\dots \varphi_{n_k}$, for which it is not clear if there is an ``off-diagonal decay" allowing to sum on $n_1,...,n_k$ even after regularizing the product. On the other hand, a local argument as in Proposition \ref{prop-Green} also fails since contrary to the truncation operator $\psi_{N^2}(-\Dlg)$, the wave operator $\cos(h\sqrt{1-\Dlg})$ for the linear wave equations does not belong to the usual symbol class $\S^0$ defined in \eqref{def-symbolclass}. However, we might be able to overcome this difficulty by replacing the local description of $\gamma_N$ in terms of $\Psi$DO by a local description of $\gamma_N(t+h,t)$ in terms of Fourier integral operators by following the construction in e.g. \cite{BGT,MSS}. We chose not to pursue this point further since our proof of well-posedness only requires the Wick powers to be controlled in $L^q([0,T];W^{-\eps,\infty}(\M))$ for some large but finite $p,q\in [1,\infty)$.
\end{remark}

Next, we prove a similar statement as in Proposition~\ref{prop-z} but for the solution $\P_N\Psi_\damp$ to truncated linear stochastic damped wave equations 
\begin{align}\label{truncLSDNLW}
d\begin{pmatrix}
u_N\\v_N
\end{pmatrix} = \begin{pmatrix}
0& 1\\\Dlg-1&0
\end{pmatrix}\begin{pmatrix}
u_N\\v_N
\end{pmatrix}dt + \begin{pmatrix}
0\\- v_Ndt+\sqrt{2}\P_NdB
\end{pmatrix}
\end{align} 
with data given by $(u_N,v_N)\big|_{t=0}=\P_N(u_0,u_1)\sim (\P_N)_{\star}\mu$. Recall that $\Psi_\damp=\Psi_\damp(u_0,u_1,\o)$ is the random variable on $\H^s(\M)\times\O$ defined in \eqref{def-Psidamp}.
\begin{proposition}\label{prop-psidamp}
$(\P_N)_{\star}\mu$ is invariant under \eqref{truncLSDNLW}, in the sense that for any continuous and bounded test function $F\in C_b(\H^s(\M);\R)$ and any $t\ge 0$,
\begin{align*}
&\int_{\H^s(\M)}\int_\O F\Big[\big(\P_N\Psi_\damp(u_0,u_1,\o),\dt\P_N\Psi_\damp(u_0,u_1,\o)\big)\Big]d\Prob(\o)d\mu(u_0,u_1)\\
&\qquad\qquad = \int_{\H^s(\M)}F\big[\big(\P_Nu_0,\P_Nu_1\big)\big]d\mu(u_0,u_1).
\end{align*}
Moreover for any $k\in\N$, $T>0$, $0<\eps\ll1$ and $1\leq p,q <\infty$ then $\big\{H_k\big(\P_N\Psi_{\damp}(u_0,u_1,\o);\s_N(x)\big)\big\}_{N\in\N}$ is a Cauchy sequence in 
\begin{align*}
L^p\big(\mu\otimes\Prob;L^q([0,T];W^{-\eps,\infty}(\M))\big)
\end{align*} and converges almost surely to a limit $\,:\!\Psi_{\damp}^k\!:\,\in L^q([0,T];W^{-\eps,\infty}(\M))$. Moreover $H_k\big(\P_N\Psi_{\damp}(u_0,u_1,\o);\s_N(x)\big)$ and $\,:\!\Psi_{\damp}^k\!:\,$ obey the tail estimates \eqref{tail-z} and \eqref{tail-z-zN}, and we also have $\Psi_{\damp}\in C\big([0,T];W^{-\eps,\infty}(\M)\big)\cap C^1\big([0,T];W^{-\eps-1,\infty}(\M)\big)$ almost surely, as well as the tail estimate
\begin{align}\label{tail-sup}
\mu\otimes\Prob\big(\|(\Psi_{\damp},\dt\Psi_\damp)\|_{C([0,T];\H^{-\eps})}>R\big)\leq Ce^{-cR^2}.
\end{align}
Lastly, $\mu$ is invariant under $(u_0,u_1)\mapsto (\Psi_\damp,\dt\Psi_\damp)$, in the same sense as above.
\end{proposition}
\begin{remark}\label{rk-SQE}\rm
Note that in the case of the stochastic quantization equation \eqref{SQE} treated in \cite{DPD}, the truncated stochastic convolution
\begin{align*}
\mathfrak{z}_N(t) = \P_N\int_{-\infty}^t e^{(t-t')(\Dlg-1)}dB(t')
\end{align*} has the same covariance function $\gamma_N$ as for $z_N$ and $\P_N\Psi_{\damp}$, so we can use the same argument as in Propositions~\ref{prop-z} and \ref{prop-psidamp} to estimate the Wick powers of $\mathfrak{z}$. In turn this would generalize the result of Da Prato and Debussche \cite{DPD} to the case of a general compact boundaryless Riemannian surface, which to the authors knowledge would be new.
\end{remark}
\begin{proof}[Proof of Proposition \ref{prop-psidamp}.]
We only prove the first assertion, since the rest of the proposition follows from the same analysis as for Proposition~\ref{prop-z}. Namely, once we have the invariance of $(\P_N)_{\star}\mu$, we know that $\P_N\Psi_{\damp}$ has the same (spatial) covariance function $\gamma_N$ as $\P_NS(t)(u_0,u_1)$, so we can write
\begin{align*}
&\int_{\H^s(\M)}\int_\O\big|(1-\Dlg)^{-\frac\eps2}H_k\big(\P_N\Psi_{\damp}(u_0,u_1,\o)(t,\x);\s_N(\x)\big)\big|^2d\Prob(\o) d\mu(u_0,u_1)\\
& =k!\big[(1-\Dlg)_{\x_1}^{-\frac\eps2}(1-\Dlg)_{\x_2}^{-\frac\eps2}\big(\gamma_N(\x_1,\x_2)^k\big)\big]\big|_{\x_1=\x_2=\x},
\end{align*}
where $\gamma_N$ is the same as in Lemma \ref{lem-covar}, and the same computations as in the proof of Proposition \ref{prop-z} apply.

Proving the invariance of $\mu_N=(\P_N)_{\star}\mu$ is equivalent to showing $\L_N^\#\mu_N=0$, where $\L_N$ is the infinitesimal generator of \eqref{truncLSDNLW} and $\L_N^\#$ is its dual acting on probability measures on $E_N\times E_N$ by
\begin{equation*}
\forall F\in C^{\infty}_b(E_N\times E_N;\R),~\int_{E_N\times E_N}F(u,v)d(\L_N^\#\mu_N) = \int_{E_N\times E_N}(\L_NF)(u,v)d\mu_N(u,v).
\end{equation*}
But in view of \eqref{truncLSDNLW}, we have $\L_N = \L_N^1+\L_N^2$, where $\L_N^1$ is the generator for the linear wave equations, and $\L_N^2$ the one for an Ornstein-Uhlenbeck process. More precisely, \eqref{truncLSDNLW} can be seen as a system of SDEs in $\R^{2\Ld_N}$, where $\Ld_N=\dim E_N$, given by
\begin{align*}
\begin{cases}
d a_n = b_n dt\\
d b_n = -\jb{\lambda_n}^2a_ndt +(-\ b_ndt+\sqrt{2}\psi_0(N^{-2}\lambda_n^2)d\beta_n)
\end{cases},~n=0,...,\Ld_N-1,
\end{align*}
whose infinitesimal generator is given by
\begin{align*}
\L_Nf(a_0,...,a_{\Ld_N-1},b_0,...,b_{\Ld_N-1}) = \sum_{n=0}^{\Ld_N-1}b_n\partial_{a_n}f-\jb{\lambda_n}^2a_n\partial_{b_n}f - b_n\partial_{b_n}f+\psi_0(N^{-2}\lambda_n^2)^2\partial_{b_n}^2f.
\end{align*}
Now if we set 
\begin{align*}
\L_N^2f = \sum_{n=0}^{\Ld_N-1}-b_n\partial_{b_n}f+\psi_0(N^{-2}\lambda_n^2)^2\partial_{b_n}^2f
\end{align*}
we recognize the generator of the Ornstein-Uhlenbeck process
\begin{align*}
\begin{cases}
a_n(t)=a_n(0),\\b_n(t)=e^{-t}b_n(0)+\sqrt{2}\psi_0(N^{-2}\lambda_n^2)\int_0^te^{-(t-t')}d\beta_n(t'),
\end{cases}
\end{align*}
and a straightforward computation using It\^o's isometry gives that $b_n$ is a mean 0 Gaussian random variable with variance
\begin{align*}
\E(b_n(t)^2)=e^{-2t}\E(b_n(0)^2)+2\psi_0(N^{-2}\lambda_n^2)^2\frac{1-e^{-2 t}}2.
\end{align*}
In particular, in view of \eqref{def-mu}, $\E(b_n(t)^2)= \psi_0(N^{-2}\lambda_n^2)^2 = \E(b_n(0)^2) $, which means that $\L_N^2$ preserves $\mu_N$. On the other hand, we have
\begin{align*}
\L_N^1 = \sum_{n=0}^{\Ld_N-1}b_n\partial_{a_n}-\jb{\lambda_n}^2a_n\partial_{b_n},
\end{align*} 
which is the generator of the truncated linear wave equations seen as the Hamiltonian system of ODEs
\begin{align*}
\begin{cases}
\frac{d}{dt}a_n = b_n,\\
\frac{d}{dt}b_n = -\jb{\lambda_n}^2a_n,
\end{cases}~n=0,...,\Ld_N-1.
\end{align*}
Now the energy of this system
\begin{align*}\En_{0,N}(a_0,...,a_{\Ld_N-1},b_0,...,b_{\Ld_N-1}) = \frac12\sum_{n=0}^{\Ld_N-1}\big(\jb{\lambda_n}^2a_n^2+b_n^2\big)\end{align*}
is conserved, and by Liouville's theorem, this system preserves the Lebesgue measure $\prod_{n=0}^{\Ld_N-1}da_ndb_n$, so we see that the measure ${\displaystyle e^{-\En_{0,N}(a_0,...,a_{\Ld_N-1},b_0,...,b_{\Ld_N-1})}\prod_{n=0}^{\Ld_N-1}da_ndb_n}$ is also conserved, which is nothing else than the conservation of $\mu_N$ in view of \eqref{def-mu}. All in all, $\L_N^\#\mu_N=0$ which concludes the proof of the invariance. 

The invariance of $\mu$ for $(\Psi,\dt \Psi)$ then follows from the invariance of $(\P_N)_\star\mu$ for $(\P_N\Psi,\P_N\dt\Psi)$ along with the almost sure convergence of $(\P_N\Psi,\P_N\dt\Psi)(t)$ towards $(\Psi,\dt\Psi)(t)$ in $\H^s(\M)$ for any $t\ge 0$ and the weak convergence of $(\P_N)_\star\mu$ towards $\mu$ (which is clear from the convergence almost surely and in $L^p(\O_0\times\O_1;\H^s(\M))$ for any $p\ge 1$ of the series in \eqref{def-mu}).

Finally, in order to show the last tail estimate \eqref{tail-sup}, in view of \eqref{def-Psidamp} we can first separate
\begin{align*}
&\mu\otimes\Prob\big(\|(\Psi_{\damp},\dt\Psi_\damp)\|_{C([0,T];\H^{-\eps})}>R\big)\\
&\le \mu\big(\sup_{t\le T}\|\dt V(t)u_0 + V(t)(u_0+u_1)\|_{H^{-\eps}}\gtrsim R\big)\\
&\qquad+\mu\big(\sup_{t\le T}\|\dt^2 V(t)u_0 + \dt V(t)(u_0+u_1)\|_{H^{-1-\eps}}\gtrsim R\big)\\
&\qquad+\Prob\Big(\sup_{t\le T}\Big\|\int_0^tV(t-t')dB(t')\Big\|_{H^{-\eps}}\gtrsim R\Big)\\
&\qquad+\Prob\Big(\sup_{t\le T}\Big\|\int_0^t\dt V(t-t')dB(t')\Big\|_{H^{-1-\eps}}\gtrsim R\Big)\\
&= \I+\II+\III+\IV.
\end{align*}
We begin by estimating $\I$. Using Chebyshev's inequality, the boundedness of $\dt^j V(t): H^s(\M)\to H^{s+j-1}(\M)$, for any $s\in\R$ and $j\ge 0$, and the Wiener chaos estimate \eqref{Wiener-chaos} with the fact that $(u_0,u_1)$ is Gaussian, we get a constant $C>0$ such that we can bound for any $T,R,\eps>0$ and $p\ge 1$
\begin{align*}
\I & \les R^{-p}\E\big[\big(\sup_{t\le T}\|\dt V(t)u_0 + V(t)(u_0+u_1)\|_{H^{-\eps}}\big)^p\big] \les R^{-p} \E\|(u_0,u_1)\|_{\H^{-\eps}}^p \\
&\les R^{-p}(p-1)^{\frac{p}2}\big(\E\|(u_0,u_1)\|_{\H^{-\eps}}^2\big)^{\frac{p}2} \le C^p (p-1)^{\frac{p}2}R^{-p}.
\end{align*}
Optimizing in $p$ finally leads to
\begin{align*}
\mu\big(\sup_{t\le T}\|\dt V(t)u_0 + V(t)(u_0+u_1)\|_{H^{-\eps}}\gtrsim R\big) \les e^{-cR^2}
\end{align*}
for some $c>0$ independent of $T$ and $R$. The estimate on $\II$ is similar. As for $\III$, we first use Doob's martingale inequality (see e.g. Theorem 3.9 in \cite{DPZ}) to bound
\begin{align*}
\III &\les R^{-p}\sup_{t\le T}\E\Big[\Big\|\int_0^tV(t-t')dB(t')\Big\|_{H^{-\eps}}^p\Big]
\end{align*}
and then conclude as above since $\int_0^tV(t-t')dB(t')$ is Gaussian. The same argument applies to $\IV$, which finally leads to \eqref{tail-sup}.
\end{proof}
\begin{remark}\rm
Note that the proof of the invariance of $(\P_N)_{\star}\mu$ above works equally well for $(\Pi_N)_{\star}\mu$. Of course, the estimates on the Wick powers require the smooth cut-off $\P_N$ instead of the sharp cut-off $\Pi_N$.
\end{remark}

\subsection{Estimate on the stochastic convolution}
As for the nonlinear wave equations with random initial data, the key point in the analysis of the stochastic nonlinear wave equations \eqref{SNLW} is the following proposition. Let us recall here that the (truncated) stochastic convolution (solution of the linear stochastic wave equation) is defined by
\begin{align*}
\Psi_N(t,x)=\P_N\int_0^t\frac{\sin\big((t-t')\sqrt{1-\Dlg}\big)}{\sqrt{1-\Dlg}}dB(t')
\end{align*}
and the cylindrical Wiener process $B$ is defined in \eqref{STWN}. The corresponding renormalization is given in \eqref{def-WickPsi}.
\begin{proposition}\label{prop-psi}
Let $0<\eps\ll 1$, $k \in\N$, $T \in (0;1]$ and $p,q \in [1,\infty)$. Then, $\big\{H_k\big(\P_N\Psi(\o);\s_N(t,\x)\big)\big\}_{N\in\N}$ is a Cauchy sequence in $L^p\big(\O;L^q\big([0,T];W^{-\eps,\infty}(\M)\big)\big)$. In particular, denoting the limit by $:\Psi^k\!:\,$, we also have that $H_k\big(\P_N\Psi(\o);\s_N(t,\x)\big)$ converges almost surely towards $ \,:\!\Psi^k\!:\,$ in $L^q\big([0,T];W^{-\eps,\infty}(\M)\big)$, and for $k=1$, we have that $\Psi$ belongs almost surely to $C([0,T];W^{-\eps,\infty}(\M))\cap C^1([0,T];W^{-\eps-1,\infty}(\M))$. Moreover $H_k\big(\P_N\Psi(\o);\s_N(t,\x)\big)$, $\,:\!\Psi^k\!:\,$, and $\Psi$ respectively obey the tail estimates \eqref{tail-z}, \eqref{tail-z-zN}, and \eqref{tail-sup}.
\end{proposition}
\begin{proof}
As before, we can compute for fixed $t\in [0,T]$ and $\x\in\M$
\begin{align*}
&\E\big[\big|(1-\Dlg)^{-\frac\eps2}H_k\big(\P_N\Psi(\o,t,\x);\s_N(t,\x)\big)\big|^2\big]\\&=\sum_{n,n'\geq 0}\frac{\varphi_n(\x)}{\jb{\lambda_n}^{\eps}}\frac{\varphi_{n'}(\x)}{\jb{\lambda_{n'}}^{\eps}}\int_{\M}\int_{\M}\E\big[H_k\big(\P_N\Psi(\o,t,\x_1);\s_N(t,\x_1)\big)\\
&\qquad\qquad\times H_k\big(\P_N\Psi(t,\x_2);\s_N(t,\x_2)\big)\big]\varphi_n(\x_1)\varphi_{n'}(\x_2)d\x_1d\x_2.
\end{align*}
Now we use \eqref{prod-wick}, hence
\begin{align*}
&\E\big[\big|(1-\Dlg)^{-\frac\eps2}H_k\big(\P_N\Psi(\o,t,\x);\s_N(t,\x)\big)\big|^2\big]\\&=k!\sum_{n,n'\geq 0}\frac{\varphi_n(\x)}{\jb{\lambda_n}^{\eps}}\frac{\varphi_{n'}(\x)}{\jb{\lambda_{n'}}^{\eps}}\int_{\M^2}\big[\gamma_N^t(\x_1,\x_2)\big]^kd\x_1d\x_2\\
&=k!\big((1-\Dlg)_{\x_1}^{-\frac\eps2}(1-\Dlg)_{\x_2}^{-\frac\eps2}\big[\gamma_N^t(\x_1,\x_2)\big]^k\big)\big|_{\x_1=\x_2=\x},
\end{align*}
where we define
\begin{equation*}
\begin{aligned}
\gamma_N^t(\x_1,\x_2) &\overset{\text{def}}{=} \E\big[\P_N\Psi(\o,t,\x_1)\cdot\P_N\Psi(\o,t,\x_2)\big]\\&= \sum_{n\geq 0}\psi_0^2(N^{-2}\lambda_n^2)\bigg(\int_0^t\bigg[\frac{\sin\big((t-t')\jb{\lambda_n}\big)}{\jb{\lambda_n}}\bigg]^2dt'\bigg)\varphi_n(\x_1)\varphi_n(\x_2),
\end{aligned}
\end{equation*}
the last equality resulting from It\^o's isometry. In particular, in view of the second line in \eqref{def-sigmat}, we see that $\gamma_N^t$ can be decomposed as
\begin{align*}
\gamma_N^t = \frac{t}{2}\gamma_N + \widetilde{\gamma}_N^t,
\end{align*}
where $\gamma_N$ is given in \eqref{def-gamma}, and
\begin{align*}
\widetilde{\gamma}_N^t(\x_1,\x_2) = -\sum_{n\geq 0}\psi_0^2(N^{-2}\lambda_n^2)\frac{\sin(2t\jb{\lambda_n})}{4\jb{\lambda_n}^3}\varphi_n(\x_1)\varphi_n(\x_2).
\end{align*}

Hence, using the product estimate of Corollary \ref{cor-besov} (iii), we get for any $t\in [0,T]$
\begin{align*}
\|(\gamma_N^t)^k\|_{B^{-\frac{\eps}2}_{\infty,\infty}(\M)\times B^{-\frac{\eps}2}_{\infty,\infty}(\M)} &\les_T \sum_{\ell=0}^k \|\gamma_N^{\ell}(\widetilde{\gamma}_N^t)^{k-\ell}\|_{B^{-\frac{\eps}2}_{\infty,\infty}(\M)\times B^{-\frac{\eps}2}_{\infty,\infty}(\M)}\\& \les \sum_{\ell=0}^k \|\gamma_N^{\ell}\|_{B^{-\frac{\eps}2}_{\infty,\infty}(\M)\times B^{-\frac{\eps}2}_{\infty,\infty}(\M)}\|\widetilde{\gamma}_N^t\|_{B^{\eps}_{\infty,\infty}(\M)\times B^{\eps}_{\infty,\infty}(\M)}^{k-\ell}. 
\end{align*}
Now from Proposition \ref{prop-Green}, we have that $\|\gamma_N^{\ell}\|_{W^{-\frac\eps2,-\frac\eps2,\infty}}$ is bounded uniformly in $N$ for any $\eps>0$. As for the other term, we can estimate it directly with Cauchy-Schwarz inequality and Lemma \ref{LEM:ev}: 
\begin{align*}
&\|\widetilde{\gamma}_N^t\|_{B^{\eps}_{\infty,\infty}(\M)\times B^{\eps}_{\infty,\infty}(\M)}\\
&= \sup_{M_1,M_2\in 2^{\Z_+}}M_1^{\eps}M_2^{\eps}\sup_{\x_1,\x_2\in\M}\Big|\sum_{n\geq 0}\psi_0(N^{-2}\lambda_n^2)\psi_{M_1^2}(\lambda_n^2)\psi_{M_2^2}(\lambda_n^2)\frac{\sin(2t\jb{\lambda_n})}{4\jb{\lambda_n}^{3}}\varphi_n(\x_1)\varphi_n(\x_2)\Big|\\
&\les \sup_{M_1,M_2\in 2^{\Z_+}}M_1^{\eps}M_2^{\eps}\sup_{\x_1,\x_2\in\M} \Big(\sum_{n\geq 0}\psi_0(N^{-2}\lambda_n^2)^2\psi_{M_1^2}(\lambda_n^2)\psi_{M_2^2}(\lambda_n^2)\frac{\sin(2t\jb{\lambda_n})^2}{\jb{\lambda_n}^{3}}\varphi_n(\x_1)^2\Big)^{\frac12}\\
&\qquad\qquad\qquad\qquad\qquad\qquad\qquad\times\Big(\sum_{n\geq 0}\psi_{M_1^2}(\lambda_n^2)\psi_{M_2^2}(\lambda_n^2)\frac{1}{\jb{\lambda_n}^{3}}\varphi_n(\x_2)^2\Big)^{\frac12}\\
&\les \sup_{M_1\sim M_2\les N}M_1^{2\eps}\sum_{n\geq 0}\psi_{M_1^2}(\lambda_n^2)\psi_{M_2^2}(\lambda_n^2)\frac{1}{\jb{\lambda_n}^{3}}\\
&\les \sup_{M_1\les N}M_1^{2\eps-1}\le C <+\infty
\end{align*}
uniformly in $N\in\N$.

 Thus we can conclude as in the proof of Proposition~\ref{prop-z} that ${\displaystyle \E\big|(1-\Dlg)^{-\eps}\,:\!\Psi_N^k(t,x)\!:\,\big|^2}$ is uniformly bounded in $N$, from which we get a uniform bound in ${\displaystyle L^p\big(\O;L^q([0,T];W^{-\eps,\infty}(\M))\big)}$ for any $1\leq p,q<\infty$. 
 
 As for the convergence of $H_k\big(\P_N\Psi(\o,t,\x);\s_N(t,\x)\big)$, we have again for $N_1< N_2$
\begin{align*}
& \E\Big|(1-\Dlg)^{-\frac\eps2}\Big[H_k\big(\P_{N_1}\Psi(\o,t,\x);\s_{N_1}(t,\x)\big)-H_k\big(\P_{N_2}\Psi(\o,t,\x);\s_{N_2}(t,\x)\big)\Big]\Big|^2\\
 &=k!(1-\Dlg)_{\x_1}^{-\frac\eps2}(1-\Dlg)_{\x_2}^{-\frac\eps2}\Big[\big(\P_{N_1}^2\gamma^t(\x_1,\x_2)\big)^k\\
 &\qquad\qquad-2\big(\P_{N_1}\P_{N_2}\gamma^t(\x_1,\x_2)\big)^k+\big(\P_{N_2}^2\gamma^t(\x_1,\x_2)\big)^k\Big]_{\big|\x_1=\x_2=\x}\\
 &\les \Big\|\big(\P_{N_1}^2\gamma^t(\x_1,\x_2)\big)^k-2\big(\P_{N_1}\P_{N_2}\gamma^t(\x_1,\x_2)\big)^k+\big(\P_{N_2}^2\gamma^t(\x_1,\x_2)\big)^k\Big\|_{B^{-\frac{\eps}2}_{\infty,\infty}(\M)\times B^{-\frac{\eps}2}_{\infty,\infty}(\M)}.
\end{align*}
Writing as above $\P_{N_2}^2\gamma^t = \frac{t}2\P_{N_2}^2\gamma + \P_{N_2}^2\wt\gamma^t$, we then estimate for $t\in [0,T]$ the contribution of
\begin{align*}
&\Big\|\big(\P_{N_2}^2\gamma^t(\x_1,\x_2)\big)^k-\big(\P_{N_1}\P_{N_2}\gamma^t(\x_1,\x_2)\big)^k\Big\|_{B^{-\frac{\eps}2}_{\infty,\infty}(\M)\times B^{-\frac{\eps}2}_{\infty,\infty}(\M)}\\
&\les \sum_{\ell=0}^k \Big\|\big(\P_{N_2}^2\gamma(\x_1,\x_2)\big)^{\ell}\big(\P_{N_2}^2\wt\gamma^t(\x_1,\x_2)\big)^{k-\ell}\\
&-\big(\P_{N_1}\P_{N_2}\gamma(\x_1,\x_2)\big)^{\ell}\big(\P_{N_1}\P_{N_2}\wt\gamma^t(\x_1,\x_2)\big)^{k-\ell}\Big\|_{B^{-\frac{\eps}2}_{\infty,\infty}(\M)\times B^{-\frac{\eps}2}_{\infty,\infty}(\M)}\\
&\les \sup_{\ell\le k} \Big\|\Big[\big(\P_{N_2}^2\gamma(\x_1,\x_2)\big)^{\ell}-\big(\P_{N_1}\P_{N_2}\gamma(\x_1,\x_2)\big)^{\ell}\Big]\big(\P_{N_2}^2\wt\gamma^t(\x_1,\x_2)\big)^{k-\ell}\Big\|_{B^{-\frac{\eps}2}_{\infty,\infty}(\M)\times B^{-\frac{\eps}2}_{\infty,\infty}(\M)}\\
&+\Big\|\big(\P_{N_1}\P_{N_2}\gamma(\x_1,\x_2)\big)^{\ell}\Big[\big(\P_{N_2}^2\wt\gamma^t(\x_1,\x_2)\big)^{k-\ell}-\big(\P_{N_1}\P_{N_2}\wt\gamma^t(\x_1,\x_2)\big)^{k-\ell}\Big]\Big\|_{B^{-\frac{\eps}2}_{\infty,\infty}(\M)\times B^{-\frac{\eps}2}_{\infty,\infty}(\M)}\\
&\les \sup_{\ell\le k}\big( \I+\II\big).
\end{align*} 
Using again the product estimate of Corollary \ref{cor-besov} (iii), we bound
\begin{align*}
\I&\les \Big\|\big(\P_{N_2}^2\gamma(\x_1,\x_2)\big)^{\ell}-\big(\P_{N_1}\P_{N_2}\gamma(\x_1,\x_2)\big)^{\ell}\Big\|_{B^{-\frac{\eps}2}_{\infty,\infty}(\M)\times B^{-\frac{\eps}2}_{\infty,\infty}(\M)}\\
&\qquad\qquad\times\Big\|\P_{N_2}^2\wt\gamma^t(\x_1,\x_2)\Big\|_{B^{\eps}_{\infty,\infty}(\M)\times B^{\eps}_{\infty,\infty}(\M)}^{k-\ell}\\
&\les N_1^{-\wt\eps},
\end{align*}
for some $0<\wt\eps\ll\eps$. This follows from \eqref{estim-Green-diff}-\eqref{estim-Green-diff2} similarly as in the proof of Proposition \ref{prop-z}, along with the previous bound on $\wt\gamma^t$. As for $\II$, we use also the product estimate to get
\begin{align*}
\II&\les \Big\|\big(\P_{N_1}\gamma(\x_1,\x_2)\big)^{\ell}\Big\|_{B^{-\frac{\eps}2}_{\infty,\infty}(\M)\times B^{-\frac{\eps}2}_{\infty,\infty}(\M)}\\
&\qquad\qquad\times\Big\|\big(\P_{N_2}^2\wt\gamma^t(\x_1,\x_2)\big)^{k-\ell}-\big(\P_{N_1}\P_{N_2}\wt\gamma^t(\x_1,\x_2)\big)^{k-\ell}\Big\|_{B^{\eps}_{\infty,\infty}(\M)\times B^{\eps}_{\infty,\infty}(\M)},
\end{align*}
and we can gain a small negative power of $N_1$ in the second term by proceeding as for the bound on $\wt\gamma^t$ above and using that the supremum of $M_1^{2\eps-1}$ now runs over $N_1\les M_1 \les N_2$. The second contribution $\big(\P_{N_1}^2\gamma^t(\x_1,\x_2)\big)^k-\big(\P_{N_1}\P_{N_2}\gamma^t(\x_1,\x_2)\big)^k$ is estimated similarly. This shows that $\big\{H_k\big(\P_N\Psi(\o,t,\x);\s_N(t,\x)\big)\big\}_{N\in\N}$ is a Cauchy sequence in $L^p(\O;L^q([0,T];W^{-\eps,\infty}(\M)))$ for any finite $p,q\ge 1$.

Let us finally turn to the continuity property of $\Psi_N$ and $\Psi$. As in the previous section, we compute for any $h,t\in [0,T]$ and $\x\in\M$
\begin{align*}
\E\big|(1-\Dlg)^{-\eps}(\th \Psi-\Psi)(\o,t,\x)\big|^2 &= \sum_{n\geq 0}\frac{\varphi_n(\x)^2}{\jb{\lambda_n}^{2\eps}}\bigg\{\int_t^{t+h}\bigg[\frac{\sin\big((t+h-t')\jb{\lambda_n}\big)}{\jb{\lambda_n}}\bigg]^2dt' \\
&\,\,+ \int_0^{t}\bigg[\frac{\sin\big((t+h-t')\jb{\lambda_n}\big)-\sin\big((t-t')\jb{\lambda_n}\big)}{\jb{\lambda_n}}\bigg]^2dt'\bigg\}\\
&\les \sum_{n\geq 0}\frac{\varphi_n(\x)^2}{\jb{\lambda_n}^{2+2\eps}}\big\{h+t\sin\big(\frac{h\jb{\lambda_n}}{2}\big)^2\big\}\\
&\les \sum_{n\geq 0}\frac{\varphi_n(\x)^2}{\jb{\lambda_n}^{2+2\eps}}(h\jb{\lambda_n})^{\eps} \les h^{\eps},
\end{align*}
which leads as in the previous section to $\Psi\in C\big([0,T];W^{-\eps,\infty}(\M)\big)$ almost surely.

Lastly, the tail estimate is obtained through the same argument as in the previous section. This concludes the proof of Proposition \ref{prop-psi}.
\end{proof}

\section{Local well-posedness results}\label{sec-local}

\subsection{Proof of Theorems~\ref{thm-SDNLW} and~\ref{THM:LWP}}\label{subsec-SDNLW+NLW}
We begin by establishing a general local well-posedness result for a perturbed version of \eqref{NLW}. Let us consider the nonlinear wave equations with a general nonlinearity
\begin{align}\label{gNLW}
\begin{cases}
\dt^2 w +(1 -  \Dlg)  w + \nu\dt w  +  F_k(w)  = 0 \\
(w, \dt w) |_{t = 0} = (0,0)
\end{cases}
\end{align}
where
\begin{align*}
F_k(w)=w^k+\sum_{\ell=0}^{k-1} f_{\ell}w^{\ell}
\end{align*}
for some functions $f_{\ell}: \R^+\times\M \rightarrow \R$, and $\nu\in \{0,1\}$. Note that here we only consider the dynamics \eqref{gNLW} starting from zero initial data, as the data for the Cauchy problem is contained in the forcing terms $f_\ell$. This is not a restriction, as the case of a general initial data $(w,\dt w)|_{t=0}=(w_0,w_1)$ can be put in the form \eqref{gNLW} by decomposing $w = (\dt V(t)w_0 + V(t)(w_0+w_1)) + W$ where $W$ solves \eqref{gNLW} with $\wt F_k(W)=W^k+\sum_{\ell=0}^{k-1}\wt f_\ell W^\ell$ for some data $\wt f_\ell$ depending on $f_\ell$ and $(w_0,w_1)$. 
\begin{proposition}\label{prop-glocal}
There exists $\eps_0=\eps_0(k)>0$ such that if $s_1=1-\eps$ for any $0<\eps<\eps_0$, then for any $q> 1$ there exists $C>0$ such that for any $R\geq 1\geq \theta>0$, and any $f_{\ell}\in L^q\big([0,1];W^{-\frac\eps2,\infty}(\M)\big)$ with $\|f_{\ell}\|_{L^q([0,1];W^{-\frac\eps2,\infty}(\M))}\leq R$, $\ell=0,...,k-1$, if we set
\begin{align*}
\dl=C(\theta R^{-1})^{q'}\in (0,1]
\end{align*} 
then \eqref{gNLW} admits a unique solution $w\in  C\big([0,\dl];H^{s_1}(\M)\big)\cap C^1([0,\dl];H^{s_1-1}(\M))$, which satisfies 
\begin{align*}
\|(w,\dt w)\|_{C([0,\dl];\H^{s_1})}\leq \theta.
\end{align*}
Moreover, the flow map 
\begin{align*}
(f_0,...,f_{k-1})\in L^q\big([0,1];W^{-\frac{\eps}2,\infty}(\M)\big)^k \longmapsto (w,\dt w)\in  C\big([0,\dl];\H^{s_1}(\M)\big)
\end{align*} 
is continuous. Lastly, the same local well-posedness result holds if we replace $F_k$ in \eqref{gNLW} by
\begin{align}\label{FNk}
F_{N,k}(w)= \P_N\bigg((\P_Nw)^k + \sum_{\ell=0}^{k-1}f_{\ell}(\P_N w)^{\ell}\bigg),
\end{align}
uniformly in $N\in\N$.
\end{proposition}
\begin{proof}
For $\dl\in (0,1]$, $\nu\in\{0,1\}$, let us define the nonlinear operator on $C\big([0,\dl];H^{s_1}(\M)\big)$ by
\begin{align*}
\Upsilon_{\dl}(w)(t)&=-\int_0^t e^{-\frac{\nu}{2}(t-t')}\frac{\sin\Big((t-t')\sqrt{1-\frac{\nu^2}{4}-\Dlg}\Big)}{\sqrt{1-\frac{\nu^2}{4}-\Dlg}}F_k(w)dt'.
\end{align*}
 We shall prove that for $\dl$ small enough, $\Upsilon_{\dl}$ defines a contraction mapping in a ball of radius $\theta$ in $C\big([0,\dl];H^{s_1}(\M)\big)\cap C^1([0,\dl];H^{s_1-1}(\M))$.
 
We use \eqref{def-func-calculus} to define and evaluate the $H^{s_1}(\M)$ norm of the operators, and that $H^{s_1}(\M)=B^{s_1}_{2,2}(\M)$, so that we get the first bound
\begin{align*}
\|(\Upsilon_{\dl}(w),\dt\Upsilon_\dl(w))\|_{C([0,\dl]; \H^{s_1})}&\les \|w^k\|_{L^{1}_{\dl} H^{s_1-1}}+\sum_{\ell=0}^{k-1}\|f_{\ell}w^{\ell}\|_{L^{1}_{\dl} H^{s_1-1}}\\
&\les\|w^{k}\|_{L^{1}_{\dl}B^{-\eps}_{2,2}}+\sum_{\ell=0}^{k-1}\|f_{\ell}w^{\ell}\|_{L^{1}_{\dl}B^{-\eps}_{2,2}}.
\end{align*}
We begin by treating the first term, which we can simply estimate by
\begin{align*}
\|w^{k}\|_{L^{1}_{\dl}B^{-\eps}_{2,2}}\les \dl\|w^k\|_{L^{\infty}_{\dl}L^2}\les \dl\|w\|_{L^{\infty}_{\dl}L^{2k}}^{k}.
\end{align*}
Thus, provided that $\eps<\frac{1}{k}$, we can use Sobolev's inequality to get the bound
\begin{align*}
\|w^k\|_{L^{1}_{\dl} H^{s_1-1}}\les \dl\|w\|_{L^{\infty}_{\dl} H^{s_1}}^{k}.
\end{align*}
As for the other terms, we now use the product rule in Corollary~\ref{cor-besov}~(iii), to get for $\ell=1,...,k-1$
\begin{align*}
\|f_{\ell}w^{\ell}\|_{L^{1}_{\dl}B^{-\eps}_{2,2}}\les \dl^{\frac1{q'}}\|f_{\ell}\|_{L^{q}_{\dl}B^{-2\eps/3}_{\infty,2}}\|w^{\ell}\|_{L^{\infty}_{\dl}B^{\eps}_{2,2}}\les \dl^{\frac1{q'}}\|f_{\ell}\|_{L^{q}_{\dl}W^{-\frac\eps2,\infty}}\|w\|_{L^{\infty}_{\dl}B^{\eps}_{2\ell,2}}^{\ell},
\end{align*}
and then use that
\begin{align*}
\|w\|_{L^{\infty}_{\dl}B^{\eps}_{2\ell,2}}\les \|w\|_{L^{\infty}_{\dl}B^{s_1}_{2,2}}
\end{align*}
for any $\ell=1,...,k-1$, provided that $\eps<\frac1{2(k-1)}$. The term for $\ell=0$ is estimated directly, so that all in all we arrive at
\begin{align*}
\|(\Upsilon_{\dl}(w),\dt\Upsilon_\dl(w))\|_{C([0,\dl]; \H^{s_1})} &\leq c_1\dl\|w\|_{L^{\infty}_{\dl} H^{s_1}}^{k}\\
&\qquad+c_2\dl^{\frac1{q'}}\sum_{\ell=0}^{k-1}\|f_{\ell}\|_{L^{q}_{\dl}W^{-\frac\eps2,\infty}}\|w\|_{L^{\infty}_{\dl} H^{s_1}}^{\ell}.
\end{align*}
In particular for $ R\geq 1\geq \theta>0$ and $\dl=C(\theta R^{-1})^{q'}$, $\Upsilon_{\dl}$ maps the ball of radius $\theta$ in itself. From the same computations, if $\Upsilon_{\dl}'$ is defined similarly to $\Upsilon_{\dl}$ with respect to another data $w_0',w_1',f_0',...,f_{k-1}'$ then we get
\begin{align}
&\|\Upsilon_{\dl}(w)-\Upsilon_{\dl}'(w')\|_{C([0,\dl]; H^{s_1})}\notag\\& \leq c_1\dl\|w-w'\|_{L^{\infty}_{\dl} H^{s_1}}\big(\|w\|_{L^{\infty}_{\dl} H^{s_1}}+\|w'\|_{L^{\infty}_{\dl} H^{s_1}}\big)^{k-1}\notag\\
&\qquad+c_2\dl^{\frac1{q'}}\|f_{0}-f_{0}'\|_{L^{q}_{\dl}W^{-\frac\eps2,\infty}}+c_3\dl^{\frac1{q'}}\sum_{\ell=1}^{k-1}\Big\{\|f_{\ell}-f_{\ell}'\|_{L^{q}_{\dl}W^{-\frac\eps2,\infty}}\|w\|_{L^{\infty}_{\dl} H^{s_1}}^{\ell}\notag\\
&\qquad+\|w-w'\|_{L^{\infty}_{\dl} H^{s_1}}\|f_{\ell}'\|_{L^{q}_{\dl}W^{-\frac\eps2,\infty}}\big(\|w\|_{L^{\infty}_{\dl} H^{s_1}}+\|w'\|_{L^{\infty}_{\dl} H^{s_1}}\big)^{\ell-1}\Big\},
\label{diff-Upsilon}
\end{align}
and similarly for the time derivative. This shows the contraction property and the continuous dependence on the $f_{\ell}$'s up to taking $\dl$ smaller depending on $c_1,c_2,c_3$. 
\end{proof}
With Proposition \ref{prop-glocal} at hand, we can now get our main local well-posedness results.
\begin{proof}[Proof of Theorems~\ref{thm-SDNLW} and~\ref{THM:LWP}.]
We begin by proving Theorem~\ref{thm-SDNLW}. Recall that we see $\Psi_{\damp}$ as a random variable on $(\H^s(\M)\times\O,\mu\otimes\Prob)$. For any $M\in\N$ we take 
\begin{align*}
\Si_{M} &= \Big\{(u_0,u_1,\o)\in\H^s(\M)\times\O,~\Psi_{\damp}\in C\big([0,1];W^{-\frac\eps2,\infty}(\M)\big) \text{ and }\forall \ell=1,...,k,\notag\\&\qquad\qquad \big\|H_{\ell}\big(\P_N\Psi_\damp(u_0,u_1,\o);\s_N(x)\big) - \,:\!\Psi_{\damp}^{\ell}(u_0,u_1,\o)\!:\,\big\|_{L^2\big([0,1];W^{-\frac{\eps}2,\infty}(\M)\big)}\to 0,\\
&\qquad\qquad \text{ and } ~\sup_{N\in\N}\big\|H_{\ell}\big(\P_N\Psi_\damp(u_0,u_1,\o);\s_N(x)\big)\big\|_{L^{2}([0,1];W^{-\frac\eps2,\infty})}\leq M\Big\}.
\end{align*}
 In view of the large deviation bounds given by Proposition~\ref{prop-psidamp}, we see that
\begin{align}\label{measure-OT}
\mu\otimes\Prob(\H^s(\M)\times\O\setminus \Si_{M})\leq Ce^{-cM^{\frac2k}}.
\end{align}
Moreover, \eqref{eq-w-r} and Proposition~\ref{prop-psidamp} show that for any $(u_0,u_1,\o)\in\Si_{M}$, we can apply Proposition~\ref{prop-glocal} with $R=M$, $\theta=1$ and $f_\ell = {\binom k \ell}H_{\ell}\big(\P_N\Psi_\damp(u_0,u_1,\o);\s_N(x)\big)$ for any $N\in\N\cup\{\infty\}$, 
with the convention that $\P_{\infty}\Psi_{\damp} = \Psi_{\damp}$ and $H_{\ell}\big(\P_\infty\Psi_\damp(u_0,u_1,\o);\s_\infty(x)\big) = \,:\!\Psi_\damp^\ell(u_0,u_1,\o)\!:\,$. We thus get solutions $w_N$ and $w_{\infty}=w$ to \eqref{eq-w-r} on $[0,T]$ with $T = CM^{-2}$ independent of $N$. Moreover since $\P_N\Psi_{\damp}\in C\big([0,T];W^{-\frac\eps2,\infty}(\M)\big)$, $N\in\N\cup\{\infty\}$, we have 
\begin{align*}
u_N=\P_N\Psi_{\damp}+w_N\in \P_N\Psi_{\damp} + C\big([0,T];H^{s_1}(\M)\big)\cap C^1\big([0,T];H^{s_1}(\M)\big).
\end{align*}
 Hence in view of Proposition \ref{prop-psidamp} we have $u_N$ and $u$ in $C\big([0,T];H^{-\eps}(\M)\big)\cap C^1\big([0,T];H^{-1-\eps}(\M)\big)$ and using again Proposition~\ref{prop-psidamp}, we get the convergences $\P_N\Psi_{\damp}\rightarrow \Psi_{\damp}$ and $w_N\rightarrow w$. From the continuous dependence in Proposition~\ref{prop-glocal}, we thus get that $u_N \rightarrow u$ in $C\big([0,T];H^{-\eps}(\M)\big)$. The proof of Theorem~\ref{thm-SDNLW} is completed by taking
\begin{align*}\Si =\liminf_{M\geq 1} \Si_{M}\end{align*} which, by \eqref{measure-OT} and Borel-Cantelli's lemma, is of full probability. The proof of Theorem~\ref{THM:LWP} follows through the same argument, with $\P_NS(t)(u_0,u_1)$ in place of $\P_N\Psi_{\damp}(u_0,u_1,\o)$ and $(\H^s(\M),\mu)$ in place of $(\H^s(\M)\times\O,\mu\otimes\Prob)$.
\end{proof}

\subsection{Deterministic estimates}
We collect here the deterministic estimates needed to prove Theorem~\ref{thm-SNLW}. Let us recall from \cite{GKO} that for $s_1\in (0,1)$, a pair $(q,r)$ is $s_1$-admissible (respectively $(\widetilde{q},\widetilde{r})$ dual $s_1$-admissible) if $1\leq \widetilde{q}<2<q\leq\infty$, $1<\widetilde{r}\leq 2 \leq r <\infty$ and
\begin{equation*}
\frac{1}{q}+\frac{2}{r}=1-s_1=\frac{1}{\widetilde{q}}+\frac{2}{\widetilde{r}}-2,\hspace{10pt}\frac{2}{q}+\frac{1}{r}\leq\frac12,\hspace{10pt}\text{ and }\hspace{10pt}\frac{2}{\widetilde{q}}+\frac{1}{\widetilde{r}}\geq\frac52.
\end{equation*}
Let us then consider the following inhomogeneous linear wave equations
\begin{align}\label{ILW}
\begin{cases}
(\dt^2+1-\Dlg)u = f \text{ on }[0,T]\times\M,\\
(u,\dt u)\big|_{t=0}=(u_0,u_1)\in \H^{s_1}(\M)
\end{cases}
\end{align}
for some $T\in (0,1]$. For $s_1\in (0,1)$ and $(q,r)$ an $s_1$-admissible pair (respectively $(\widetilde{q},\widetilde{r})$ a dual $s_1$-admissible pair), we set
\begin{align}\label{def-X}
X^{s_1}_T= C\big([0,T];H^{s_1}(\M)\big)\cap C^1\big([0,T];H^{s_1-1}(\M)\big)\cap L^q\big([0,T];L^r(\M)\big)
\end{align}
and
\begin{align*}
\widetilde{X}^{s_1}_T=  L^1\big([0,T];H^{s_1-1}(\M)\big)+ L^{\widetilde{q}}\big([0,T];L^{\widetilde{r}}(\M)\big).
\end{align*}
\begin{lemma}\label{lem-strichartz}
Let $u$ be a solution of \eqref{ILW}, then the following Strichartz estimate holds:
\begin{align}\label{Strichartz}
\|u\|_{X^{s_1}_T}\les \|(u_0,u_1)\|_{\H^{s_1}}+\|f\|_{\widetilde{X}^{s_1}_T}.
\end{align}
\end{lemma}
\begin{proof}
Due to the finite speed of propagation and in the absence of boundary, this follows from the same Strichartz estimates as in \cite{K,MSS} for the variable coefficients linear wave equations on $\R^2$.
\end{proof}
Next, we recall the following technical result from \cite{GKO}.
\begin{lemma}\label{lem-pair}
Let $s_1$ be as in Theorem~\ref{thm-SNLW}. Then there exist an $s_1$-admissible pair $(q,r)$ and a dual $s_1$-admissible pair $(\widetilde{q},\widetilde{r})$ satisfying
\begin{align}\label{prop-pair}
q\geq k\widetilde{q},~r\geq k\widetilde{r}
\end{align}
where the first inequality is strict in the case $s_1>\scrit$.
\end{lemma}
\begin{proof}
This is the content of the discussion in \cite[Subsection 3.1]{GKO}.
\end{proof}
\subsection{Proof of Theorem~\ref{thm-SNLW}}
We finally prove the local result for SNLW. As above, we define for $N\in\N\cup\{\infty\}$ and $(u_0,u_1)\in\H^{s_1}(\M)$,
\begin{align*}
\Upsilon_T(w)&=\cos\big(t\sqrt{1-\Dlg}\big)u_0+\frac{\sin\big(t\sqrt{1-\Dlg}\big)}{\sqrt{1-\Dlg}}u_1\\
&\qquad -\sum_{\ell=0}^{k}{\binom k \ell}\int_0^t\frac{\sin((t-t')\sqrt{1-\Dlg})}{\sqrt{1-\Dlg}}H_{\ell}\big(\P_N\Psi(t');\s_N(t',x)\big)w^{k-\ell}(t')dt',
\end{align*}
with the same convention as above for $N=\infty$.

We then prove a result similar to \cite[Proposition 3.5]{GKO}. 
\begin{proposition}\label{prop-subcrit}
Let $k\in\N$ and $s_1$ be as in Theorem~\ref{thm-SNLW}, and take $(q,r)$ and $(\widetilde{q},\widetilde{r})$ given by Lemma~\ref{lem-pair}. Then there exist $0<\eps\ll 1$ and $\al>0$ such that for any $N\in\N\cup\{\infty\}$,
\begin{multline}\label{fixpt-SNLW}
\|\Upsilon_T(w)\|_{X^{s_1}_T} \les \|(u_0,u_1)\|_{\H^{s_1}} + \big\|H_{k}\big(\P_N\Psi(t);\s_N(t,x)\big)\big\|_{L^1_TH^{s_1-1}}\\ + T^{\al}\sum_{\ell=1}^{k-1}\big\|H_{\ell}\big(\P_N\Psi(t);\s_N(t,x)\big)\big\|_{L^p_TW^{-\frac\eps2,\infty}}\|w\|_{X^{s_1}_T}^{k-\ell}+ T^{\frac1{\widetilde{q}}-\frac{k}{q}}\|w\|_{X^{s_1}_T}^k,
\end{multline}
for some large $p$. Moreover, a similar estimate holds for the difference as in \eqref{diff-Upsilon}.
\end{proposition}
\begin{proof}
The linear solution with the term for $\ell=k$ in $\Upsilon_T$ are directly estimated with the Strichartz estimate \eqref{Strichartz} of Lemma~\ref{lem-strichartz} to give the first two terms in the right-hand side of \eqref{fixpt-SNLW}.

 As for the term $\ell=0$, we have from the Strichartz estimate \eqref{Strichartz} and H\"older's inequality with \eqref{prop-pair}
 \begin{align*}
 \bigg\|\int_0^t\frac{\sin((t-t')\sqrt{1-\Dlg})}{\sqrt{1-\Dlg}}w^{k}(t')dt'\bigg\|_{X^{s_1}_T}\les \|w^k\|_{\widetilde{X}^{s_1}_T} \les \|w\|_{L^{k\widetilde{q}}_TL^{k\widetilde{r}}}^k \les T^{\frac1{\widetilde{q}}-\frac{k}{q}}\|w\|_{X^{s_1}_T}^k.
\end{align*}  
 
 Hence it remains to show
\begin{align*}
&\bigg\|\int_0^t\frac{\sin((t-t')\sqrt{1-\Dlg})}{\sqrt{1-\Dlg}}H_{\ell}\big(\P_N\Psi(t');\s_N(t',x)\big)w^{k-\ell}(t')dt'\bigg\|_{X^{s_1}_T}\\
&\qquad\les T^\al \big\|H_{\ell}\big(\P_N\Psi(t);\s_N(t,x)\big)\big\|_{L^p_TW^{-\frac\eps2,\infty}} \|w\|_{X^{s_1}_T}^{k-\ell},
\end{align*}
for $\ell=1,...,k-1$. As in \cite[Proposition 3.5]{GKO}, by interpolation we have for any $0<\eps< s_1\wedge (1-s_1)$
\begin{align}\label{embedding}
\widetilde{X}^{s_1}_T \supset L^{\widetilde{q_1}}\big([0,T];W^{-\eps,\widetilde{r_1}}(\M)\big) ~\text{ and }~ L^{q_1}\big([0,T];W^{\eps,r_1}(\M)\big) \supset X^{s_1}_T,
\end{align}
with
\begin{align}\label{eq-q}
\frac{1}{q_1}=\frac{1-\eps/s_1}{q} ~\text{ and }~\frac{1}{r_1}=\frac{1-\eps/s_1}{r}+\frac{\eps/s_1}{2},
\end{align}
and
\begin{align}\label{eq-qtilde}
\frac{1}{\widetilde{q_1}}= \frac{1-\eps/(1-s_1)}{\widetilde{q}}+\frac{\eps/(1-s_1)}{1} ~\text{ and }~ \frac{1}{\widetilde{r_1}}=\frac{1-\eps/(1-s_1)}{\widetilde{r}}+\frac{\eps/(1-s_1)}{2}.
\end{align}
Then, using Lemma~\ref{lem-strichartz} with the first embedding in \eqref{embedding}, we have 
\begin{align*}
&\bigg\|\int_0^t\frac{\sin((t-t')\sqrt{1-\Dlg})}{\sqrt{1-\Dlg}}H_{\ell}\big(\P_N\Psi(t');\s_N(t',x)\big)w^{k-\ell}(t')dt'\bigg\|_{X^{s_1}_T}\\ &\les \big\|H_{\ell}\big(\P_N\Psi(t);\s_N(t,x)\big)w^{k-\ell}\big\|_{\widetilde{X}^{s_1}_T}\\&  \les \big\|H_{\ell}\big(\P_N\Psi(t);\s_N(t,x)\big)w^{k-\ell}\big\|_{L^{\widetilde{q_1}}_TW^{-\eps,\widetilde{r_1}}}\\
&\les \big\|H_{\ell}\big(\P_N\Psi(t);\s_N(t,x)\big)w^{k-\ell}\big\|_{L^{\widetilde{q_1}}_TB^{-\frac\eps2}_{\widetilde{r_1},\infty}}.
\end{align*}
 Next, we can use Corollary~\ref{cor-besov}~(iii) and (i) with H\"older's inequality to estimate this last term with
\begin{align*}
&\big\|H_{\ell}\big(\P_N\Psi(t);\s_N(t,x)\big)\big\|_{L^{1/(\frac1{\widetilde{q_1}}-\frac1{\widetilde{q_2}})}_TB^{-\frac\eps2}_{\infty,\infty}} \|w^{k-\ell}\|_{L^{\widetilde{q_2}}_TB^{\frac{2\eps}3}_{\widetilde{r_1},\infty}}\\ & \les \big\|H_{\ell}\big(\P_N\Psi(t);\s_N(t,x)\big)\big\|_{L^{1/(\frac1{\widetilde{q_1}}-\frac1{\widetilde{q_2}})}_TW^{-\frac\eps2,\infty}} \|w\|_{L^{(k-\ell) \widetilde{q_2}}_TW^{\eps,(k-\ell) \widetilde{r_1}}}^{\ell},
\end{align*}
where $\widetilde{q_1}<\widetilde{q_2}<\widetilde{q}$. The proof of Proposition~\ref{prop-subcrit} is then completed once we notice that for $\ell\ge 1$,
\begin{align*}
\|w\|_{L^{(k-\ell) \widetilde{q_2}}_TW^{\eps,(k-\ell) \widetilde{r_1}}} \les T^\al\|w\|_{L^{q_1}_TW^{\eps,r_1}}
\end{align*}
for some small $\al >0$ provided that $(k-1)\widetilde{q_2}<q_1$ and $(k-1)\widetilde{r_1}\leq r_1$, which can be insured by taking $\eps$ small enough in view of the choice of $\widetilde{q_2}$ and \eqref{prop-pair}-\eqref{eq-q}-\eqref{eq-qtilde}. Lastly, we invoke the second embedding in \eqref{embedding} to conclude the proof of the proposition.
\end{proof}
With this proposition at hand, we can conclude as in Subsection~\ref{subsec-SDNLW+NLW} in the subcritical case $s>\scrit$, with a stopping time $T=T_\o(\|(u_0,u_1)\|_{\H^{s_1}})>0$. However, in the case $k\geq 4$ and $s=\scrit$ then we have $T^{\frac1{\widetilde{q}}-\frac{k}{q}} = 1$ and so we cannot recover the contraction property by taking $T=T_\o(\|(u_0,u_1)\|_{\H^{s_1}})$ small enough. Instead, defining as in \cite{GKO} the slightly weaker norm
\begin{align*}
\|u\|_{Y^{s_1}_T}= \max\big(\|u\|_{L^q_TL^r}, \|u\|_{L^q_TL^r}^{1-\frac{\eps}{s_1}}\|u\|_{L^{\infty}_TH^{s_1}}^{\frac{\eps}{s_1}}\big),
\end{align*}
we can repeat the argument as in the proof of Proposition~\ref{prop-subcrit} using the interpolation inequality $\|u\|_{L^{q_1}_TW^{\eps,r_1}}\les \|u\|_{Y^{s_1}_T}$, to get
\begin{align*}
\|\Upsilon_T(w)\|_{Y^{s_1}_T} &\les \|S(t)(u_0,u_1)\|_{Y^{s_1}_T} + \big\|H_{k}\big(\P_N\Psi(t);\s_N(t,x)\big)\big\|_{L^1_TH^{s_1-1}}\\ &~~+ T^\al\sum_{\ell=1}^{k-1}\big\|H_{\ell}\big(\P_N\Psi(t);\s_N(t,x)\big)\big\|_{L^p_TW^{-\frac\eps2,\infty}}\|w\|_{Y^{s_1}_T}^{k-\ell}+ \|w\|_{Y^{s_1}_T}^k,
\end{align*}
and similarly for the difference estimate. Since $\|\cdot\|_{Y^{s_1}_T}\rightarrow 0$ as $T\rightarrow 0$, taking then $T$ small enough such that 
\begin{align*}
\|S(t)(u_0,u_1)\|_{Y^{s_1}_T} +  \big\|H_{k}\big(\P_N\Psi(t);\s_N(t,x)\big)\big\|_{L^1_TH^{s_1-1}} \leq \frac{\theta}2
\end{align*}
 for some small $0<\theta\ll 1$, then $\Upsilon_T$ defines a contraction on the ball of radius $\theta$ (in $Y^{s_1}_T$). Lastly, repeating again the argument to obtain \eqref{fixpt-SNLW} with the interpolation inequality we can control
\begin{align*}
\|w\|_{X^{s_1}_T}=\|\Upsilon_T(w)\|_{X^{s_1}_T}&\les \|(u_0,u_1)\|_{\H^{s_1}} + \big\|H_{k}\big(\P_N\Psi(t);\s_N(t,x)\big)\big\|_{L^1_TH^{s_1-1}}\\ &~~+ T^\al\sum_{\ell=1}^{k-1}\big\|H_{\ell}\big(\P_N\Psi(t);\s_N(t,x)\big)\big\|_{L^p_T W^{-\frac\eps2,\infty}}\|w\|_{Y^{s_1}_T}^{k-\ell}+ \|w\|_{Y^{s_1}_T}^k,
\end{align*}
which shows that $w\in X^{s_1}_T$ and concludes the proof of Theorem~\ref{thm-SNLW}.

\section{Global well-posedness and invariance of the Gibbs measure}\label{sec-global}
In this last section, we present the proof of Theorem~\ref{thm-invariance}, the one for Theorem~\ref{thm-global} following through the same argument. In the rest of the section, we then assume that $k$ is an odd integer, and we fix some parameters $s<0<s_2<s_1<1+s$ such that $0<-s< 1-s_1 < 1-s_2 \ll 1$. We also simply denote $\Psi_\damp$ by $\Psi$.
\subsection{The frequency truncated SDNLW}
As in \cite{BO96,Tz,BT2,BTT2}, for any $N\in\N$ and $k\ge 2$ we look at the approximating equation
\begin{align}\label{FNLW}
\begin{cases}
\big(\dt^2+1-\Dlg + \dt\big)u + \P_NH_k\big(\P_Nu_N;\s_N(x)\big) = \sqrt{2}\xi,\\
(u,\dt u)\big|_{t=0}=(u_0,u_1).
\end{cases}
\end{align}
Note that the same argument as in the previous section shows $\mu\otimes\Prob$-almost sure local well-posedness for \eqref{FNLW}, thus defining a local flow map
\begin{align*}
\Phi^N(t) : \H^s(\M)\times \O \to \H^s(\M).
\end{align*}
We have the following global well-posedness result for \eqref{FNLW}.
\begin{proposition}\label{prop-FNLW}
For any $N\in\N$, \eqref{FNLW} is $\mu\otimes\Prob$-almost surely globally well-posed. Moreover, the truncated Gibbs measure
\begin{align}
d\rho_{N,k+1}=Z_N^{-1}\exp\bigg(-\frac{1}{k+1}\int_{\M}\,:\!(\P_N u)^{k+1}\!:\,dx\bigg)d\mu,
\label{rhoN}
\end{align} 
is invariant under \eqref{FNLW}, in the sense that for any $F\in C_b(\H^s(\M);\R)$ and any $t\ge 0$,
\begin{align*}
\int_{\H^s(\M)}\int_\O F\big[\Phi^N(t)(u_0,u_1,\o)\big]d\Prob(\o)d\rho_{N,k+1}(u_0,u_1) = \int_{\H^s(\M)}F(u_0,u_1)d\rho_{N,k+1}(u_0,u_1).
\end{align*}
\end{proposition}
\begin{proof}
After expanding the solution to \eqref{FNLW} as $u_N = \Psi + w_N$ and writing the equation for $w_N$, we can apply Proposition~\ref{prop-glocal} above to get local well-posedness for $w_N$, for all $N\in\N$, in the sense that there exists some stopping time $T_N$ almost surely positive such that there exists a unique solution $w_N\in C([0,T_N];H^{s_1}(\M))\cap C^1([0,T_N];H^{s_1-1}(\M))$ to
\begin{align}\label{FwN}
\begin{cases}
(\dt^2+1-\Dlg+\dt)w_N + F_{N,k}(w_N)=0,\\
(w_N,\dt w_N)|_{t=0}=(0,0),
\end{cases}
\end{align}
where $F_{N,k}$ is as in \eqref{FNk} with $f_{\ell}=H_{\ell}\big(\P_N\Psi(u_0,u_1,\o);\s_N(x)\big)$, and with $s_1$ as in Proposition \ref{prop-glocal} (with $\eps$ replaced by $-s$). Thus justifies that the local flow map
\begin{align*}
\Phi^N(t) : (u_0,u_1,\o) \mapsto \big(\Psi(u_0,u_1,\o)(t) + w_N(t),\dt\Psi(u_0,u_1,\o)(t) + \dt w_N(t)\big)
\end{align*}
is indeed almost surely well-defined on $[0,T_N]$ for some $T_N=T_N(u_0,u_1,\o)>0$.

 Then, defining the energy
\begin{align*}
\En_N=\frac12\int_{\M}\big\{(\dt w_N)^2 + |\nabla w_N|^2 + w_N^2\big\}dx + \frac{1}{k+1}\int_{\M}(\P_Nw_N)^{k+1}dx,
\end{align*}
we can use \eqref{FwN} and \eqref{Hsum} to compute
\begin{align*}
\frac{d}{dt}\En_N &= -\|\dt w_N\|_{L^2}^2 +\big\langle \dt w_N, \P_NH_k\big(\P_Nw_N+\P_N\Psi;\s_N(x)\big))-\P_N(\P_Nw_N)^k\big\rangle\\
&\le\Big\langle \dt w_N, \P_NH_k\big(\P_N\Psi;\s_N(x)\big)\Big\rangle\\
&\qquad\qquad+\sum_{\ell=1}^{k-1}{\binom k \ell}\Big\langle \dt w_N,\P_N\Big[(\P_Nw_N)^{k-\ell}H_{\ell}\big(\P_N\Psi;\s_N(x)\big)\Big]\Big\rangle,
\end{align*}
where $\langle \cdot,\cdot\rangle$ is the usual inner product in $L^2(\M)$. With $\En_N(0)=0$, this gives
\begin{align}\label{En}
\En_N(t) &\le  \int_0^t\Big\langle \dt w_N(t'), \P_NH_k\big(\P_N\Psi(t');\s_N(x)\big)\Big\rangle dt'\notag\\
&\qquad+\sum_{\ell=1}^{k-1}{\binom k \ell}\int_0^t\Big\langle \dt w_N(t'),\P_N\Big[\big(\P_Nw_N(t')\big)^{k-\ell}H_{\ell}\big(\P_N\Psi(t');\s_N(x)\big)\Big]\Big\rangle dt'.
\end{align}

The first term in the right-hand side of \eqref{En} can be estimated via Cauchy-Schwarz and Young's inequalities to get the bound
\begin{align*}
&\Big|\int_0^t\Big\langle \dt w_N(t'), \P_NH_k\big(\P_N\Psi(t');\s_N(x)\big)\Big\rangle dt'\Big|\\&\les \int_0^t\|\dt w_N(t')\|_{L^2}^2dt' + \big\|\P_NH_k\big(\P_N\Psi(t');\s_N(x)\big)\big\|_{L^2_tL^2}^2\\
&\les \int_0^t\En_N(t')dt' + C(N,t)
\end{align*}
for some constant $C(N,t)$ almost surely finite for any finite $N$ and $t$. In the second step we used that for fixed $N\in\N$, $\P_NH_k\big(\P_N\Psi(t');\s_N(x)\big)$ is indeed smooth with $L^2(\R_+;L^2(\M))$-norm depending on (and blowing-up with) $N$, and that $k$ is odd so that the potential part of the energy is non-negative.

As for the terms in the sum above, even though we work with $N$ fixed and do not need to have bounds uniform in $N$, the homogeneity in the terms on the second line of \eqref{En} does not allow us to conclude directly by a crude estimate on these terms and Gronwall's inequality when\footnote{When $k=3$, the integration by part trick is not needed, and one can instead use the argument of \cite{BT3}.} $k\ge 5$. Thus we use the integration by part trick of \cite{OP} to get for $1\le \ell \le k-1$:
\begin{align*}
&\int_0^t\Big\langle \dt w_N(t'),\P_N\Big[\big(\P_Nw_N(t')\big)^{k-\ell}H_{\ell}\big(\P_N\Psi(t');\s_N(x)\big)\Big]\Big\rangle dt'\\
&  = \int_0^t\Big\langle \frac1{k-\ell+1}\dt (\P_Nw_N(t'))^{k-\ell+1},H_{\ell}\big(\P_N\Psi(t');\s_N(x)\big)\Big\rangle dt'\\
& = c\Big\langle (\P_Nw_N(t))^{k-\ell+1},H_{\ell}\big(\P_N\Psi(t);\s_N(x)\big)\Big\rangle\\
&\qquad\qquad - c\int_0^t\Big\langle (\P_Nw_N(t'))^{k-\ell+1},\dt\P_N\Psi_N(t') H_{\ell-1}\big(\P_N\Psi(t');\s_N(x)\big)\Big\rangle dt'
\end{align*} 
where we used \eqref{Hdif} in the last step. The first term can be bounded by
\begin{align*}
&\Big|\Big\langle (\P_Nw_N(t))^{k-\ell+1},H_{\ell}\big(\P_N\Psi(t);\s_N(x)\big)\Big\rangle\Big|\\ &\les \|\P_Nw_N(t)\|_{L^{k-\ell+1}}^{k-\ell+1}\big\|H_{\ell}\big(\P_N\Psi(t);\s_N(x)\big)\big\|_{L^\infty}\\
&\le \frac{\eps}{k+1} \|\P_Nw_N(t)\|_{L^{k+1}}^{k+1} + C(\eps)\big\|H_{\ell}\big(\P_N\Psi(t);\s_N(x)\big)\big\|_{L^\infty}^{\frac{k+1}{\ell}}\\
& \le \eps\En_N(t) + C(\eps)\big\|H_{\ell}\big(\P_N\Psi(t);\s_N(x)\big)\big\|_{L^\infty}^{\frac{k+1}{\ell}}
\end{align*} 
by using the compactness of $\M$ and Young's inequality (since $1\le \ell \le k-1$), the definition of $\En_N$ and the same remark as above, for any $0<\eps\ll 1$ so that we can absorb the term with $\En_N(t)$ in the left-hand side of \eqref{En}. Note that from the proof of Proposition \ref{prop-psidamp} we have that the second term is bounded by $C(N,t)$ which is almost surely finite for any finite $t\ge 0$.

As for the other term, we have as above
\begin{align*}
&\Big|\int_0^t\Big\langle (\P_Nw_N(t'))^{k-\ell+1},\dt\P_N\Psi_N(t') H_{\ell-1}\big(\P_N\Psi(t');\s_N(x)\big)\Big\rangle dt'\Big|\\
 &\les \int_0^t\int_{\M}(\P_Nw_N(t'))^{k+1}dxdt' + \int_0^t\|\dt\Psi_N(t')\|_{L^\infty}^{\frac{k+1}\ell}\big\|H_{\ell-1}\big(\P_N\Psi(t');\s_N(x)\big)\big\|_{L^\infty}^{\frac{k+1}\ell}dt'\\
 &\les \int_0^t\En_N(t')dt' + C(N,t).
\end{align*}

 Hence using Gronwall's inequality with $\En_N(0)=0$, we deduce that
\begin{align*}
\sup_{t<T_N}\En_N(t) \les C(N,T_N) <\infty
\end{align*}
almost surely on the set $\{T_N<\infty\}$. Finally, using again that $k$ is odd, we conclude that
\begin{align*}
\sup_{t<T_N}\|(w_N(t),\dt w_N(t))\|_{\H^{1}}^2 \le \sup_{t<T_N}\En_N(t)  <\infty
\end{align*} 
almost surely on $\{T_N<\infty\}$. This shows that $w_N$ exists globally, and so does $u_N$.

As for the invariance of $\rho_{N,k+1}$ under the flow $\Phi^N$ of \eqref{FNLW}, we can write 
\begin{align*}
\Phi^N(t,\o) = \big(\Pi_N\Phi^N(t,\o), (1-\Pi_N)(\Psi,\dt\Psi)(t,\o)\big) \text{ on }(E_N\times E_N)\oplus (E_N^{\perp}\times E_N^{\perp}),
\end{align*} 
when we see $(\Psi(t,\o),\dt\Psi(t,\o))$ as a measurable map from $\H^s(\M)$ to $\H^s(\M)$.

First, for the linear part $(1-\Pi_N)(\Psi(t,\o),\dt\Psi(t,\o))$, we can repeat the argument of Proposition \ref{prop-psidamp} to get it leaves the Gaussian measure $(1-\Pi_N)_\star\mu$ invariant; indeed, we have for any $F\in C_b(\H^s(\M);\R)$ and initial data $(u_0,u_1)$ with law $(1-\Pi_N)_\star\mu$:
\begin{align*}
&\int_{\H^s(\M)}\int_\O F\big[(1-\Pi_N)\vec\Psi(t,u_0,u_1,\o)\big]d\Prob(\o)d\mu(u_0,u_1)\\
&=\lim_{M\to\infty}\int_{\H^s(\M)}\int_\O F\big[\P_M(1-\Pi_N)\vec\Psi(t,u_0,u_1,\o)\big]d\Prob(\o)d\mu(u_0,u_1)\intertext{by the dominated convergence theorem, where $\vec\Psi=(\Psi,\dt\Psi)$. Then from the same argument as in the proof of Proposition \ref{prop-psidamp} we have that $[\P_M(1-\Pi_N)]_\star\mu$ is invariant for $\P_M(1-\Pi_N)(\Psi(t,\o),\dt\Psi(t,\o))$, so we can continue with}
&=\lim_{M\to\infty}\int_{\H^s(\M)}F\big[\P_M(1-\Pi_N) u_0,\P_M(1-\Pi_N)u_1)\big]d\mu(u_0,u_1)\\
&=\int_{\H^s(\M)}F\big[(1-\Pi_N) u_0,(1-\Pi_N)u_1)\big]d\mu(u_0,u_1).
\end{align*}
This shows the invariance of $(1-\Pi_N)_\star\mu$ under $(1-\Pi_N)(\Psi(t,\o),\dt\Psi(t,\o))$.

 On the other hand, decomposing $\Pi_Nu_N=\sum_{\lambda_n\leq N}a_n\varphi_n$ and $\Pi_N\dt u_N = \sum_{\lambda_n\leq N}b_n\varphi_n$, we can write $\Pi_N\Phi^N$ as the flow of the finite-dimensional system of stochastic differential equations (SDEs) on $\R^{2\Ld_N}$:
\begin{align}\label{abNLW}
\begin{cases}
da_n = b_ndt\\
{\displaystyle db_n = \Big[-\jb{\lambda_n}^2a_n - \Big\langle \P_NH_k\Big(\P_N\sum_{n_1=0}^{\Ld_N-1}
a_{n_1}\varphi_{n_1};\s_N(x)\Big),\varphi_n\Big\rangle  - b_n\Big]dt + \sqrt{2}d\beta_n(t)}
\end{cases}
\end{align}
for $n=0,...,\Ld_N-1$, where as in Proposition~\ref{prop-psidamp} we define $\Ld_N = \dim E_N$. If we redefine the truncated energy
\begin{align*}
\En_N(a_0,...,a_{\Ld_N-1},b_0,...,b_{\Ld_N-1})&\deff \frac12\sum_{n=0}^{\Ld_N-1} (\jb{\lambda_n}^2a_n^2+b_n^2)\\
&\qquad + \frac{1}{k+1}\int_{\M}H_{k+1}\big(\P_N\sum_{n=0}^{\Ld_N-1}a_n\varphi_n(x);\s_N(x)\big)dx,
\end{align*}
we can repeat the argument of the proof of Proposition~\ref{prop-psidamp} with $\En_N$ instead of $\En_{0,N}$ to get that the truncated Gibbs measure $Z_N^{-1}e^{-G_{N,k+1}}(\Pi_N)_{\star}\mu$, with the density $e^{-G_{N,k+1}}$ as in Lemma~\ref{lem-Gibbs} and the partition function $Z_N$, is invariant under the dynamics of \eqref{abNLW}. All in all, this shows that the full dynamics $\Phi^N = \big(\Pi_N\Phi^N,(1-\Pi_N)(\Psi,\dt\Psi)\big)$ for \eqref{FNLW} leaves $\rho_{N,k+1} = Z_N^{-1}e^{-G_{N,k+1}}(\Pi_N)_\star\mu\otimes(1-\Pi_N)_\star\mu$ invariant.
\end{proof}

\subsection{Proof of Theorem \ref{thm-invariance}}
We now prove the almost sure global existence for \eqref{SDNLW} and the invariance of the Gibbs measure. We begin by constructing a set of arbitrary small complementary probability on which we have good control on the solution to \eqref{FNLW}. We follow closely \cite{BTT2} (see also \cite{HM} for the argument in the context of stochastic equations).

For $N\in\N$, recall that $\Phi^N(t)$ is the global stochastic flow map of \eqref{FNLW} given by Proposition \ref{prop-FNLW}, and take $(u_0,u_1)$ with law $\rho_{N,k+1}$. Note that $\Phi^N(t)(u_0,u_1)$ still exists globally for $\rho_{N,k+1}\otimes\Prob$-almost every $(u_0,u_1,\o)$ since $\rho_{N,k+1}\ll\mu$. By Proposition \ref{prop-FNLW}, we thus have that for any $t_0\ge 0$, the law of $\Phi^N(t_0)(u_0,u_1)$ is also given by $\rho_{N,k+1}$. Moreover, since $B$ in \eqref{STWN} is a cylindrical Wiener process on $L^2(\M)$, we also have that for any $t_0\ge 0$, the translation $t\mapsto t+t_0$ defines a measure-preserving transformation $\tau_{t_0}$ on $(\O,\Prob)$ given by 
\begin{align}\label{Markov}
B(t,\tau_{t_0}(\o))= B(t+t_0,\o)-B(t_0,\o).
\end{align} 
 We can thus extend $\Phi^N(t):\H^s(\M)\times \O\to \H^s(\M)$ as a measure-preserving map
\begin{align*}
\wt\Phi^N(t): \begin{cases}
\big(\H^s(\M)\times \O,\rho_{N,k+1}\otimes\Prob\big) \to \big(\H^s(\M)\times \O,\rho_{N,k+1}\otimes\Prob\big)\\
(u_0,u_1,\o)\mapsto \big(\Phi^N(t)(u_0,u_1,\o),\tau_t(\o)\big).
\end{cases}
\end{align*}
We then have the following control on $\Phi^N(t)$.
\begin{proposition}\label{prop-Si}
There exists $C>0$ such that for all $m,N\in\N$, there exists a measurable set $\Si_N^{m}\subset \H^s(\M)\times\O$ such that
\begin{align}\label{measure-Si}
\rho_{N,k+1}\otimes\Prob(\H^s(\M)\times\O\setminus\Si_N^{m})\leq 2^{-m},
\end{align}
and for all $(u_0,u_1,\o)\in\Si_N^{m}$ and $t\ge 0$, the solution $\Phi^N(t)(u_0,u_1,\o)$ to \eqref{FNLW} satisfies
\begin{align}
\big\|\Phi^N(t)(u_0,u_1,\o)\big\|_{\H^{s}} \leq C\big(m+\log(1+t)\big)^{\frac{k}2}.
\label{estim-Si}
\end{align}

\end{proposition}
\begin{proof}
First, we recall that $\Psi=\Psi_{\damp}$ is the stochastic process on $\H^s(\M)\times\O$ defined by
\begin{align*}
\Psi(t,u_0,u_1,\o) = \dt V(t)u_0 + V(t)(u_0+u_1) + \sqrt{2}\int_{0}^tV(t-t')dB^\o(t').
\end{align*}
Then, for $m,j\in \N$, we set 
\begin{align}\label{dlmj}
\dl = D^{-4}(m+j)^{-2k}
\end{align}
given by Proposition \ref{prop-glocal} with $R=D(m+j)^{\frac{k}2}$, $\theta = R^{-1}$ and $q=2$, for some $D\gg 1$ independent of $N,m,j$ to be fixed later, such that as in Proposition \ref{prop-glocal} in the nonlinear estimates we have $C\dl^{\frac1{q'}}R \le \frac12$ for various constants $C$ such as in \eqref{diff-Upsilon}. 

Next, as in \cite{BTT2} (see also \cite{BO94,BO96}), we can define
\begin{align*}
\Si_{N}^{m,j} &\deff \bigcap_{a=0}^{[2^j/\dl]}\wt\Phi^N(a\dl)^{-1}\Big(\B_N^{m,j}(D)\Big)
\end{align*}
where $[2^j/\dl]$ denotes the integer part of $2^j/\dl$, and
\begin{align}
\B_N^{m,j}(D)&\deff \Big\{(u_0,u_1,\o)\in \H^s(\M)\times\O,\notag\\&\qquad\big\|\big(\Psi(u_0,u_1,\o),\dt\Psi(u_0,u_1,\o)\big)\big\|_{C([0,1];\H^s)} \le D(m+j)^{\frac{k}2},\label{condition1}\\
&\qquad\forall \ell=1,...,k,~\big\|H_{\ell}\big(\P_N\Psi(u_0,u_1,\o);\s_N(x)\big)\big\|_{L^2([0,1];W^{s,\infty})}\le D(m+j)^{\frac{k}2},\label{condition2}\\
&\qquad\big\|H_{\ell}\big(\P_M\Psi(u_0,u_1,\o);\s_M(x)\big)-H_{\ell}\big(\P_N\Psi(u_0,u_1,\o);\s_N(x)\big)\big\|_{L^2([0,1]; W^{s,\infty})}\notag\\
&\qquad\qquad\le M^{-\eps}D(m+j)^{\frac{k}2},~~\forall M\le N\Big\},
\label{condition3}
\end{align} 
for some $0<\eps\ll -s$ in view of \eqref{tail-z-zN}, and $D\gg 1$ is to be taken large enough but independent of $m,j,N$. In particular, note that using \eqref{Markov} and Proposition \ref{prop-glocal} with the choice of $\dl$ in \eqref{dlmj}, for any $a=0,...,[2^j/\dl]$ and $\wt\Phi^N(a\dl)(u_0,u_1,\o)\in \B_N^{m,j}(D)$ we have that 
\begin{align*}
w_{N,a}(t) \deff \Phi^N(t+a\dl)(u_0,u_1,\o)-\Psi(t,\wt\Phi^N(a\dl)(u_0,u_1,\o))
\end{align*}
satisfies
\begin{align}
\|(w_{N,a},\dt w_{N,a})\|_{C([0,\dl];\H^{s_1})}\le D^{-1}(m+j)^{-\frac{k}2},
\label{bd-wN}
\end{align}
where $s_1=1+2s$. This implies that for any $a=0,...,[2^j/\dl]$ and any $(u_0,u_1,\o)\in \Phi^N(a\dl)^{-1}\B_N^{m,j}(D)$, the use of \eqref{condition1} and \eqref{bd-wN} leads to
\begin{align}
&\big\|\Phi^N(t+a\dl)(u_0,u_1,\o)\big\|_{C([0,\dl];\H^{s})}\notag\\
&\le \big\|\big(\Psi(t,\wt\Phi^N(a\dl)(u_0,u_1,\o)),\dt\Psi(t,\wt\Phi^N(a\dl)(u_0,u_1,\o))\|_{C([0,\dl];\H^s)}+\|(w_{N,a},\dt w_{N,a})\|_{C([0,\dl];\H^{s_1})}\notag\\
&\le D(m+j)^{\frac{k}2}+D^{-1}(m+j)^{-\frac{k}2} \le D(m+j+1)^{\frac{k}2}
\label{bduN}
\end{align}
provided that $D$ is large enough.

Next, using that $\wt\Phi^N(t): \big(\H^s(\M)\times \O,\rho_{N,k+1}\otimes\Prob\big) \to \big(\H^s(\M)\times \O,\rho_{N,k+1}\otimes\Prob\big)$ is measure-preserving, we can estimate
\begin{align*}
&\rho_{N,k+1}\otimes\Prob\big(\H^s(\M)\times\O\setminus\Si_{N}^{m,j})\\&\leq \sum_{a=0}^{[2^j/\dl]}\rho_{N,k+1}\otimes\Prob\Big\{\wt\Phi^N(a\dl)^{-1}\Big(\H^s(\M)\times\O\setminus B_N^{m,j}(D)\Big)\Big\}\\
&\les \frac{2^j}{\dl}\rho_{N,k+1}\otimes\Prob\Big(\H^s(\M)\times\O\setminus B_N^{m,j}(D)\Big)\intertext{Using Cauchy-Schwarz inequality with the uniform (in $N$) integrability property of the density $e^{-G_{N,k+1}(u_0)}$ of $\rho_{N,k+1}$ given by Lemma \ref{lem-Gibbs}, we can continue with}
&\les\frac{2^j}{\dl}\|e^{-G_{N,k+1}}\|_{L^2(\mu_0)}\mu\otimes\Prob\Big(\H^s(\M)\times\O\setminus B_N^{m,j}(D)\Big)^{\frac12}\\
&\les\frac{2^j}{\dl}\bigg\{\mu\otimes\Prob\Big(\big\|\big(\Psi(u_0,u_1,\o),\dt\Psi(u_0,u_1,\o)\big)\big\|_{C([0,1];\H^s)} > D(m+j)^{\frac{k}2}\Big)^{\frac12}\\
&\qquad + \sum_{\ell=1}^k\mu\otimes\Prob\Big(\big\|H_{\ell}\big(\P_N\Psi(u_0,u_1,\o);\s_N(x)\big)\big\|_{L^2([0,1];W^{s,\infty})}> D(m+j)^{\frac{k}2}\Big)^{\frac12}\\
&\qquad+\sum_{\ell=1}^k\sum_{M\le N}\mu\otimes\Prob\Big(\big\|H_{\ell}\big(\P_M\Psi(u_0,u_1,\o);\s_M(x)\big)\\
&\qquad\qquad-H_{\ell}\big(\P_N\Psi(u_0,u_1,\o);\s_N(x)\big)\big\|_{L^2([0,1];W^{s,\infty})}\notag > M^{-\eps}D(m+j)^{\frac{k}2}\Big)^{\frac12}\bigg\}.
\end{align*}
Using the tail estimates \eqref{tail-z}, \eqref{tail-z-zN} (which also hold for $\Psi$) and \eqref{tail-sup} given by Proposition \ref{prop-psidamp} together with \eqref{dlmj}, we can finally bound for some $0<\wt\eps\ll\eps\ll -s$
\begin{align}
\rho_{N,k+1}\otimes\Prob\big(\H^s(\M)\times\O\setminus\Si_{N}^{m,j}\big)& \les 2^jD^{4}(m+j)^{2k}\Big\{ e^{-cD^2(m+j)^k}\notag\\
&+\sum_{\ell=1}^k\Big(e^{-cD^{\frac2\ell}(m+j)^{\frac{k}\ell}}+\sum_{M\le N}e^{-cM^{\wt\eps}D^{\frac2\ell}(m+j)^{\frac{k}\ell}}\Big)\Big\}\notag\\
&\les 2^jD^{4}(m+j)^{2k}e^{-cD^{\frac2k}(m+j)}\le 2^{-(m+j)}
\label{Si1}
\end{align}
for $D\gg 1$, independently of $N,m,j$.
 
Next, we define
\begin{align*}
\Si_N^m \deff \bigcap_{j=1}^\infty \Si_N^{m,j}.
\end{align*}
With this definition, we see that \eqref{measure-Si} is a direct consequence of \eqref{Si1}. Moreover, for any $(u_0,u_1,\o)\in\Si_N^m$ and $t\ge 0$, if $j\in\N$ is such that $2^{j-1}< 1+t \le 2^j$, then  \eqref{estim-Si} follows from \eqref{bduN} since $(u_0,u_1,\o)\in \Si_N^{m,j}$.
\end{proof}
We can now finish the proof of the global existence. Let us set
\begin{align*}
\Si^m = \limsup_{N\rightarrow\infty}\Si_{N}^{m}
\end{align*}
and
\begin{align*}
\Si = \bigcup_{m\in\N}\Si^m.
\end{align*}
First, we show that $\Si$ is of full $\rho_{k+1}\otimes\Prob$-measure. From Fatou's lemma we get for any $m\in\N$
\begin{align*}
\rho_{k+1}\otimes\Prob(\Si^m)&\geq \limsup_{N\rightarrow\infty}\rho_{k+1}\otimes\Prob(\Si_N^m)\intertext{Using next the convergence of the density $e^{-G_{N,k+1}}$ of $\rho_{N,k+1}$ to that of $\rho_{k+1}$ given by Lemma \ref{lem-Gibbs} (ii), and \eqref{Si1}, we deduce the lower bound}
&\ge \limsup_{N\rightarrow\infty} \rho_{N,k+1}\otimes\Prob(\Si_N^m) \ge 1-\sum_{j\in\N}2^{-(m+j)}.
\end{align*}
This proves that 
\begin{align*}
\rho_{k+1}\otimes \Prob(\Si) \ge \lim_{m\to\infty} \rho_{k+1}\otimes\Prob(\Si^m) \ge 1- \lim_{m\to\infty}\sum_{j\in\N}2^{-(m+j)} = 1.
\end{align*}

Now for any $(u_0,u_1,\o)\in \Si$, we have by construction that there exists $m\in\N$, $C>0$ and a sequence $N_p\rightarrow\infty$ such that for all $j,p\in\N$ and all $0\le t\le 2^j$,
\begin{align}
\|\Phi^{N_p}(t)(u_0,u_1,\o)\|_{\H^{s}} \le CD(m+j+1)^{\frac{k}2}.
\label{global-bd1}
\end{align}
Thus the global well-posedness part of Theorem~\ref{thm-invariance} follows from the following proposition.
\begin{proposition}\label{prop-approx}
Let $m,j\in\N$, $N_p\to\infty$ and $(u_0,u_1,\o) \in\bigcap_{p\in\N}\Si_{N_p}^{m,j}$. Then $\big\{\Phi^{N_p}(t)(u_0,u_1,\o)-\Psi(t)(u_0,u_1,\o)\big\}_{p\in\N}$ is a Cauchy sequence in $C([0,2^j];\H^{s_2}(\M))$. In particular, $\big\{\Phi^{N_p}(t)(u_0,u_1,\o)\big\}_{p\in\N}$ is a Cauchy sequence in $C([0,2^j];\H^s(\M))$. Here $s<0<s_2<s_1<1+s$ with $0<-s<1-s_1<1-s_2\ll 1$.
\end{proposition}
Note that contrary to \cite{BTT2}, we prove convergence for $\Phi^{N_p}(t)(u_0,u_1,\o)-\Psi(t)(u_0,u_1,\o)$ instead of $\P_{N_p}\Phi^{N_p}(t)(u_0,u_1,\o)$, as in \cite{BO96}. This allows us to prove the convergence in the stronger topology of $\H^{s_2}(\M)$ instead of $\H^s(\M)$, which is used to control the difference between the flow initiated at $\Phi^{N_p}(a\dl)(u_0,u_1,\o)$ and at $\Phi^N(a\dl)(u_0,u_1,\o)$ for $a=1,...,[2^j/\dl]$. 

First, note that since our general local well-posedness result in Proposition~\ref{prop-glocal} is robust enough, we can use the same argument as for the proof of Theorem~\ref{thm-SDNLW} in the previous section, with the truncated dynamics \eqref{FNLW1} in place of \eqref{trunc-SDNLW}, to get that the limit $\Phi(t)(u_0,u_1,\o)=(u,\dt u) = \lim_{N\to\infty}\Phi^N(t)(u_0,u_1,\o)$ exists in $C([0,T];\H^s(\M))$ on a set $\wt\Si$ of full $\mu\otimes\Prob$-probability, for a random time $T=T(u_0,u_1,\o)$ $\mu\otimes\Prob$-almost surely positive, and coincides with the local solution constructed in Theorem~\ref{thm-SDNLW}. Then we use the previous proposition to construct iteratively $\Phi(t)(u_0,u_1,\o)$ on larger and larger time intervals. Indeed, up to replacing $\Si$ by $\Si\cap\wt\Si$, which is still of full probability, we can use Proposition \ref{prop-approx} along with the definition of $\Si$, to get that for any $(u_0,u_1,\o)\in\Si$, there exists $m\in\N$ and $N_p\to\infty$ as $p\to\infty$ such that $(u_0,u_1,\o)\in \bigcap_{p,j\in\N}\Si_{N_p}^{m,j}$. In view of the previous proposition, it follows from \eqref{global-bd1} that for any $t\ge 0$
\begin{align*}
\|\Phi(t)(u_0,u_1,\o)\|_{\H^{s}}&=  \lim_{p\to\infty} \big\|\Phi^{N_p}(t)(u_0,u_1,\o)\big\|_{\H^{s}} \leq CD\big(m+\log(1+t)\big)^{\frac{k}2}.
\end{align*}
In particular $\Phi(t)(u_0,u_1,\o)$ is globally defined for any $(u_0,u_1,\o)\in\Si$. 

The invariance of $\rho_{k+1}$ then follows directly from the invariance of $\rho_{N_p,k+1}$ under \eqref{FNLW} along with the convergence of $\Phi^{N_p}(t)(u_0,u_1,\o)$ towards $\Phi(t)(u_0,u_1,\o)$ given by the previous proposition and the convergence of $\rho_{N_p,k+1}$ towards $\rho_{k+1}$ given by Lemma \ref{lem-Gibbs}. Indeed, as in \cite{Tolomeo}, for any initial data $(u_0,u_1)$ with law $\rho$, any test function $F\in C_b(\H^s;\R)$ and any $t\ge 0$, we have by Lemma \ref{lem-Gibbs}, Proposition \ref{prop-approx} and the dominated convergence theorem
\begin{align*}
&\int_{\H^s(\M)}\int_\O  F\big[\Phi(t)(u_0,u_1,\o)\big]d\Prob(\o)d\rho_{k+1}(u_0,u_1)\\
&=\ZZ^{-1}\int_{\H^s(\M)}\int_\O F\big[\Phi(t)(u_0,u_1,\o))\big]e^{-G_{k+1}(u_0)}d\Prob(\o)d\mu(u_0,u_1)\\
&= \lim_{p\to\infty}\ZZ_{N_p}^{-1}\int_{\H^s}\int_\O F\big[\Phi^{N_p}(t)(u_0,u_1,\o)\big]e^{-G_{N_p,k+1}(u_0)}d\Prob(\o)d\mu(u_0,u_1)\intertext{where $\ZZ=\int_{\H^s(\M)}e^{-G_{k+1}(u_0)}d\mu(u_0,u_1)$ and $\ZZ_{N_p}=\int_{\H^s(\M)}e^{-G_{N_p,k+1}(u_0)}d\mu(u_0,u_1)$. Now we can use the invariance of $\rho_{N_p,k+1}$ under $\Phi^{N_p}(t)$ given by Proposition \ref{prop-FNLW}, and we can continue with}
&=\lim_{p\to\infty}\int_{\H^s}F(u_0,u_1)d\rho_{N_p,k+1}\\
&=\int_{\H^s} F(u_0,u_1)d\rho_{k+1}.
\end{align*}
This shows the invariance of $\rho_{k+1}$. Hence the proof of Theorem \ref{thm-invariance} will be completed once we prove the proposition.
\begin{proof}[Proof of Proposition~\ref{prop-approx}]
Let us fix $m,j\in\N$, $\dl>0$ as in \eqref{dlmj}, and $(u_0,u_1,\o)\in\bigcap_{p\in\N}\Si_{N_p}^{m,j}$. In the following, we fix two (large) integers $N,M\in\{N_p\}_{p\in\N}$. Again, we write
\begin{align*}
(w_N,\dt w_N)(t) = \Phi^N(t)(u_0,u_1,\o)-\Psi(t,u_0,u_1,\o),
\end{align*}
and we denote by $\Phi^N_1(t)$  (respectively $\Phi^N_2(t)$) the first (respectively second) component of $\Phi^N(t)$. We will control inductively the difference $(w_N(t),\dt w_N(t))-(w_M(t),\dt w_M(t))$ on the time intervals $[a\dl,(a+1)\dl]$, $a=1,...,[2^j/\dl]$. We begin by controlling the difference on the first time interval, corresponding to $a=0$. 

Then on $[0,\dl]$, we have for $N\le M$:
\begin{align}
&w_{M}(t)-w_N(t)\notag\\
&= \sum_{\ell=0}^k{\binom k \ell}\int_0^tV(t-t')\bigg\{\P_{M}\Big[H_{\ell}\big(\P_{M}\Psi(t',u_0,u_1,\o);\s_{M}(x)\big)\big(\P_{M}w_{M}(t')\big)^{k-\ell}\Big]\notag\\
&\qquad-\P_{N}\Big[H_{\ell}\big(\P_{N}\Psi(t',u_0,u_1,\o);\s_N(x)\big)\big(\P_Nw_N(t')\big)^{k-\ell}\Big]\bigg\}dt'\notag\\
&=\sum_{\ell=0}^{k}{\binom k \ell}\int_{0}^tV(t-t')\bigg\{(\P_{M}-\P_N)\Big[H_{\ell}\Big(\P_{M}\Psi(t',u_0,u_1,\o);\s_{M}(x)\Big)\big(\P_{M}w_{M}(t')\big)^{k-\ell}\Big]\notag\\
&\qquad+\P_{N}\Big[\Big(H_{\ell}\big(\P_{M}\Psi(t',u_0,u_1,\o);\s_{M}(x)\big)\notag\\
&\qquad-H_{\ell}\big(\P_{N}\Psi(t',u_0,u_1,\o);\s_N(x)\big)\Big)\big(\P_{M}w_{M}(t')\big)^{k-\ell}\Big]\notag\\
&\qquad+\P_N\Big[H_{\ell}\big(\P_{N}\Psi(t',u_0,u_1,\o);\s_N(x)\big)\Big(\big(\P_{M}w_{M}(t')\big)^{k-\ell}-\big(\P_Nw_N(t')\big)^{k-\ell}\Big)\Big]\bigg\}dt'\notag\\
&=:\I+\II+\III.
\label{I-II-III}
\end{align}

To estimate these terms, we proceed as in the proof of Proposition~\ref{prop-glocal}. We begin with
\begin{align*}
&\|\I\|_{C([0,\dl];H^{s_2})}\\&\les \sum_{\ell=0}^k\Big\|(\P_{M}-\P_N)\Big[H_{\ell}\Big(\P_{M}\Psi(u_0,u_1,\o);\s_{M}(x)\Big)\big(\P_{M}w_{M}\big)^{k-\ell}\Big]\Big\|_{L^1([0,\dl];H^{s_2-1})}\\
&\les \sum_{\ell=0}^k N^{s_2-s_1}\Big\|\Big[H_{\ell}\Big(\P_{M}\Psi(u_0,u_1,\o);\s_{M}(x)\Big)\big(\P_{M}w_{M}\big)^{k-\ell}\Big]\Big\|_{L^1([0,\dl];H^{s_1-1})}\\
&\les \sum_{\ell=0}^k N^{s_2-s_1}\dl^{\frac12}\big\|H_{\ell}\Big(\P_{M}\Psi(u_0,u_1,\o);\s_{M}(x)\Big)\big\|_{L^2([0,1];W^{s,\infty})}\|\P_{M}w_{M}\|_{C([0,\dl];H^{s_1})}^{k-\ell}\\
&\le C(m,j) N^{s_2-s_1},
\end{align*}
for some constant $C(m,j)$ independent of $N,M$, where the second to last estimate comes from the same argument as in the proof of Proposition \ref{prop-glocal}, and the last one from the condition \eqref{condition2} given by $(u_0,u_1,\o)\in \B_{M}^{m,j}(D)$, with \eqref{bd-wN} and the choice of $\dl$ in \eqref{dlmj}.

Similarly, we bound
\begin{align*}
&\|\mathrm{II}\|_{C([0,\dl];H^{s_2})}\\
&\les\sum_{\ell=1}^{k}\Big\|\Big[H_{\ell}\big(\P_{M}\Psi(u_0,u_1,\o);\s_{M}(x)\big)\\
&\qquad-H_{\ell}\big(\P_{N}\Psi(u_0,u_1,\o);\s_N(x)\big)\Big]\big(\P_{M}w_{M}\big)^{k-\ell}\Big\|_{L^1([0,\dl];H^{s_2-1})}\\
&\les \dl^{\frac12}\bigg\{\Big\|H_{k}\big(\P_{M}\Psi(u_0,u_1,\o);\s_{M}(x)\big)-H_{k}\big(\P_{N}\Psi(u_0,u_1,\o);\s_N(x)\big)\Big\|_{L^2([0,\dl];H^{s_2-1})}\\
&\qquad+\sum_{\ell=1}^{k-1}\Big\|H_{\ell}\big(\P_{M}\Psi(u_0,u_1,\o);\s_{M}(x)\big)-H_{\ell}\big(\P_{N}\Psi(u_0,u_1,\o);\s_N(x)\big)\Big\|_{L^2([0,\dl];W^{s,\infty})}\\
&\qquad\qquad\times\big\|\P_{M}w_{M}\big\|_{C([0,\dl];H^{s_1})}^{k-\ell}\bigg\}\\
& \le C(m,j) N^{-\eps},
\end{align*} 
where the last bound comes from \eqref{condition3} given by $(u_0,u_1,\o)\in\B_{M}^{m,j}(D)$ and from \eqref{bd-wN} with the choice of $\dl$. 

Finally, we can further decompose
\begin{align*}
\III&=\sum_{\ell=0}^{k-1}{\binom k \ell}\int_0^tV(t-t')\P_N\bigg\{H_{\ell}\big(\P_{N}\Psi(t',u_0,u_1,\o);\s_N(x)\big)\Big[\Big(w_N(t')^{k-\ell}-\big(\P_Nw_N(t')\big)^{k-\ell}\Big)\\
&\qquad\qquad+\Big(\big(\P_{M}w_{M}(t')\big)^{k-\ell}-w_{M}(t')^{k-\ell}\Big)+\Big(w_{M}(t')^{k-\ell}-w_N(t')^{k-\ell}\Big)\Big]\bigg\}dt'\\
&=:\III_1+\III_2+\III_3.
\end{align*}

We estimate similarly as before
\begin{align*}
\|\III_1\|_{C([0,\dl];H^{s_2})}&\les \sum_{\ell=0}^{k-1}\dl^{\frac12}\big\|H_{\ell}\big(\P_{N}\Psi(u_0,u_1,\o);\s_N(x)\big)\big\|_{L^2([0,1];W^{s,\infty})}\\
&\qquad\times\big\|(1-\P_{N})w_{N}\big\|_{C([0,\dl];H^{s_2})}\Big(\big\|\P_{N}w_{N}\big\|_{C([0,\dl];H^{s_2})}^{k-\ell-1}+\big\|w_{N}\big\|_{C([0,\dl];H^{s_2})}^{k-\ell-1}\Big)\\
&\le C(m,j) N^{s_2-s_1},
\end{align*}
where the first estimate follows from the same argument as in the proof of Proposition~\ref{prop-glocal}, provided that $s_2<s_1<1+s<1$ is close enough to 1. The same argument applies to $\III_2$ and gives the same bound (with $M$ in place of $N$), and the last term can be bounded similarly by
\begin{align*}
\big\|\III_3\big\|_{C([0,\dl];H^{s_2})}&\les\sum_{\ell=0}^{k-1}\dl^{\frac12}\big\|H_{\ell}\big(\P_{N}\Psi(u_0,u_1,\o);\s_N(x)\big)\big\|_{L^2([0,1];W^{s,\infty})}\\
&\qquad\times\big\|w_{M}-w_N\big\|_{C([0,\dl];H^{s_2})}\Big(\big\|w_{M}\big\|_{C([0,\dl];H^{s_2})}^{k-\ell-1}+\big\|w_{N}\big\|_{C([0,\dl];H^{s_2})}^{k-\ell-1}\Big)\\
&\le C\dl^{\frac12}D(m+j)^{\frac{k}2}\|w_{M}-w_N\|_{C([0,\dl];H^{s_2})},
\end{align*}
where the last estimate comes again from \eqref{condition2} thanks to $(u_0,u_1,\o)\in\B_{M}^{m,j}(D)$, and from \eqref{bd-wN} with the argument of Proposition~\ref{prop-glocal} applied with $s_2$ (provided that $s_2$ is sufficiently close to 1). With our choice of $C\dl^{\frac12}D(m+j)^{\frac{k}2}=C\dl^{\frac12}R\le \frac12$, we can absorb this last term in the left-hand side of \eqref{I-II-III}.

The same arguments also apply to control $\Phi^N_2(t)-\Phi^M_2(t)$ on $[0,\dl]$. Therefore, gathering the estimates above leads to
\begin{align}
\big\|\Phi^{M}(t)(u_0,u_1,\o)-\Phi^N(t)(u_0,u_1,\o)\big\|_{C([0,\dl];\H^{s_2})} \le C(m,j)\big(N^{s_2-s_1}+N^{-\eps}\big)
\label{S1}
\end{align}
for any $N\le M$. In particular this shows the convergence on the time interval $[0,\dl]$.

We now investigate the convergence on the second time interval $[\dl,2\dl]$: we first decompose
\begin{align*}
&\big\|\Phi^{M}(t+\dl)(u_0,u_1,\o)-\Phi^N(t+\dl)(u_0,u_1,\o)\big\|_{C([0,\dl];\H^{s_2})}\\& \le \big\|\Phi^{M}(t)\wt\Phi^{M}(\dl)(u_0,u_1,\o)-\Phi^N(t)\wt\Phi^{M}(\dl)(u_0,u_1,\o)\big\|_{C([0,\dl];\H^{s_2})}\\
&\qquad+\big\|\Phi^{N}(t)\wt\Phi^{M}(\dl)(u_0,u_1,\o)-\Phi^N(t)\wt\Phi^{N}(\dl)(u_0,u_1,\o)\big\|_{C([0,\dl];\H^{s_2})}.
\end{align*}

Note that replacing $\Psi(u_0,u_1,\o)$ by $\Psi\big(\wt\Phi^{M}(\dl)(u_0,u_1,\o)\big)$ in the previous estimates and using that we still have $\wt\Phi^{M}(\dl)(u_0,u_1,\o)\in\B_{M}^{m,j}$ by choice of $(u_0,u_1,\o)\in\Si_{M}^{m,j}$ shows that the first term above is still bounded by
\begin{align*}
&\big\|\Phi^{M}(t)\wt\Phi^{M}(\dl)(u_0,u_1,\o)-\Phi^N(t)\wt\Phi^{M}(\dl)(u_0,u_1,\o)\big\|_{C([0,\dl];\H^{s_2})}\\
&\qquad\le C(m,j)\big(N^{s_2-s_1}+N^{-\eps}\big)
\end{align*}
for any $N\le M$.

Thus we need to deal with the second term. We can redefine
\begin{align*}
w_{M}(t) = \Phi_1^{N}(t)\wt\Phi^{M}(\dl)(u_0,u_1,\o) - \Psi\big(t,\wt\Phi^{M}(\dl)(u_0,u_1,\o)\big)
\end{align*}
and
\begin{align*}
w_{N}(t) = \Phi_1^{N}(t)\wt\Phi^{N}(\dl)(u_0,u_1,\o) - \Psi\big(t,\wt\Phi^{N}(\dl)(u_0,u_1,\o)\big),
\end{align*}
and since $N,M\in\{N_p\}_{p\in\N}$ and $(u_0,u_1,\o)\in \bigcap_{p\in\N}\Si_{N_p}^{m,j}$ we have in particular that both $\wt\Phi^{M}(\dl)(u_0,u_1,\o)\in\B_M^{m,j}(D)$ and $\wt\Phi^{N}(\dl)(u_0,u_1,\o)\in\B_N^{m,j}(D)$ so that both $w_M$ and $w_N$ are well-defined and enjoy the bound \eqref{bd-wN}. Moreover, by definition of $\Phi_1^N(t)$ and $\Psi$, they satisfy the following Duhamel formula:
\begin{align*}
w_M = \sum_{\ell=0}^k{\binom k \ell}\int_0^tV(t-t')\P_N\bigg\{H_{\ell}\Big(\P_N\Psi\big(t',\wt\Phi^M(\dl)(u_0,u_1,\o)\big);\s_N(x)\Big)\big(\P_Nw_M(t')\big)^{k-\ell}\bigg\}dt',
\end{align*}
and
\begin{align*}
w_N = \sum_{\ell=0}^k{\binom k \ell}\int_0^tV(t-t')\P_N\bigg\{H_{\ell}\Big(\P_N\Psi\big(t',\wt\Phi^N(\dl)(u_0,u_1,\o)\big);\s_N(x)\Big)\big(\P_Nw_N(t')\big)^{k-\ell}\bigg\}dt'.
\end{align*}

To estimate in $C([0,\dl];H^{s_2}(\M))$ the difference
\begin{align*}
&\Phi_1^{N}(t)\wt\Phi^{M}(\dl)(u_0,u_1,\o)-\Phi_1^N(t)\wt\Phi^{N}(\dl)(u_0,u_1,\o)\\
&=\Psi\big(t,\wt\Phi^{M}(\dl)(u_0,u_1,\o)\big)-\Psi\big(t,\wt\Phi^{N}(\dl)(u_0,u_1,\o)\big)+w_M-w_N,
\end{align*}
we first bound directly the linear terms by
\begin{align*}
&\big\|\Psi\big(t,\wt\Phi^{M}(\dl)(u_0,u_1,\o)\big)-\Psi\big(t,\wt\Phi^{N}(\dl)(u_0,u_1,\o)\big)\big\|_{C([0,\dl];H^{s_2})}\\&\le \big\|\Phi^{M}(\dl)(u_0,u_1,\o)\big)-\Phi^{N}(\dl)(u_0,u_1,\o)\big\|_{\H^{s_2}}\\
& \le C(m,j)\big(N^{s_2-s_1}+N^{-\eps}\big)
\end{align*}
thanks to \eqref{def-Psidamp} and \eqref{S1}.

To estimate the difference of the nonlinear components, we decompose
\begin{align*}
&w_M-w_N\\
 &= \sum_{\ell=0}^k{\binom k \ell}\int_0^tV(t-t')\P_N\bigg\{H_{\ell}\Big(\P_N\Psi\big(t',\wt\Phi^M(\dl)(u_0,u_1,\o)\big);\s_N(x)\Big)\big(\P_Nw_M(t')\big)^{k-\ell}\\
&\qquad\qquad-H_{\ell}\Big(\P_N\Psi\big(t',\wt\Phi^N(\dl)(u_0,u_1,\o)\big);\s_N(x)\Big)\big(\P_Nw_N(t')\big)^{k-\ell}\bigg\}dt'\\
&=\sum_{\ell=0}^k{\binom k \ell}\int_0^tV(t-t')\P_N\bigg\{\Big[H_{\ell}\Big(\P_N\Psi\big(t',\wt\Phi^M(\dl)(u_0,u_1,\o)\big);\s_N(x)\Big)\\
&\qquad\qquad\qquad-H_{\ell}\Big(\P_N\Psi\big(t',\wt\Phi^N(\dl)(u_0,u_1,\o)\big);\s_N(x)\Big)\Big]\big(\P_Nw_M(t')\big)^{k-\ell}\\
&\qquad+H_{\ell}\Big(\P_N\Psi\big(t',\wt\Phi^N(\dl)(u_0,u_1,\o)\big);\s_N(x)\Big)\Big[\big(\P_Nw_M(t')\big)^{k-\ell}-\big(\P_Nw_N(t')\big)^{k-\ell}\Big]\bigg\}dt'\\
&=:\wt \II + \wt\III.
\end{align*}

First note that we can estimate $\wt \III$ exactly as $\III$ in \eqref{I-II-III}, giving the bound
\begin{align*}
\big\|\wt\III\big\|_{C([0,\dl];H^{s_2})}&\le C(m,j)N^{s_2-s_1} + \frac12\big\|w_M-w_N\big\|_{C([0,\dl];H^{s_2})}.
\end{align*}

Finally, we estimate the remaining term by 
\begin{align*}
&\big\|\wt\II\big\|_{C([0,\dl];H^{s_2})}\\&\les \sum_{\ell=1}^k\Big\|\Big[H_{\ell}\Big(\P_N\Psi\big(\wt\Phi^M(\dl)(u_0,u_1,\o)\big);\s_N(x)\Big)\\
&\qquad\qquad-H_{\ell}\Big(\P_N\Psi\big(\wt\Phi^N(\dl)(u_0,u_1,\o)\big);\s_N(x)\Big)\Big]\big(\P_Nw_M\big)^{k-\ell}\Big\|_{L^1([0,\dl];H^{s_2-1})}
\end{align*}
Writing then 
\begin{align*}
H_{\ell}(u;\s_N)-H_{\ell}(v;\s_N) = -\sum_{i=0}^{\ell-1}{\binom \ell i}H_i(u;\s_N)(v-u)^{\ell-i}
\end{align*}
thanks to \eqref{Hsum}, we can then estimate the previous term with
\begin{align*}
&\sum_{\ell=1}^k\sum_{i=0}^{\ell-1}\dl^{\frac12}\Big\|\Big[\P_N\Psi\big(\wt\Phi^M(\dl)(u_0,u_1,\o)\big)-\P_N\Psi\big(\wt\Phi^N(\dl)(u_0,u_1,\o)\big)\Big]^{\ell-i}\\
&\qquad\times H_{i}\Big(\P_N\Psi\big(\wt\Phi^N(\dl)(u_0,u_1,\o)\big);\s_N(x)\Big)\big(\P_Nw_M\big)^{k-\ell}\Big\|_{L^2([0,\dl];H^{s_2-1})}\\
&\les \sum_{\ell=1}^k\sum_{i=0}^{\ell-1}\dl^{\frac12}\Big\|H_i\Big(\P_N\Psi\big(\wt\Phi^N(\dl)(u_0,u_1,\o)\big);\s_N(x)\Big)\Big\|_{L^2([0,1];W^{s,\infty})}\\
&\qquad\times\Big\|\Psi\big(\wt\Phi^M(\dl)(u_0,u_1,\o)\big)-\Psi\big(\wt\Phi^N(\dl)(u_0,u_1,\o)\big)\Big\|_{C([0,\dl];H^{s_2})}^{\ell-i}\big\|w_M\big\|_{C([0,\dl];H^{s_2})}^{k-\ell},
\end{align*}
provided again that $s_2$ is close enough to 1 (depending on $k$). Using then that $\wt\Phi^N(\dl)(u_0,u_1,\o)\in\B_N^{m,j}(D)$ and \eqref{bd-wN},\eqref{S1}, we finally get
\begin{align*}
\big\|\wt\II\big\|_{C([0,\dl];H^{s_2})}\le C(m,j)\big(N^{s_2-s_1}+N^{-\eps}\big).
\end{align*}

Gathering the estimates above, we obtain 
\begin{align*}
\big\|w_M-w_N\big\|_{C([0,\dl];H^{s_2})}\le 4C(m,j)(N^{s_2-s_1}+N^{-\eps})
\end{align*}
which leads to
\begin{align*}
\big\|(w_N,\dt w_N)(t+\dl)-(w_M,\dt w_M)(t+\dl)\big\|_{C([0,\dl];\H^{s_2})} \le C_2(m,j)\big(N^{s_2-s_1}+N^{-\eps}\big)
\end{align*}
for some larger constant $C_2(m,j)\ge C(m,j)$. This shows that $\big\{(w_{N_p},\dt w_{N_p})\big\}_{p\in\N}$ is also a Cauchy sequence in $C([0,2\dl];\H^{s_2}(\M))$.

We can then proceed inductively on $a=0,...,[2^j/\dl]$ and repeat the previous estimates by using that at each step $\wt\Phi^N(a\dl)(u_0,u_1,\o)\in\B_N^{m,j}(D)$ since $(u_0,u_1,\o)\in \Si_N^{m,j}$. Thus we deduce that there exists a (large) constant $C_{2^j/\dl}(m,j)>0$ such that for any $N,M \in \{N_p\}_{p\in\N}$ with $N\le M$ it holds
\begin{align}
\big\|(w_N, \dt w_N)-(w_M,\dt w_M)\big\|_{C([0,2^j];\H^{s_2})} \le C_{2^j/\dl}(m,j)\big(N^{s_2-s_1}+N^{-\eps}\big).
\end{align}
This is enough to show the convergence of $\{\Phi^{N_p}(t)(u_0,u_1,\o)-\Psi(t)(u_0,u_1,\o)\}$ in $C([0,2^j];\H^{s_2}(\M))$. As a result, $\big\{\Phi^{N_p}(t)(u_0,u_1,\o)\big\}_{p\in\N}$ converges in $C([0,2^j];\H^s(\M))$. This concludes the proof of Proposition \ref{prop-approx}.
\end{proof}
\begin{remark}\rm
By slightly modifying the proof of Proposition~\ref{prop-approx}, we can indeed show that for $(u_0,u_1,\o)\in\Sigma$, the entire sequence $\big\{\Phi^N(t)(u_0,u_1,\o)\big\}_{N\in\N}$ converges in $C([0,2^j];\H^s(\M))$ for any $j\in\N$. See for example Corollary 9.11 in \cite{OOT}.
\end{remark}

\begin{ackno}\rm 
The authors are thankful to Younes Zine for careful proofreading. They also wish to thank the anonymous referees for their careful proofreading and helpful remarks which improved the quality of the paper.

T.O.~was supported by the European Research Council (grant no.~637995 ``ProbDynDispEq''
and grant no.~864138 ``SingStochDispDyn"). T.R.~was supported by the European Research Council (grant no.~637995 ``ProbDynDispEq'') and the German Research Foundation (DFG) through the CRC 1283. N.T.~was supported by the ANR grant ODA (ANR-18-CE40-0020-01).
\end{ackno}

\end{document}